\documentclass[11pt]{amsart}
\usepackage{fullpage}
\usepackage{amssymb}
\usepackage{fullpage}
\usepackage{graphicx}
\usepackage{subfigure}
\usepackage{psfrag}
\usepackage{color}
\newtheorem{theorem}{Theorem}[section]

\newtheorem{proposition}[theorem]{Proposition}

\newtheorem{remark}[theorem]{Remark}
\numberwithin{equation}{section}
\title{On dual Schur domain decomposition method for linear first-order transient problems}
\author{K.~B.~Nakshatrala} 
\address{Correspondence to: Dr.~Kalyana Babu Nakshatrala, Department of 
Mechanical Engineering, 216 Engineering/Physics Building, Texas A\&M 
University, College Station, Texas-77843, USA. TEL:+1-979-845-1292}
\email{knakshatrala@tamu.edu} 
\author{A.~Prakash} 
\address{Dr. Arun Prakash, School of Civil Engineering, 550 Stadium Mall 
Drive, Room \#4119, Purdue University, West Lafayette, Indiana-47907, USA. 
TEL:+1-765-494-6696} \email{arunprakash@purdue.edu} 
\author{K.~D.~Hjelmstad} 
\address{Dr. Keith D. Hjelmstad, University Vice President \& Dean, Arizona 
State University, Mesa, Arizona, USA.}  
\email{keith.hjelmstad@asu.edu}
\begin{document}


\begin{abstract}
  This paper addresses some numerical and theoretical aspects of dual Schur domain decomposition 
  methods for linear first-order transient partial differential equations. The spatially discrete 
  system of equations resulting from a dual Schur domain decomposition method can be expressed as 
  a system of differential algebraic equations (DAEs). In this work, we consider the trapezoidal 
  family of schemes for integrating the ordinary differential equations (ODEs) for each subdomain 
  and present four different coupling methods, corresponding to different algebraic constraints, 
  for enforcing kinematic continuity on the interface between the subdomains. Unlike the continuous 
  formulation, the discretized formulation of the transient problem is unable to enforce simultaneously 
  the continuity of both the primary variable and its rate along the subdomain interface (except for 
  the backward Euler method). 
  
  Method $1$ ($\boldsymbol{d}$-continuity) is based on the conventional approach using continuity 
  of the primary variable and we show that this method is unstable for a lot of commonly used time 
  integrators including the mid-point rule. To alleviate this difficulty, we propose a new Method 
  $2$ (Modified $\boldsymbol{d}$-continuity) and prove its stability for coupling all time integrators 
  in the trapezoidal family (except the forward Euler). Method $3$ ($\boldsymbol{v}$-continuity) is 
  based on enforcing the continuity of the time derivative of the primary variable. However, this 
  constraint introduces a drift in the primary variable on the interface. We present Method $4$ 
  (Baumgarte stabilized) which uses Baumgarte stabilization to limit this drift and we derive bounds 
  for the stabilization parameter to ensure stability. 
  Our stability analysis is based on the ``energy'' method, and one of the main contributions 
  of this paper is the extension of the energy method (which was previously introduced in the 
  context of numerical methods for ODEs) to assess the stability of numerical formulations for 
  index-$2$ differential-algebraic equations (DAEs). Finally, we present numerical examples to 
  corroborate our theoretical predictions. 
\end{abstract}
\maketitle


\section{INTRODUCTION}
Solving partial differential equations using domain decomposition techniques has been an active area of research 
for quite some time \cite{Quarteroni_Valli_DD,Smith_DD,Toselli_DD}. Much of this research has addressed static 
problems \cite{Farhat_Roux_IJNME_1991_v32_p1205,Fragakis_Papadrakakis_CMAME_2003_v192_p3799,Park_CompMech_2000_v24_p476}, 
designing preconditioners \cite{Rixen_Farhat_IJNME_1999_v44_p489,Klawonn_Widlund_CAPM_2001_v54_p57}, parallel implementation 
\cite{Farhat_Roux_IJNME_1991_v32_p1205,Farhat_Chen_Mandel_IJNME_1995_v38_p3831}, efficient computer 
implementation \cite{Farhat_Roux_IJNME_1991_v32_p1205,Gravouil_Combescure_IJNME_2001_v50_p199,
  Prakash_Hjelmstad_IJNME_2004_v61_p2183}, solvers \cite{Rixen_Farhat_Tezaur_Mandel_IJNME_1999_v46_p501}, 
etc. On the other hand, domain decomposition techniques for transient problems are not as mature as 
those associated with static problems. Some of the representative papers on domain decomposition methods 
for transient problems are \cite{Park_JAM_1980_v47_p370,Farhat_Crivelli_Roux_IJNME_1994_v37_p1945,Farhat_Chen_Mandel_IJNME_1995_v38_p3831,
  Farhat_Crivelli_Geradin_CMAME_1995_v125_p71,Gravouil_Combescure_IJNME_2001_v50_p199,
  Combescure_Gravouil_CMAME_2002_v191_p1129,Prakash_Hjelmstad_IJNME_2004_v61_p2183}. 
A study of the choice of the interface boundary conditions has been done for second-order transient problems 
(i.e., structural dynamics) using the Newmark family of time integrators. For example, see Cardona and Geradin 
\cite{Cardona_Geradin_ComputStruc_1989_v33_p801}, Farhat \textit{et. al.} \cite{Farhat_Crivelli_Geradin_CMAME_1995_v125_p71}, 
and Combescure and Gravouil \cite{Combescure_Gravouil_CMAME_2002_v191_p1129}. All these papers deal with structural dynamics, 
which is a second order transient systems. 

In Reference \cite{Nakshatrala_Hjelmstad_Tortorelli_IJNME} a coupling method has been 
proposed for first-order transient systems that can accommodate different time integrators 
and time steps in different subdomains. This was made possible by making the calculation 
of the interface Lagrange multiplier explicit. In this paper we restrict our study to time 
stepping schemes from the generalized trapezoidal family, and assume same time integrator 
and time step in all subdomains. This will allow the calculation of the interface Lagrange 
multiplier to be implicit, and thereby increasing the overall numerical stability. 

The main objective of this article is to present four variants of dual Schur domain decomposition 
method for linear first-order transient systems, and assess their stability. We restrict our study 
to conforming computational meshes. The stability analysis is based on the ``energy'' method 
\cite{Richtmyer_Morton,Hughes}. The stability results for first-order systems (which are index-2 
DAEs) presented in this paper are \emph{not} a direct extension of results for ODEs or second-order 
systems (e.g., structural dynamics), and to the best of our knowledge, have not been reported in the 
literature. 

\begin{remark}
  A thorough discussion of other computational aspects like preconditioners, parallel 
  implementation issues (like scalability and speedup), direct and iterative solvers 
  in the context of domain decomposition method is beyond the scope of this paper. These 
  computational aspects have been addressed adequately in the literature and we have 
  cited a few representative references to this end. These techniques can be generally 
  employed without difficulty for the proposed domain decomposition methods. 
\end{remark}

\subsection{Main contributions} Some of the main contributions of 
this paper are 
\begin{itemize}
\item extension of the energy method for assessing stability to linear index-2 
  differential-algebraic equations,
\item stability analysis of the $\boldsymbol{d}$-continuity method, and demonstrated the 
  instability of the $\boldsymbol{d}$-continuity method under many common time integrators 
  from the generalized trapezoidal family,
\item modified $\boldsymbol{d}$-continuity method and its stability proof, 
\item derived an upper bound for the user-defined parameter in the Baumgarte stabilization 
  to ensure numerical stability, and also estimated the critical time step for the method, 
  and
\item verified numerically that the optimal spatial convergence rate is unaffected by these domain 
  decomposition methods.
\end{itemize}

\section{TIME CONTINUOUS GOVERNING EQUATIONS}
We will consider linear transient heat conduction as our model problem to illustrate 
various coupling methods. Consider a continuous domain $\Omega$ with prescribed 
temperature on $\partial_1 \Omega$ and prescribed fluxes on $\partial_2 \Omega$ 
(and for well-posedness we have $\partial_1 \Omega \cap \partial_2 \Omega = \emptyset$ 
and  $\partial_1 \Omega \cup \partial_2 \Omega = \partial \Omega$). The semi-discrete 
finite element equations for linear transient heat conduction can be written as 
(for example, see Hughes \cite{Hughes})
\begin{align}
  \label{Eqn:Trapezoidal_semi_discrete}
  \boldsymbol{M} \dot{\boldsymbol{u}} + \boldsymbol{K}\boldsymbol{u} = \boldsymbol{f}, \quad \forall t \in [0,T] 
\end{align}
where $\boldsymbol{u}(t)$ is the nodal temperature vector, $t$ denotes the time, superposed 
dot denotes the time derivative, $\boldsymbol{M}$ is the capacity matrix, $\boldsymbol{K}$ 
is the conductivity matrix, and $\boldsymbol{f}(t)$ is the (prescribed) external heat source 
vector. The capacity matrix $\boldsymbol{M}$ is assumed to be symmetric and positive definite, 
and the conductivity matrix $\boldsymbol{K}$ is assumed to be symmetric and positive semidefinite. 
The following initial conditions and constraints complete the above differential system:
\begin{align}
  \label{Eqn:Trapezoidal_initial_constraints}
  \boldsymbol{u}(t = 0) = \boldsymbol{u}^{(0)}, \quad \boldsymbol{u}|_{\partial_1 \Omega} = 
  \boldsymbol{u}_{\mathrm{p}}(t)
\end{align}
where $\boldsymbol{u}^{(0)}$ is the prescribed nodal initial temperature, and 
$\boldsymbol{u}_{\mathrm{p}}$ is the prescribed nodal temperature on the boundary 
$\partial_1 \Omega$. 

\subsection{Decomposed problem}
We decompose the computational domain into $S$ subdomains following a dual Schur formulation. 
This approach automatically implies that the equilibrium of the interface fluxes is enforced 
through the Lagrange multipliers. We assume that the kinematic constraints are linearly 
independent. We assume that the kinematic constraints are linearly independent and are 
compactly expressed using signed Boolean matrices. A signed Boolean matrix has entries 
either $-1$, $0$ or $+1$ such that each row has at most one non-zero entry, see Reference 
\cite{Nakshatrala_Hjelmstad_Tortorelli_IJNME}. 
Note that by the construction of the signed Boolean matrix $\boldsymbol{C}_i$, its $1$-norm is 
$\|\boldsymbol{C}_i\|_{1} \leq 1$, which will be used in Section \ref{Sec:DD_S4_Stability_Analysis}. 
Also note that if $\boldsymbol{C}_i \neq \boldsymbol{0}$ (which is the case in this paper) 
$\|\boldsymbol{C}_i\|_{1} = 1$.

\begin{remark}
  Using signed Boolean matrices as defined above one can handle cross points (which 
  are the points at which more than two subdomains intersect), and should be able 
  to construct linearly independent kinematic constraints. For further details see 
  Reference \cite{Nakshatrala_Hjelmstad_Tortorelli_IJNME}.
\end{remark}

The decomposition into subdomains of the differential system given by equation 
\eqref{Eqn:Trapezoidal_semi_discrete} can be written as follows: $\forall t \in [0,T]$ 
we have 
\begin{align}
  \label{Eqn:Trapezoidal_domain_equation}
  & \boldsymbol{M}_{i} \dot{\boldsymbol{u}}_{i} + \boldsymbol{K}_{i} \boldsymbol{u}_{i} = 
  \boldsymbol{f}_{i} + {\boldsymbol{C}_{i}}^{\mathrm{T}} 
  \boldsymbol{\lambda} \quad \forall \; i = 1, \cdots, S \\
  \label{Eqn:Trapezoidal_kinematic_continuity}
  & \sum_{i = 1}^{S} \boldsymbol{C}_{i} \boldsymbol{u}_i = \boldsymbol{0}
\end{align}
where $\boldsymbol{\lambda}$ represents the vector of Lagrange multipliers, and 
$\boldsymbol{C}_i$ the signed Boolean connectivity matrices for perfect connection 
between compatible meshes. Note that the time continuous formulation is capable of 
enforcing the continuity of all the kinematic quantities and their time derivatives 
along the subdomain interface. However, this is not true of time discretized equations 
as we will see in the next section. 

\begin{remark} 
For a derivation of equations 
\eqref{Eqn:Trapezoidal_domain_equation}-\eqref{Eqn:Trapezoidal_kinematic_continuity} 
see Reference \cite[Appendix A]{Nakshatrala_Hjelmstad_Tortorelli_IJNME}. 
\end{remark}

It is important to note that the undecomposed problem (given by equation \eqref{Eqn:Trapezoidal_semi_discrete}) 
is a system of ODEs, whereas the decomposed problem consists of a system of ODEs given by equation 
\eqref{Eqn:Trapezoidal_domain_equation}, and a system of algebraic equations given by equation 
\eqref{Eqn:Trapezoidal_kinematic_continuity}. Such a system of equations (that consists of a 
system of ODEs and a system of algebraic equations) belong to a broader class of equations 
called \emph{differential-algebraic equations} (DAEs). For completeness, we present a very 
brief discussion of DAEs. 

An important special class is the semi-explicit DAE (also referred to an ODE with 
algebraic constraints), which can be mathematically written as 
\begin{align}
  \dot{\boldsymbol{x}} = \boldsymbol{h}_1(t,\boldsymbol{x},\boldsymbol{y}), \quad 
  \boldsymbol{0} = \boldsymbol{h}_2(t,\boldsymbol{x},\boldsymbol{y})
\end{align}
It is easy to see that the governing equations of the decomposed problem, given by equations 
\eqref{Eqn:Trapezoidal_domain_equation}-\eqref{Eqn:Trapezoidal_kinematic_continuity}, form 
a semi-explicit DAE.  

The index of a DAE is a measure to assess the difficulty to solve it numerically. Some of 
the common indices are -- the differential index, local index, geometric index, tractability 
index, and perturbation index \cite{Ascher_Petzold}. A popular index that is easy to compute 
is the differential index. For a semi-explicit DAE, the differential index can be defined as 
the minimum number of times the (algebraic) constraint must be differentiated to put the 
DAE in the standard ODE form 
\begin{align*}
  \dot{\boldsymbol{z}} = \boldsymbol{q} (t, \boldsymbol{z}), \quad \boldsymbol{z} 
  := \left(\boldsymbol{x}^{\mathrm{T}},\boldsymbol{y}^{\mathrm{T}}\right)^{\mathrm{T}}
\end{align*} 
using purely algebraic manipulations. The higher the differential index, the greater 
is the difficulty to solve the DAE numerically. For the DAE given by equations 
\eqref{Eqn:Trapezoidal_domain_equation}-\eqref{Eqn:Trapezoidal_kinematic_continuity}, 
the differential index is $2$. Note that the differential index of a system of ODEs 
is $0$. For a thorough discussion of differential-algebraic equations refer to 
\cite{Ascher_Petzold,Mattheij_Molenaar}. 

It is well-known that numerical solvers for ODEs may not work well for DAEs. In many cases 
there are stability and accuracy (e.g., drift in the kinematic constraints) issues. This fact 
has been discussed in a seminal paper by Petzold \cite{Petzold_SIAMJSciStatComp_1982_v3_p367}. 
Time stepping schemes from the generalized trapezoidal family have been developed primarily 
for numerically solving ODEs. The stability results for most numerical integrators that have 
been reported in the literature (for example, Reference \cite{Hughes}) are appropriate for ODEs. 
The present paper addresses the stability of the time stepping schemes from the generalized 
trapezoidal family when applied to DAEs of the form given by equations 
\eqref{Eqn:Trapezoidal_domain_equation}-\eqref{Eqn:Trapezoidal_kinematic_continuity}.

\section{TIME DISCRETE EQUATIONS}
The time interval of interest $[0,T]$ is divided into $N$ equal time steps of size $\Delta t > 0$  
and the equations are to be solved numerically at discrete instants of time $t_n=t_0 + n \Delta t$ 
with $t_0 = 0$ and $t_N = T$. 
One of the popular ways to solve equation \eqref{Eqn:Trapezoidal_semi_discrete} numerically 
is by employing a time stepping scheme from the generalized trapezoidal family. The numerical 
solution procedure for the generalized trapezoidal family can be written as: given 
$(\boldsymbol{d}^{(n-1)}, \boldsymbol{v}^{(n-1)})$ obtain $(\boldsymbol{d}^{(n)}, 
\boldsymbol{v}^{(n)})$ by solving 
\begin{align}
  \label{Eqn:Trapezoidal_time_discrete}
  & \boldsymbol{M} \boldsymbol{v}^{(n)} + \boldsymbol{K} \boldsymbol{d}^{(n)} = \boldsymbol{f}^{(n)} \\
  \label{Eqn:Trapezoidal_generalized_trapezoidal}
  & \boldsymbol{d}^{(n)} = \boldsymbol{d}^{(n-1)} + \Delta t \left((1 - \gamma) \boldsymbol{v}^{(n-1)} 
  + \gamma \boldsymbol{v}^{(n)}\right)
\end{align}
where $0 \leq \gamma \leq 1$ is a user-defined parameter, and 
\begin{align}
  \label{Eqn:Trapezoidal_un_vn}
  \boldsymbol{d}^{(n)} \approx \boldsymbol{u}(t_n), \; 
  \boldsymbol{v}^{(n)} \approx \dot{\boldsymbol{u}}(t_n), \; 
  \boldsymbol{f}^{(n)} \approx \boldsymbol{f}(t_n) \quad \forall n
\end{align}
Some of the popular time integrators that belong to the generalized trapezoidal family 
are forward Euler $(\gamma = 0)$, midpoint rule $(\gamma = 1/2)$, and backward Euler 
$(\gamma = 1)$. It is well-known that the time stepping schemes from the generalized 
trapezoidal family when applied to ODEs are unconditionally stable for $\gamma \geq 1/2$ 
and conditionally stable for $\gamma < 1/2$. For further details see Reference \cite{Hughes}.

\subsection{Decomposed problem}
In the time discrete case of the decomposed problem two cases can be envisaged for enforcing 
the kinematic continuity. We can either prescribe continuity of temperature (which we refer to 
as $\boldsymbol{d}$-continuity), or continuity of rate of temperature (which we refer to as 
$\boldsymbol{v}$-continuity). These two kinematic continuity constraints can be mathematically 
written as follows.
\begin{align}
  \mbox{$\boldsymbol{d}$-continuity:} \quad \sum_{i=1}^{S} \boldsymbol{C}_i 
  \boldsymbol{d}_i^{(n)} = \boldsymbol{0} \quad \forall n \\
  \mbox{$\boldsymbol{v}$-continuity:} \quad \sum_{i=1}^{S} \boldsymbol{C}_i 
  \boldsymbol{v}_i^{(n)} = \boldsymbol{0} \quad \forall n
\end{align}
From a discrete point of view, we cannot, in general, enforce the continuity of both the temperature 
and the rates at the interface. (One exception, as we will show later, is the backward Euler with 
$\boldsymbol{d}$-continuity in which we can satisfy both the discrete kinematic continuity constraints, 
see Remark \ref{Remark:DD_backward_Euler_d_continuity}.) Based on the choice of kinematic continuity 
constraints we have two classes of domain decomposition method. We now present four different domain 
decomposition methods of which two methods employ $\boldsymbol{d}$-continuity, the third method employs 
$\boldsymbol{v}$-continuity, and the fourth uses a linear combination of $\boldsymbol{d}$-continuity and 
$\boldsymbol{v}$-continuity.

\section{DUAL SCHUR DOMAIN DECOMPOSITION METHODS}
\label{Sec:Trapezoidal_DD_methods}
\subsection{$\boldsymbol{d}$-continuity domain decomposition method}
This method can be written as: $\forall n = 1, \cdots, N$ and $\forall i = 1, \cdots, S$; 
obtain $\left(\boldsymbol{d}_i^{(n)},\boldsymbol{v}_i^{(n)}, \boldsymbol{\lambda}^{(n)}\right)$ 
by solving the following linear system of algebraic equations. 
\begin{align}
  \label{Eqn:Trapezoidal_d_continuity_DD_domain}
  & \boldsymbol{M}_{i} \boldsymbol{v}_{i}^{(n)} + \boldsymbol{K}_{i} \boldsymbol{d}_{i}^{(n)} = 
  \boldsymbol{f}_{i}^{(n)} + {\boldsymbol{C}_{i}}^{\mathrm{T}} \boldsymbol{\lambda}^{(n)} \\
  \label{Eqn:Trapezoidal_d_continuity_DD_trapezoidal}
  & \boldsymbol{d}_i^{(n)} = \boldsymbol{d}_i^{(n-1)} + \Delta t 
  \left( (1 - \gamma) \boldsymbol{v}_i^{(n-1)} + \gamma \boldsymbol{v}_i^{(n)}\right) \\
  \label{Eqn:Trapezoidal_d_continuity_DD_kinematic}
  & \sum_{i = 1}^{S} \boldsymbol{C}_{i} \boldsymbol{d}_i^{(n)} = \boldsymbol{0}
\end{align}

\subsection{Modified $\boldsymbol{d}$-continuity domain decomposition method}
This method is actually motivated by the stability analysis of the $\boldsymbol{d}$-continuity domain 
decomposition method, which is presented in a later section. The modified $\boldsymbol{d}$-continuity 
method can be written as: $\forall n = 0, \cdots, N-1$ and $\forall i = 1, \cdots, S$; obtain 
$\left(\boldsymbol{d}_i^{(n+1)}, \boldsymbol{v}_i^{(n + \gamma)}, \boldsymbol{\lambda}^{(n + \gamma)}\right)$ 
by solving the following linear system of algebraic equations.
\begin{align}
\label{Eqn:Trapezoidal_modified_d_continuity_DD_domain}
& \boldsymbol{M}_i \boldsymbol{v}_i^{(n + \gamma)} + \boldsymbol{K}_i \boldsymbol{d}_i^{(n+\gamma)} 
= \boldsymbol{f}_i^{(n+\gamma)} + \boldsymbol{C}_i^{\mathrm{T}} \boldsymbol{\lambda}^{(n + \gamma)}\\
\label{Eqn:Trapezoidal_modified_d_continuity_DD_trapezoidal}
& \boldsymbol{d}_i^{(n + 1)} = \boldsymbol{d}_i^{(n)} + \Delta t \boldsymbol{v}_i^{(n+\gamma)}; \; 
\boldsymbol{d}_i^{(n + \gamma)} = (1 - \gamma) \boldsymbol{d}_i^{(n)} + \gamma \boldsymbol{d}_i^{(n+1)} \\
\label{Eqn:Trapezoidal_modified_d_continuity_DD_kinematic}
& \sum_{i=1}^{S} \boldsymbol{C}_i \boldsymbol{d}_i^{(n+1)} = \boldsymbol{0} 
\end{align}
If required, calculate $\boldsymbol{v}_i^{(n+1)}$ and $\boldsymbol{\lambda}^{(n+1)}$ 
as follows (see Figure \ref{Fig:DD_interpolation_lambda}).
\begin{align}
  \label{Eqn:DD_modified_v_interpolation}
\boldsymbol{v}_i^{(n+1)} &= \gamma \boldsymbol{v}_i^{(n+\gamma)} + 
(1 - \gamma) \boldsymbol{v}_i^{(n+1+\gamma)} \\
  \label{Eqn:DD_modified_lambda_interpolation}
\boldsymbol{\lambda}^{(n+1)} &= \gamma \boldsymbol{\lambda}^{(n+\gamma)} + 
(1 - \gamma) \boldsymbol{\lambda}^{(n+1+\gamma)} 
\end{align}

\begin{figure}
  \psfrag{d}{$\boldsymbol{d}$-continuity method}
  \psfrag{md}{modified $\boldsymbol{d}$-continuity method}
  \psfrag{g}{$(1- \gamma) \Delta t$}
  \psfrag{g1}{$\gamma\Delta t$}
  \psfrag{tn11}{$t_{n-1}$}
  \psfrag{tn}{$t_{n}$}
  \psfrag{tn1}{$t_{n+1}$}
  \psfrag{tn2}{$t_{n+2}$}
  \psfrag{tn11g}{$t_{n-1 + \gamma}$}
  \psfrag{tng}{$t_{n + \gamma}$}
  \psfrag{tn1g}{$t_{n+1+\gamma}$}
  \psfrag{Ln11}{$\boldsymbol{\lambda}^{(n-1)}$}
  \psfrag{Ln}{$\boldsymbol{\lambda}^{(n)}$}
  \psfrag{Ln1}{$\boldsymbol{\lambda}^{(n+1)}$}
  \psfrag{Ln2}{$\boldsymbol{\lambda}^{(n+2)}$}
  \psfrag{Ln11g}{$\boldsymbol{\lambda}^{(n-1+\gamma)}$}
  \psfrag{Lng}{$\boldsymbol{\lambda}^{(n + \gamma)}$}
  \psfrag{Ln1g}{$\boldsymbol{\lambda}^{(n+1+\gamma)}$}
  \includegraphics[scale=0.75]{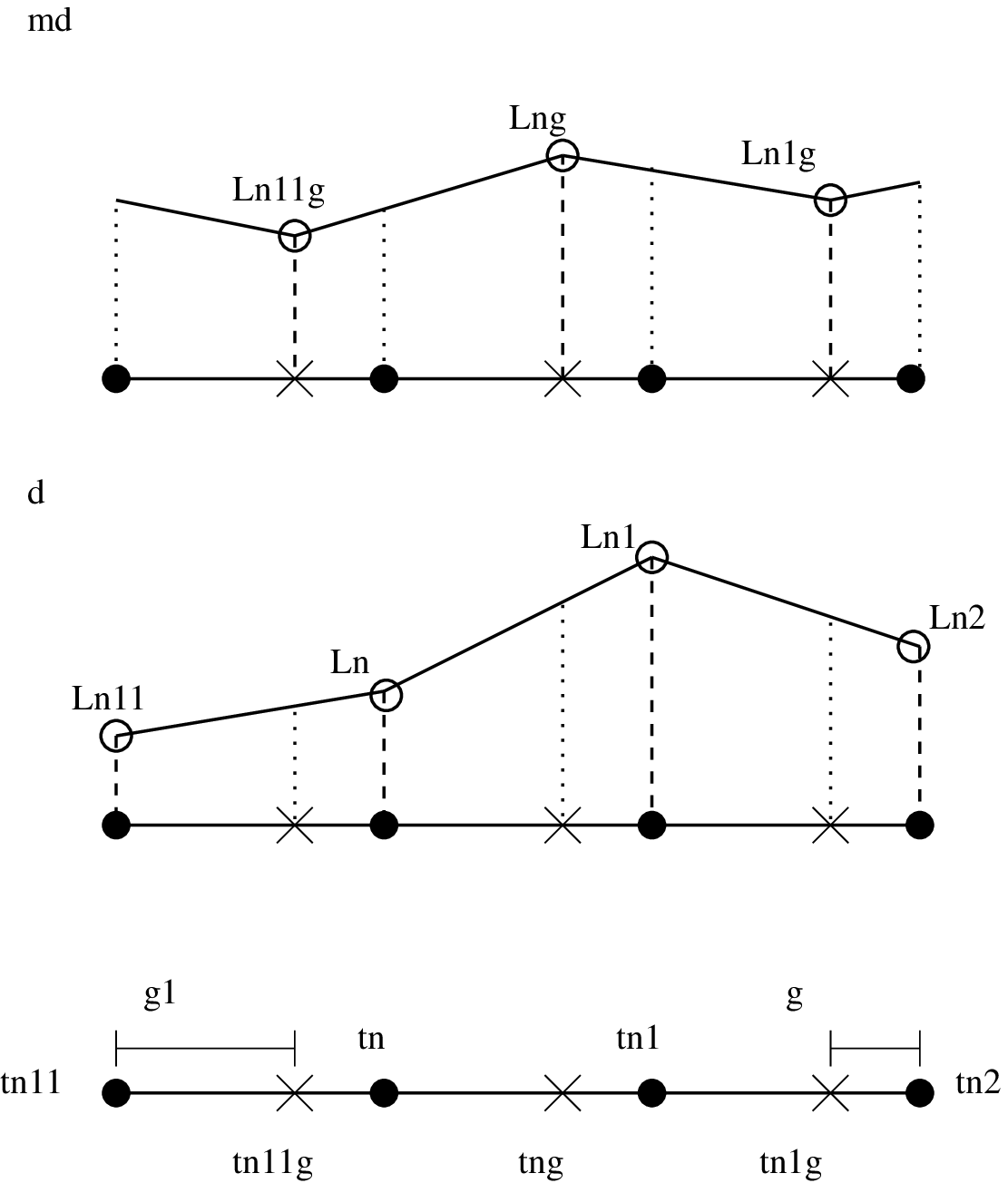}
  \caption{This figure shows the main difference between the interpolations used in the 
    $\boldsymbol{d}$-continuity and  modified $\boldsymbol{d}$-continuity domain decomposition 
    methods for calculating Lagrange multipliers and rates at various time levels. In the 
    $\boldsymbol{d}$-continuity method, one solves for the Lagrange multipliers (and rates) 
    at the integral time levels $t_n$. If one chooses, the Lagrange multipliers at the time 
    levels $t_{n + \gamma}$ can be obtained as $\boldsymbol{\lambda}^{(n + \gamma)} = (1 - \gamma) 
    \boldsymbol{\lambda}^{(n)} + \gamma \boldsymbol{\lambda}^{(n+1)}$, which is illustrated in the 
    middle figure. In the modified $\boldsymbol{d}$-continuity method one solves for the Lagrange 
    multipliers at the weighted time levels $t_{n + \gamma}$. If one chooses, the Lagrange multipliers 
    at integral time levels can be obtained as $\boldsymbol{\lambda}^{(n)} = \gamma 
    \boldsymbol{\lambda}^{(n - 1 + \gamma)} + (1 - \gamma) \boldsymbol{\lambda}^{(n+\gamma)}$, which 
    is illustrated in the top figure. Similar explanation holds for interpolating the rates 
    at various time levels. \label{Fig:DD_interpolation_lambda}}
\end{figure}

\begin{remark}
In the modified $\boldsymbol{d}$-continuity domain decomposition method the equilibrium is 
enforced at time level $(n + \gamma)$ and the kinematic constraints are enforced at time level 
$(n+1)$ (that is, at an integer time level). See Appendix for a numerical implementation procedure 
for the modified $\boldsymbol{d}$-continuity method. 
\end{remark}
\begin{remark}
  \label{Remark:Trapezoidal_forward_Euler_ill_posed}
  One cannot employ the forward Euler $(\gamma = 0)$ in the $\boldsymbol{d}$-continuity 
  and modified $\boldsymbol{d}$-continuity domain decomposition methods. To see this, 
  let us consider the $\boldsymbol{d}$-continuity method. 
  
  For the case $\gamma = 0$, given $\boldsymbol{d}_i^{(n-1)}$ and $\boldsymbol{v}_i^{(n-1)}$, 
  the calculation of $\boldsymbol{d}_i^{(n)}$ becomes explicit as we can directly evaluate 
  the quantity using equation \eqref{Eqn:Trapezoidal_d_continuity_DD_trapezoidal}. 
  The kinematic constraints given by equation \eqref{Eqn:Trapezoidal_d_continuity_DD_kinematic} 
  may not be consistent with the $\boldsymbol{d}_i^{(n)}$ computed from equation 
  \eqref{Eqn:Trapezoidal_d_continuity_DD_trapezoidal}. 
  We still need to find $\boldsymbol{v}_i^{(n)}$ and $\boldsymbol{\lambda}^{(n)}$ but we have only the 
  subdomain equation \eqref{Eqn:Trapezoidal_d_continuity_DD_domain}, which is also of the same size as 
  the individual subdomain vectors $\boldsymbol{v}_i^{(n)}$. With forward Euler, the system given by equations 
  \eqref{Eqn:Trapezoidal_d_continuity_DD_domain}-\eqref{Eqn:Trapezoidal_d_continuity_DD_kinematic} 
  will be under-determined, and hence the problem is not well-posed. Similarly, one can show that 
  the forward Euler is not compatible with the modified $\boldsymbol{d}$-continuity method. 
\end{remark}

\subsection{$\boldsymbol{v}$-continuity domain decomposition method} 
This method is based on the standard index reduction technique by analytical differentiation 
of the (kinematic) constraints. The method can be written as: $\forall n = 1, \cdots, N$ and 
$\forall i = 1, \cdots, S$; obtain $\left(\boldsymbol{d}_i^{(n)},\boldsymbol{v}_i^{(n)}, 
\boldsymbol{\lambda}^{(n)}\right)$ by solving the following linear system of algebraic 
equations. 
\begin{align}
  \label{Eqn:Trapezoidal_v_continuity_DD_domain}
  & \boldsymbol{M}_{i} \boldsymbol{v}_{i}^{(n)} + \boldsymbol{K}_{i} \boldsymbol{d}_{i}^{(n)} = 
  \boldsymbol{f}_{i}^{(n)} + {\boldsymbol{C}_{i}}^{\mathrm{T}} \boldsymbol{\lambda}^{(n)} \\
  \label{Eqn:Trapezoidal_v_continuity_DD_trapezoidal}
  & \boldsymbol{d}_i^{(n)} = \boldsymbol{d}_i^{(n-1)} + \Delta t 
  \left( (1 - \gamma) \boldsymbol{v}_i^{(n-1)} + \gamma \boldsymbol{v}_i^{(n)}\right) \\
  \label{Eqn:Trapezoidal_v_continuity_DD_kinematic}
  & \sum_{i = 1}^{S} \boldsymbol{C}_{i} \boldsymbol{v}_i^{(n)} = \boldsymbol{0}
\end{align}

Since the $\boldsymbol{v}$-continuity method enforces the constraints on the rates, one 
may get significant drift in the original constraint (i.e., continuity of temperature 
along the subdomain interfaces) for larger time steps and for larger number of steps 
(i.e., longer simulation time). This irrecoverable drift in the original constraint 
makes the $\boldsymbol{v}$-continuity method undesirable in such situations. 
One of the popular ways to control drift in the constraints is Baumgarte stabilization, which 
has been proposed in the context of multibody dynamics \cite{Baumgarte_CMAME_1972_v1_p1}. With a 
slight modification, one can extend Baumgarte stabilization to index-2 DAEs first-order transient 
systems (which is the case in this paper). Following this approach, we construct a new domain 
decomposition method, which is outlined in the next subsection. 

\subsection{Baumgarte stabilized domain decomposition method}
This method enforces a linear combination of the $\boldsymbol{d}$-continuity and 
$\boldsymbol{v}$-continuity as the kinematic constraints. 
The method can be written as: $\forall n = 1, \cdots, N$ and $\forall i = 1, \cdots, S$; 
obtain $\left(\boldsymbol{d}_i^{(n)},\boldsymbol{v}_i^{(n)}, \boldsymbol{\lambda}^{(n)} 
\right)$ by solving the following linear system of algebraic equations:
\begin{align}
  \label{Eqn:Trapezoidal_Baumgarte_domain}
  & \boldsymbol{M}_{i} \boldsymbol{v}_{i}^{(n)} + \boldsymbol{K}_{i} \boldsymbol{d}_{i}^{(n)} = 
  \boldsymbol{f}_{i}^{(n)} + {\boldsymbol{C}_{i}}^{\mathrm{T}} \boldsymbol{\lambda}^{(n)} \\
  \label{Eqn:Trapezoidal_Baumgarte_trapezoidal}
  & \boldsymbol{d}_i^{(n)} = \boldsymbol{d}_i^{(n-1)} + \Delta t 
  \left( (1 - \gamma) \boldsymbol{v}_i^{(n-1)} + \gamma \boldsymbol{v}_i^{(n)}\right) \\
  \label{Eqn:Trapezoidal_Baumgarte_kinematic}
  &\sum_{i=1}^{S} \boldsymbol{C}_i \boldsymbol{v}_i^{(n)} + \frac{\alpha}{\Delta t} 
  \sum_{i=1}^{S} \boldsymbol{C}_i \boldsymbol{d}_i^{(n)} = \boldsymbol{0} 
\end{align}
where $\alpha > 0$ is a dimensionless user-defined parameter, which we will refer to as 
the Baumgarte parameter. As we will see in Section \ref{Sec:DD_S4_Stability_Analysis}, 
the choice of $\alpha$ will effect the accuracy (the drift in the original constraint, 
that is, continuity of temperature along the subdomain interface) and stability (i.e., 
the critical time step). Typically, the larger the Baumgarte parameter, the smaller will be the 
drift in the original constraint, but also smaller may be the critical time step (which 
may depend on the choice of $\gamma$). In Section \ref{Sec:DD_S4_Stability_Analysis} we 
derive bounds on the choice of $\alpha$ that ensures stability of the domain decomposition 
method. 

Also, note that, unlike the $\boldsymbol{d}$-continuity and modified $\boldsymbol{d}$-continuity 
methods, one can employ the forward Euler $(\gamma = 0)$ under the Baumgarte stabilized domain 
decomposition method. 

In the remainder of the paper we will address the following questions by providing 
rigorous mathematical justifications.
\begin{itemize}
\item Under what conditions do the $\boldsymbol{d}$-continuity and $\boldsymbol{v}$-continuity 
  domain decomposition methods give stable results?
\item Do the algebraic equations alter the stability of the numerical time stepping 
  schemes from the generalized trapezoidal family? If so, which of these time stepping 
  schemes are stable?
\item Do any of the time-stepping schemes from the generalized trapezoidal family 
  achieve continuity of both temperature and rate of temperature along the subdomain 
  interface? 
\end{itemize}
\begin{remark}
Before we perform stability analyses of the methods presented in the previous section, we 
briefly comment on the implementation of dual Schur domain decomposition methods. Various 
ways of implementating dual Schur domain decomposition methods have been proposed in the 
literature. A thorough description is beyond the scope of this paper. But, herein, we just 
cite few representative papers that can be consulted for effective implementation of dual 
Schur domain decomposition methods. 

The FETI method is a popular way of implementing the dual Schur domain decomposition 
method which has been shown to possess good convergence and scability properties 
\cite{Farhat_Roux_IJNME_1991_v32_p1205,Farhat_Chen_Mandel_IJNME_1995_v38_p3831,
  Farhat_Crivelli_Roux_IJNME_1994_v37_p1945}. An interesting approach (in the context 
of structural dynamics) has been developed in \cite{Gravouil_Combescure_IJNME_2001_v50_p199,
  Prakash_Hjelmstad_IJNME_2004_v61_p2183}. Any of the aforementioned implementation methodologies 
can be employed for the methods presented in this paper. 
\end{remark}

\section{STABILITY ANALYSIS}
\label{Sec:DD_S4_Stability_Analysis}
In this section we assess the stability of the $\boldsymbol{d}$-continuity, modified 
$\boldsymbol{d}$-continuity, $\boldsymbol{v}$-continuity and Baumgarte stabilized 
domain decomposition methods using the ``energy'' method \cite{Richtmyer_Morton,Hughes}. Note that 
the energy method provides sufficient conditions for stability. However, in many 
instances the obtained bounds are quite sharp. For example, in the case of ODEs, 
the obtained stability condition using the energy method (that is, the critical 
time step) has turned out to be both necessary and sufficient \cite{Hughes}. 

Let $\mathbb{R}^m$ denote the standard $m$-dimensional Euclidean space. We say the 
vectors $\boldsymbol{x}^{(n)} \in \mathbb{R}^{m} \; (n = 0, 1, 2, \cdots)$ are bounded 
$\forall n$ if there exists a positive constant $C$ independent of $n$ such that 
\begin{align}
  \| \boldsymbol{v}^{(n)}\| < C \quad \forall n
\end{align}
where $\| \cdot \|$ is some convenient norm defined on $\mathbb{R}^m$ (for example, 
say the $2$-norm). Note that in finite dimensional vector spaces all norms are 
equivalent \cite{Halmos}. 

We employ the following notation for the jump and average operators over a time step:  
\begin{align}
  \left[\boldsymbol{x}^{(n)}\right] = \boldsymbol{x}^{(n+1)} - \boldsymbol{x}^{(n)}, \quad   
  \left\{\boldsymbol{x}^{(n)}\right\} = \frac{1}{2} \left(\boldsymbol{x}^{(n+1)} + \boldsymbol{x}^{(n)} \right)
\end{align}
It is easy to show the following identities: 
\begin{align}
  \label{Eqn:Trapezoidal_x_gamma}
  (1 - \gamma) \boldsymbol{x}^{(n)} + \gamma \boldsymbol{x}^{(n+1)} 
  = \left(\gamma - \frac{1}{2}\right) \left[\boldsymbol{x}^{(n)}\right] + \left\{\boldsymbol{x}^{(n)}\right\}
\end{align}
and for any symmetric matrix $\boldsymbol{S}$ we have 
\begin{align}
  \label{Eqn:Trapezoidal_Symmetric_S}
  \left\{\boldsymbol{x}^{(n)}\right\}^{\mathrm{T}} \boldsymbol{S} 
  \left[\boldsymbol{x}^{(n)}\right] = 
  \frac{1}{2} \left[{\boldsymbol{x}^{(n)}}^{\mathrm{T}} \boldsymbol{S} \boldsymbol{x}^{(n)} \right] 
\end{align}

For convenience we define the matrix $\boldsymbol{A}_i$ as 
\begin{align}
  \label{Eqn:Trapezoidal_A_i}
  \boldsymbol{A}_i := \boldsymbol{M}_i + \left(\gamma - \frac{1}{2}\right) \Delta t \boldsymbol{K}_i
\end{align}

We choose the time step $\Delta t$ in such a way that it satisfies the stability requirements 
for all individual unconstrained subdomains (i.e., $\boldsymbol{\lambda} = \boldsymbol{0}$). 
That is,
\begin{align}
  \Delta t < \mathrm{min} \left(\Delta t_{1}^{\mathrm{crit}}, \cdots, \Delta t_{S}^{\mathrm{crit}}\right)
\end{align}
and the critical time step for subdomain $i$ is given by 
\begin{align}
  \Delta t_{i}^{\mathrm{crit}} = \left\{\begin{array}{cl}
      \frac{2}{\omega_{i}^{\mathrm{max}} (1 - 2 \gamma)} & 0 \leq \gamma < 1/2 \\
      +\infty & 1/2 \leq \gamma \leq 1
    \end{array}\right.
\end{align}
where $\omega_{i}^{\mathrm{max}}$ is the maximum eigenvalue of the generalized eigenvalue problem 
\begin{align}
  \label{Eqn:Trapezoidal_GEVP}
  \omega_i \boldsymbol{M}_i \boldsymbol{\phi}_i = \boldsymbol{K}_{i} \boldsymbol{\phi}_i
\end{align}
Note that all the eigenvalues $\omega_i$ are real and non-negative \cite{Hughes}.
For the chosen time step as described above, the matrices $\boldsymbol{A}_i \; (i = 1, \cdots, S)$ 
(defined in equation \eqref{Eqn:Trapezoidal_A_i}) are positive definite. Also, note that the matrices 
$\boldsymbol{A}_i$ are symmetric. 

\subsection{Stability of the $\boldsymbol{d}$-continuity method}
\label{Subsec:DD_d_continuity}
Equation \eqref{Eqn:Trapezoidal_d_continuity_DD_domain} implies 
\begin{align}
  \boldsymbol{M}_i \boldsymbol{v}^{(n + \gamma)}_i
  + \boldsymbol{K}_i \boldsymbol{d}^{(n + \gamma)}_i
  = \boldsymbol{f}_i^{(n + \gamma)} + \boldsymbol{C}_i^{\mathrm{T}} \boldsymbol{\lambda}^{(n + \gamma)} 
\end{align}
where $\boldsymbol{\lambda}^{(n + \gamma)} := (1 - \gamma) \boldsymbol{\lambda}^{(n)} 
+ \gamma \boldsymbol{\lambda}^{(n+1)}$; $\boldsymbol{f}_i^{(n + \gamma)} := (1 - \gamma) 
\boldsymbol{f}_i^{(n)} + \gamma \boldsymbol{f}_i^{(n+1)}$; and similar expressions for 
$\boldsymbol{d}_i^{(n+\gamma)}$ and $\boldsymbol{v}_i^{(n + \gamma)}$.
Using equations \eqref{Eqn:Trapezoidal_d_continuity_DD_trapezoidal}, \eqref{Eqn:Trapezoidal_x_gamma} 
and \eqref{Eqn:Trapezoidal_A_i}; the above equation can be written as 
\begin{align}
  \frac{1}{\Delta t} \boldsymbol{A}_i \left[\boldsymbol{d}^{(n)}_i \right] 
  + \boldsymbol{K}_i \left\{\boldsymbol{d}^{(n)}_i\right\}
  = \boldsymbol{f}_i^{(n + \gamma)} + \boldsymbol{C}_i^{\mathrm{T}} \boldsymbol{\lambda}^{(n + \gamma)}
\end{align}
Premultiplying the above equation by the vector $\left\{\boldsymbol{d}_i^{(n)}\right\}$,  
using the identity given in equation \eqref{Eqn:Trapezoidal_Symmetric_S}, and then summing 
over all $S$ subdomains; we get 
\begin{align}
  \sum_{i = 1}^{S} \frac{1}{2 \Delta t} \left[{\boldsymbol{d}^{(n)}_i}^{\mathrm{T}}
    \boldsymbol{A}_i \boldsymbol{d}^{(n)}_i \right] 
  + \sum_{i=1}^{S} {\left\{\boldsymbol{d}_{i}^{(n)}\right\}}^\mathrm{T} 
  \boldsymbol{K}_i \left\{\boldsymbol{d}^{(n)}_i\right\}
  = \sum_{i=1}^{S} {\left\{\boldsymbol{d}_i^{(n)}\right\}}^{\mathrm{T}} 
  \left(\boldsymbol{f}_i^{(n + \gamma)} + \boldsymbol{C}_i^{\mathrm{T}} \boldsymbol{\lambda}^{(n + \gamma)}\right) 
\end{align}
For stability analysis, one assumes the externally applied forces to be zero (i.e., 
$\boldsymbol{f}_i^{(n)} = \boldsymbol{0}; \; \forall i = 1,\cdots,S; \; \forall n$) 
\cite{Hughes}. Thus, we have
\begin{align}
  \sum_{i = 1}^{S} \frac{1}{2 \Delta t} \left[{\boldsymbol{d}^{(n)}_i}^{\mathrm{T}}
    \boldsymbol{A}_i \boldsymbol{d}^{(n)}_i \right] 
  + \sum_{i=1}^{S} {\left\{\boldsymbol{d}_{i}^{(n)}\right\}}^\mathrm{T} 
  \boldsymbol{K}_i \left\{\boldsymbol{d}^{(n)}_i\right\}
  = \sum_{i=1}^{S} {\left\{\boldsymbol{d}_i^{(n)}\right\}}^{\mathrm{T}} 
  \boldsymbol{C}_i^{\mathrm{T}} \boldsymbol{\lambda}^{(n + \gamma)}
\end{align}
By invoking the fact that the matrices $\boldsymbol{K}_i \; (i = 1, \cdots , S)$ 
are positive semidefinite, we conclude that 
\begin{align}
  \sum_{i = 1}^{S} \frac{1}{2 \Delta t} \left[{\boldsymbol{d}^{(n)}_i}^{\mathrm{T}}
    \boldsymbol{A}_i \boldsymbol{d}^{(n)}_i \right] 
  \leq \sum_{i = 1}^{S} {\left\{\boldsymbol{d}^{(n)}_i \right\}}^{\mathrm{T}} \boldsymbol{C}_i^{\mathrm{T}} 
  \boldsymbol{\lambda}^{(n + \gamma)} 
  = {\boldsymbol{\lambda}^{(n + \gamma)}}^{\mathrm{T}} \sum_{i=1}^{S} \boldsymbol{C}_i \left\{\boldsymbol{d}^{(n)}_i\right\} 
\end{align}
Using the linearity of the average operator (which allows us to interchange the 
summation and average operation), and the $\boldsymbol{d}$-continuity given by 
equation \eqref{Eqn:Trapezoidal_d_continuity_DD_kinematic}; we conclude that 
\begin{align}
 \sum_{i = 1}^{S} \frac{1}{2 \Delta t} \left[{\boldsymbol{d}^{(n)}_i}^{\mathrm{T}}
   \boldsymbol{A}_i \boldsymbol{d}^{(n)}_i \right] \leq 0 \quad \forall n
\end{align}
Using the definition and linearity of the jump operator, we conclude 
\begin{align}
  \sum_{i=1}^{S}{\boldsymbol{d}_i^{(n+1)}}^{\mathrm{T}}\boldsymbol{A}_{i} \boldsymbol{d}_i^{(n+1)} \leq 
  \sum_{i=1}^{S} {\boldsymbol{d}_i^{(n)}}^{\mathrm{T}}\boldsymbol{A}_{i} \boldsymbol{d}_i^{(n)} 
  \leq \cdots \leq 
  \sum_{i=1}^{S} {\boldsymbol{d}_i^{(0)}}^{\mathrm{T}}\boldsymbol{A}_{i} \boldsymbol{d}_i^{(0)} 
\end{align}
Since $\forall i = 1, \cdots, S$ the initial vectors $\boldsymbol{d}_i^{(0)}$ are bounded 
and the matrices $\boldsymbol{A}_i$ are positive definite; we conclude that the vectors 
$\boldsymbol{d}_i^{(n)} \; (i = 1, \cdots, S)$ are bounded $\forall n$. This implies that, 
from the trapezoidal equation  \eqref{Eqn:Trapezoidal_d_continuity_DD_trapezoidal}, the 
vectors $\boldsymbol{v}_i^{(n+\gamma)} := (1 - \gamma) \boldsymbol{v}_i^{(n)} + \gamma 
\boldsymbol{v}_i^{(n+1)} \; (i = 1, \cdots, S)$ are also bounded $\forall n$. 

We now show that the Lagrange multipliers $\boldsymbol{\lambda}^{(n + \gamma)}$ are bounded 
$\forall n$. There are several ways in proving this result. Herein, we show using the Schur 
complement operator (which is widely used in computer implementations, see Appendix). An 
alternate derivation is presented in subsection \ref{Subsec:DD_stability_v_continuity} for 
proving a similar result. Under zero external force the subdomain governing equation 
\eqref{Eqn:Trapezoidal_d_continuity_DD_domain} implies 
\begin{align}
  \label{Eqn:DD_d_continuity_step5}
  \boldsymbol{M}_i \boldsymbol{v}_i^{(n + \gamma)} + \boldsymbol{K}_i \boldsymbol{d}_i^{(n + \gamma)} 
  = \boldsymbol{C}_i^{\mathrm{T}} \boldsymbol{\lambda}^{(n + \gamma)}, \quad \forall i
\end{align}
Using $\boldsymbol{v}_i^{(n + \gamma)} = \left(\boldsymbol{d}_i^{(n+1)} - \boldsymbol{d}_i^{(n)}
\right)/\Delta t$ (see equation \eqref{Eqn:Trapezoidal_d_continuity_DD_trapezoidal}) and 
$\boldsymbol{d}_i^{(n + \gamma)} = (1 - \gamma) \boldsymbol{d}_i^{(n)} + \gamma \boldsymbol{d}_i^{(n+1)}$, 
the above equation can be written as 
\begin{align}
  \label{Eqn:DD_d_continuity_step5a}
  & \boldsymbol{d}_i^{(n+1)} = \tilde{\boldsymbol{M}}_i^{-1} \left(\boldsymbol{M}_i - (1 - \gamma) 
    \Delta t \boldsymbol{K}_i \right) \boldsymbol{d}_i^{(n)} + \Delta t \tilde{\boldsymbol{M}}_i^{-1} 
  \boldsymbol{C}_i^{\mathrm{T}} \boldsymbol{\lambda}^{(n + \gamma)} \\
  \label{Eqn:DD_d_continuity_M_i_tilde}
  & \mbox{where} \quad \tilde{\boldsymbol{M}}_i  := \boldsymbol{M}_i + \gamma \Delta t \boldsymbol{K}_i
\end{align}
By premultiplying equation \eqref{Eqn:DD_d_continuity_step5a} with $\boldsymbol{C}_i$, 
then summing over the number of subdomains (i.e., the index $i = 1, \cdots, S$), and 
using the kinematic constraint for $\boldsymbol{d}$-continuity method (equation 
\eqref{Eqn:Trapezoidal_d_continuity_DD_kinematic}); we get 
\begin{align}
  \label{Eqn:DD_d_continuity_step5b}
  & \boldsymbol{\lambda}^{(n + \gamma)} = -\frac{1}{\Delta t} \boldsymbol{G}^{-1} 
  \sum_{i = 1}^{S} \boldsymbol{C}_i \tilde{\boldsymbol{M}}_i^{-1} \left(\boldsymbol{M}_i - 
    (1 - \gamma) \Delta t \boldsymbol{K}_i \right) \boldsymbol{d}_i^{(n)} \\
  \label{Eqn:DD_d_continuity_G}
  & \mbox{where} \quad \boldsymbol{G} := \sum_{i=1}^{S} \boldsymbol{C}_i 
  \tilde{\boldsymbol{M}}_i^{-1} \boldsymbol{C}_i^{\mathrm{T}} 
\end{align}
Since the vectors $\boldsymbol{d}_i^{(n)}$ are bounded $\forall n$, from equation 
\eqref{Eqn:DD_d_continuity_step5b}, we conclude that the Lagrange multipliers 
$\boldsymbol{\lambda}^{(n+\gamma)}$ are also bounded $\forall n$. 
\begin{remark}
  \label{Remark:DD_d_continuity_M_i_tilde_G_positive_definite}
  Since the matrix $\tilde{\boldsymbol{M}}_i$ is positive definite (and hence invertible), 
  all the steps in equation \eqref{Eqn:DD_d_continuity_step5a} are valid.  
  Also, since the constraints are assumed to be linearly independent, one can easily 
  show that the matrix $\boldsymbol{G}$ (which is sometimes called the Schur complement 
  operator) is positive definite (and hence invertible). Therefore, all the steps in 
  equation \eqref{Eqn:DD_d_continuity_step5b} are valid. Also see Appendix. 
\end{remark} 

But the proof of stability is not yet complete as we have not said anything about the 
boundedness of $\boldsymbol{v}_i^{(n)}$ and $\boldsymbol{\lambda}^{(n)}$ and these 
quantities are used to advance the solution. Before we comment on the boundedness of 
rates and Lagrange multipliers at integer time levels $t_n$, we state and prove the 
following general result, which will be useful in assessing stability of coupling 
algorithms. We employ the standard notation used in mathematical analysis. To avoid 
ambiguity, let $\mathbb{N} := \left\{0,1,2, \cdots\right\}$.
\begin{proposition}
  \label{Proposition:DD_Proposition_on_sn_gamma}
  Let $\left(s^{(n)}\right)_{n \in \mathbb{N}}$ be a sequence of real numbers and let 
  $s^{(n + \gamma)} := (1 - \gamma) s^{(n)} + \gamma s^{(n+1)}$, where $0 \leq \gamma 
  \leq 1$. If the element $s^{(0)}$ and the sequence $\left(s^{(n + \gamma)}
  \right)_{n \in \mathbb{N}}$ are bounded, then for any $\gamma > 1/2$ the (original) 
  sequence $\left(s^{(n)}\right)_{n \in \mathbb{N}}$ is also bounded. 
\end{proposition}
\begin{proof}
We split the proof into two different cases: $\gamma = 1$ and $1/2 < \gamma < 1$.

\textit{First case $(\gamma = 1)$}: The proof is trivial as in this case $s^{(n + \gamma)} 
= s^{(n+1)}$. The sequence $\left(s^{(n)}\right)_{n \in \mathbb{N}}$ can be obtained by 
augmenting the element $s^{(0)}$ at the start of the sequence 
$\left(s^{(n + \gamma)}\right)_{n \in \mathbb{N}}$. Since the element $s^{(0)}$ and the 
sequence $\left(s^{(n + \gamma)}\right)_{n \in \mathbb{N}}$ are bounded, we conclude that 
the original sequence $\left(s^{(n)}\right)_{n \in \mathbb{N}}$ is also bounded. 

\textit{Second case $(1/2 < \gamma < 1)$}: Construct a new sequence $\left(c^{(n)}\right)_{n \in \mathbb{N}}$ 
such that $c^{(n)} = \left|s^{(n)}\right|$. If $\left(s^{(n)}\right)_{n \in \mathbb{N}}$ is bounded then 
$\left(c^{(n)}\right)_{n \in \mathbb{N}}$ is also bounded and vice-versa. We now prove the proposition 
by the method of contradiction. 

Assume that the sequence $\left(c^{(n)}\right)_{n \in \mathbb{N}}$ is unbounded. Then we can 
find a strictly increasing unbounded subsequence $\left(c^{(n_k)}\right)_{k \in \mathbb{N}} 
\; (n_k \in \mathbb{N}, 0 \leq n_1 < n_2 < \cdots)$ such that 
\begin{align}
  & 0 < c^{(n_1)} < c^{(n_2)} < \cdots \rightarrow \infty \quad \mathrm{and} \\
  \label{Eqn:Trapezoidal_cm_leq_cnk}
  & c^{(m)} < c^{(n_k)} \quad  \forall m < n_k, m \in \mathbb{N}
\end{align}
Since the sequence $\left(s^{(n + \gamma)}\right)_{n \in \mathbb{N}}$ is bounded, we can find 
$0 < M \in \mathbb{R}$ such that 
\begin{align}
  \label{Eqn:Trapezoidal_sngamma_leq_M}
  \left|s^{(n + \gamma)}\right| < M \quad \forall n
\end{align}
Since the subsequence $\left(c^{(n_k)}\right)_{k \in \mathbb{N}}$ is strictly increasing and 
unbounded, we can find an element of this subsequence such that 
\begin{align}
  \label{Eqn:Trapezoidal_cnp1_geq_M}
  c^{(n_p)} \geq \frac{1}{2 \gamma - 1} M, \quad p \in \mathbb{N}
\end{align}
Since $1/2 < \gamma < 1$ (and, therefore $\gamma > 2 \gamma - 1$), from equations 
\eqref{Eqn:Trapezoidal_sngamma_leq_M} and \eqref{Eqn:Trapezoidal_cnp1_geq_M} we have  
\begin{align}
  \label{Eqn:Trapezoidal_main_inequality}
  \gamma c^{(n_p)} > (2 \gamma - 1) c^{(n_p)} \geq M >  \left|s^{(n + \gamma)}\right| \geq 0 \quad \forall n
\end{align}

Using the definition of $s^{(n_p - 1 + \gamma)}$ we have 
\begin{align}
  c^{(n_p - 1)} = \left|s^{(n_p - 1)}\right| &= \frac{1}{1 - \gamma} \left|s^{(n_p - 1 + \gamma)} - \gamma s^{(n_p)}\right| 
\end{align}
By using the triangle inequality (and noting that $c^{(n_p)} := |s^{(n_p)}|$) we conclude that 
\begin{align}
  c^{(n_p - 1)} \geq \frac{1}{1 - \gamma} \left|\gamma \left|s^{(n_p)}\right| - \left|s^{(n_p - 1 + \gamma)}\right| \right| 
  = \frac{1}{1 - \gamma} \left|\gamma c^{(n_p)} - \left|s^{(n_p - 1 + \gamma)}\right| \right| 
\end{align}
By using equation \eqref{Eqn:Trapezoidal_main_inequality} we obtain 
\begin{align} 
  c^{(n_p - 1)} > \frac{1}{1 - \gamma} \left|\gamma c^{(n_p)} - (2 \gamma - 1) c^{(n_p)}\right| = c^{(n_p)} 
\end{align}
which is a contradiction as it violates equation \eqref{Eqn:Trapezoidal_cm_leq_cnk}. This 
implies that the sequence $\left(c^{(n)}\right)_{n \in \mathbb{N}}$ is bounded, and so should 
be the sequence $\left(s^{(n)}\right)_{n \in \mathbb{N}}$. This completes the proof. 
\end{proof}
\begin{remark}
  \label{Remark:DD_Remark_on_gamma_leq_dot5}
  The result proved in Proposition \ref{Proposition:DD_Proposition_on_sn_gamma}, in general, 
  cannot be extended to $0 < \gamma \leq 1/2$. Counterexamples for the cases $0 < \gamma 
  < 1/2$ and $\gamma = 1/2$ are given in figures \ref{Fig:DD_Unstable_sequence_gamma_leq_dot5} 
  and \ref{Fig:DD_Unstable_sequence_gamma_dot5}, respectively. As discussed in Remark 
  \ref{Remark:Trapezoidal_forward_Euler_ill_posed}, the forward Euler $(\gamma = 0)$ is 
  not well-posed under the $\boldsymbol{d}$-continuity method. 
  This implies that, under the $\boldsymbol{d}$-continuity domain decomposition method, many of 
  the popular time stepping schemes from the generalized trapezoidal family (e.g., the forward 
  Euler and midpoint rule) are unstable. The midpoint rule $(\gamma = 1/2)$ is unconditionally 
  stable for linear first-order ODEs. But for linear index-2 DAEs, the midpoint rule can be 
  unstable (and is right on the boundary of the instability region). 
\end{remark}

\begin{figure}
  \psfrag{0}{$0$}
  \psfrag{1}{$1$}
  \psfrag{2}{$2$}
  \psfrag{3}{$3$}
  \psfrag{4}{$4$}
  \psfrag{n}{$n$}
  \psfrag{g}{$\gamma$}
  \psfrag{gl}{$\gamma < 1/2$}
  \psfrag{s0}{$s^{(0)}$}
  \psfrag{s1}{$s^{(1)}$}
  \psfrag{s2}{$s^{(2)}$}
  \psfrag{s3}{$s^{(3)}$}
  \psfrag{sn}{$s^{(n)}$}
  \psfrag{sngamma}{$s^{(n + \gamma)}, \forall n$}
  \includegraphics[scale=0.7]{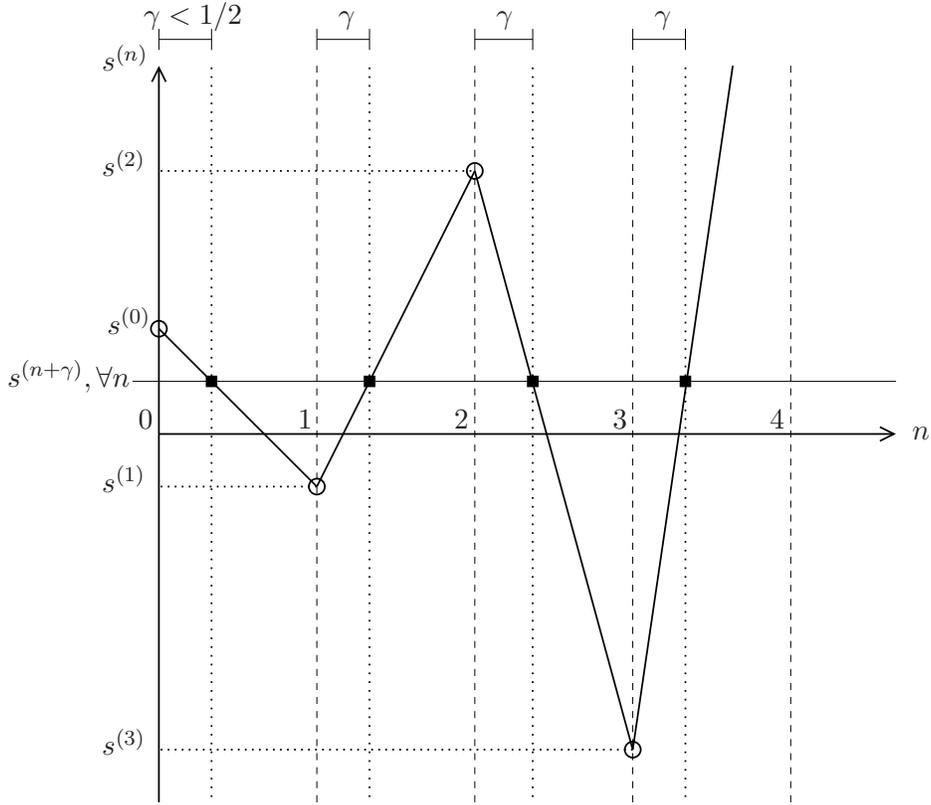}
  \caption{A counterexample for the case $0 \leq \gamma < 1/2$. The figure presents an example 
    in which the (constant) sequence $(s^{(n + \gamma)})_{n \in \mathbb{N}}$ is bounded but 
    $(s^{(n)})_{n \in \mathbb{N}}$ is unbounded. The elements of $(s^{(n+ \gamma)})_{n \in 
      \mathbb{N}}$ are indicated using filled squares, and those of $(s^{(n)})_{n \in 
      \mathbb{N}}$ are indicated using circles. 
    Recall that $s^{(n + \gamma)} := (1 - \gamma) s^{(n)} + \gamma s^{(n+1)}, \; \forall n$. 
  \label{Fig:DD_Unstable_sequence_gamma_leq_dot5}} 
\end{figure}

\begin{figure}
  \begin{align*}
    \begin{array}{c|cccccc} 
      n & 0 & 1 & 2 & 3 & 4 & \cdots \\ \hline
      s^{(n + 1/2)} & +1 & -1 & +1 & -1 & +1 & \cdots \\
      s^{(n)} & 0 & +2 & -4 & +6 & -8 & \cdots  
    \end{array}
  \end{align*}
  \caption{A counterexample for $\gamma = 1/2$. The sequence $\left(s^{(n + 1/2)}\right)_{n \in \mathbb{N}}, 
    \; s^{(n + 1/2)} := \left(s^{(n)} + s^{(n+1)}\right)/2,$ is bounded but the sequence 
    $\left(s^{(n)}\right)_{n \in \mathbb{N}}$ is unbounded. \label{Fig:DD_Unstable_sequence_gamma_dot5}}
\end{figure}

Returning to the proof of stability of the $\boldsymbol{d}$-continuity domain decomposition method, 
we have proved that $\boldsymbol{v}_i^{(n + \gamma)}$ and $\boldsymbol{\lambda}^{(n + \gamma)}$ are 
bounded $\forall i$ and $\forall n$. Applying Proposition \ref{Proposition:DD_Proposition_on_sn_gamma} 
for individual components of the vectors $\boldsymbol{v}_{i}^{(n + \gamma)}$ and 
$\boldsymbol{\lambda}^{(n + \gamma)}$, we can conclude that for $1/2 < \gamma \leq 1$ the vectors 
$\boldsymbol{v}_i^{(n)}$ and $\boldsymbol{\lambda}^{(n)}$ are bounded $\forall i$ and $\forall n$. 
As shown mathematically in Remark \ref{Remark:DD_Remark_on_gamma_leq_dot5}, the quantities 
$\boldsymbol{v}_i^{(n)}$ and $\boldsymbol{\lambda}^{(n)}$ may not be bounded when $0 \leq \gamma 
\leq 1/2$. In a later subsection we will present physical systems that, in fact, exhibit this kind 
of (numerical) unbounded behavior for $\gamma \leq 1/2$ under the $\boldsymbol{d}$-continuity method. 
This completes the stability analysis of the $\boldsymbol{d}$-continuity method. 

\subsection{Stability of modified $\boldsymbol{d}$-continuity method}
\label{Subsec:DD_stability_modified_d_continuity}
The majority of the proof for this method is identical to the initial part of the proof for 
the $\boldsymbol{d}$-continuity method (presented in subsection \ref{Subsec:DD_d_continuity}) 
up to the step where we deduced that the vectors $\boldsymbol{\lambda}^{(n+ \gamma)}$ and 
$\boldsymbol{v}_i^{(n+\gamma)}$ are bounded. The only additional thing we have to prove is 
the boundedness of the vectors $\boldsymbol{\lambda}^{(n)}$ and $\boldsymbol{v}_i^{(n)}$, 
which is a direct consequence of the triangle inequality. To wit, using equation 
\eqref{Eqn:DD_modified_lambda_interpolation} (and also see Figure 
\ref{Fig:DD_interpolation_lambda}) we have 
\begin{align}
  \label{Eqn:DD_triangle_inequality}
  \left\|\boldsymbol{\lambda}^{(n)}\right\| = 
  \left\|\gamma \boldsymbol{\lambda}^{(n-1+\gamma)} + 
    (1 - \gamma) \boldsymbol{\lambda}^{(n + \gamma)}\right\| 
  \leq \gamma \left\|\boldsymbol{\lambda}^{(n-1+\gamma)}\right\| + 
  (1 - \gamma) \left\|\boldsymbol{\lambda}^{(n+\gamma)}\right\| 
\end{align}
Since the vectors $\boldsymbol{\lambda}^{(n+\gamma)}$ are bounded $\forall n$, from the above equation 
it is evident that the Lagrange multipliers at integral time levels $\boldsymbol{\lambda}^{(n)}$ are 
bounded $\forall n$. Using a similar reasoning, since the vectors $\boldsymbol{v}_i^{(n+\gamma)}$ are 
bounded $\forall n$, the rates at the integer time levels $\boldsymbol{v}_i^{(n)}$ are also bounded. 
(Note that the vectors $\boldsymbol{d}_i^{(n)}$ are also bounded, which has been proven in subsection 
\ref{Subsec:DD_d_continuity}.) This means that under this method all quantities of interest are 
bounded. This completes the stability analysis of the modified $\boldsymbol{d}$-continuity method. 

Note that, as discussed in Remark \ref{Remark:Trapezoidal_forward_Euler_ill_posed}, the forward 
Euler $(\gamma = 0)$ is not well-posed under the modified $\boldsymbol{d}$-continuity method. All 
other time integrators from the generalized trapezoidal family $(0 < \gamma \leq 1)$ can be employed 
and are stable under the modified $\boldsymbol{d}$-continuity domain decomposition method. 

\begin{remark} 
  \label{Remark:DD_backward_Euler_d_continuity}
  For the $\boldsymbol{d}$-continuity and modified $\boldsymbol{d}$-continuity domain 
  decomposition methods, with backward Euler $(\gamma = 1)$ one can achieve the continuity 
  of both temperatures and temperature rates along the subdomain interface. To wit, the 
  $\boldsymbol{d}$-continuity constraints (which are basically the continuity of temperatures 
  along the interface, and are given by equation \eqref{Eqn:Trapezoidal_d_continuity_DD_kinematic} 
  or \eqref{Eqn:Trapezoidal_modified_d_continuity_DD_kinematic}) imply 
  \begin{align}
    \sum_{i = 1}^{S} \boldsymbol{C}_i \left[\boldsymbol{d}_i^{(n-1)}\right] = \boldsymbol{0}, \; \forall n
  \end{align}
  For backward Euler we have $\left[\boldsymbol{d}_i^{(n-1)}\right] = \Delta t \boldsymbol{v}_i^{(n)}$. 
  This implies the continuity of rates along the subdomain interface.
  \begin{align}
    \sum_{i = 1}^{S} \boldsymbol{C}_i \boldsymbol{v}_i^{(n)} = \boldsymbol{0}, \; \forall n
  \end{align}
\end{remark}

\subsection{Stability of the $\boldsymbol{v}$-continuity method}
\label{Subsec:DD_stability_v_continuity}
We follow a similar procedure employed in subsection \ref{Subsec:DD_d_continuity}. Using equations 
\eqref{Eqn:Trapezoidal_d_continuity_DD_domain}-\eqref{Eqn:Trapezoidal_d_continuity_DD_trapezoidal} 
(and as usual neglecting the external forcing function for stability analysis \cite{Hughes}; i.e., 
$\boldsymbol{f}_i^{(n)} = \boldsymbol{0}$) one obtains 
\begin{align}
  \sum_{i = 1}^{S} \frac{1}{2} \left[{\boldsymbol{v}^{(n)}_i}^{\mathrm{T}}
    \boldsymbol{A}_i \boldsymbol{v}^{(n)}_i \right] 
  + \sum_{i=1}^{S} \Delta t {\left\{\boldsymbol{v}_{i}^{(n)}\right\}}^\mathrm{T} 
  \boldsymbol{K}_i \left\{\boldsymbol{v}^{(n)}_i\right\}
  = \sum_{i=1}^{S} {\left\{\boldsymbol{v}_i^{(n)}\right\}}^{\mathrm{T}} 
  \boldsymbol{C}_i^{\mathrm{T}} \left[\boldsymbol{\lambda}^{(n)}\right]
\end{align}
Since  the matrices $\boldsymbol{K}_i \; (i = 1, \cdots , S)$ are positive 
semidefinite, and $\Delta t > 0$; we conclude 
\begin{align}
  \sum_{i = 1}^{S} \frac{1}{2} \left[{\boldsymbol{v}^{(n)}_i}^{\mathrm{T}}
    \boldsymbol{A}_i \boldsymbol{v}^{(n)}_i \right] 
  \leq \sum_{i = 1}^{S} {\left\{\boldsymbol{v}^{(n)}_i \right\}}^{\mathrm{T}} \boldsymbol{C}_i^{\mathrm{T}} 
  \left[\boldsymbol{\lambda}^{(n)} \right] 
  = {\left[\boldsymbol{\lambda}^{(n)}\right]}^{\mathrm{T}} \sum_{i=1}^{S} 
  \boldsymbol{C}_i \left\{\boldsymbol{v}^{(n)}_i\right\} 
\end{align}
Using the $\boldsymbol{v}$-continuity given by equation 
\eqref{Eqn:Trapezoidal_v_continuity_DD_kinematic} we conclude that 
\begin{align}
 \sum_{i = 1}^{S} \frac{1}{2} \left[{\boldsymbol{v}^{(n)}_i}^{\mathrm{T}}
   \boldsymbol{A}_i \boldsymbol{v}^{(n)}_i \right] \leq 0 \quad \forall n
\end{align}
Using the definition and linearity of the jump operator, we conclude 
\begin{align}
  \sum_{i=1}^{S}{\boldsymbol{v}_i^{(n+1)}}^{\mathrm{T}}\boldsymbol{A}_{i} \boldsymbol{v}_i^{(n+1)} \leq 
  \sum_{i=1}^{S} {\boldsymbol{v}_i^{(n)}}^{\mathrm{T}}\boldsymbol{A}_{i} \boldsymbol{v}_i^{(n)} 
  \leq \cdots \leq 
  \sum_{i=1}^{S} {\boldsymbol{v}_i^{(0)}}^{\mathrm{T}}\boldsymbol{A}_{i} \boldsymbol{v}_i^{(0)} 
\end{align}
Since $\forall i = 1, \cdots, S$ the initial vectors $\boldsymbol{v}_i^{(0)}$ are 
bounded and the matrices $\boldsymbol{A}_i$ are positive definite; we conclude that 
$\boldsymbol{v}_i^{(n)}$ are bounded $\forall n$. This implies that, from the trapezoidal 
equation \eqref{Eqn:Trapezoidal_v_continuity_DD_trapezoidal}, $\left[\boldsymbol{d}_i^{(n)}
\right]$ is bounded. 

One can further show that, under the $\boldsymbol{v}$-continuity method, the jump in the 
Lagrange multipliers $\left[\boldsymbol{\lambda}^{(n)} \right]$ is also bounded. There are 
several ways to prove this result, which was the case even under the $\boldsymbol{d}$-continuity 
method for proving the boundedness of the Lagrange multipliers $\boldsymbol{\lambda}^{(n+\gamma)}$ 
(see subsection \ref{Subsec:DD_d_continuity}). As mentioned earlier, herein we take a slightly 
different approach (than the one in subsection \ref{Subsec:DD_d_continuity}), and which is also 
applicable for the Baumgarte stabilized method for showing a similar result. 
To this end, we start with 
\begin{align}
  \label{Eqn:DD_v_continuity_step4}
  \boldsymbol{M}_i \left[\boldsymbol{v}_i^{(n)}\right] + \boldsymbol{K}_i \left[\boldsymbol{d}_i^{(n)}\right] = 
  \boldsymbol{C}_i^{\mathrm{T}} \left[\boldsymbol{\lambda}^{(n)}\right], \quad \forall i
\end{align}
(Note that in the above equation we have used the fact that the external force on all 
subdomains is zero, $\boldsymbol{f}_i^{(n)} = \boldsymbol{0}$.) By premultiplying the 
above equation with $\boldsymbol{C}_i \tilde{\boldsymbol{M}}_i^{-1}$, and then summing 
over the number of subdomains, we get 
\begin{align}
  \label{Eqn:DD_v_continuity_step5}
  \left[\boldsymbol{\lambda}^{(n)}\right] = \boldsymbol{G}^{-1} 
  \sum_{i}^{S} \boldsymbol{C}_i \tilde{\boldsymbol{M}}_i^{-1} 
  \left(\boldsymbol{M}_i \left[\boldsymbol{v}_i^{(n)}\right] + 
    \boldsymbol{K}_i \left[\boldsymbol{d}_i^{(n)}\right] \right)
\end{align}
where the matrices $\tilde{\boldsymbol{M}}_i$ and $\boldsymbol{G}$ are defined in 
equations \eqref{Eqn:DD_d_continuity_M_i_tilde} and \eqref{Eqn:DD_d_continuity_G}, 
respectively. Since the vectors $\left[\boldsymbol{d}_i^{(n)}\right]$ and 
$\boldsymbol{v}_i^{(n)}$ (and hence $\left[\boldsymbol{v}_i^{(n)}\right]$) are bounded 
$\forall n$, from equation \eqref{Eqn:DD_v_continuity_step5}, we conclude that 
$\left[\boldsymbol{\lambda}^{(n)}\right]$ is bounded $\forall n$. 
\begin{remark}
  In order to show the above result (that the vector $\left[\boldsymbol{\lambda}^{(n)} \right]$ 
  is bounded), in the step just above equation \eqref{Eqn:DD_v_continuity_step5}, one can 
  premultiply with any matrix of the form $\boldsymbol{C}_i \boldsymbol{H}_i$ where 
  $\boldsymbol{H}_i$ is some positive definite matrix. 
  Since $\tilde{\boldsymbol{M}}_i^{-1}$ is a positive definite matrix (as discussed in 
  Remark \ref{Remark:DD_d_continuity_M_i_tilde_G_positive_definite} that the matrix 
  $\tilde{\boldsymbol{M}}_i$ is positive definite), the choice $\boldsymbol{H}_i = 
  \tilde{\boldsymbol{M}}_i^{-1}$ is valid. We have chosen this particular choice as 
  to be able to use some of the earlier results (e.g., the matrix $\boldsymbol{G}$ 
  is invertible), and to avoid introducing additional notation.
\end{remark}

For the $\boldsymbol{v}$-continuity domain decomposition method, solely based on the energy 
method, one cannot infer the boundedness of the quantities $\boldsymbol{d}_i^{(n)}$ and 
$\boldsymbol{\lambda}^{(n)}$. Also, there can be drift in the continuity of temperatures 
along the subdomain interface. That is, $\sum_{i = 1}^{S} \boldsymbol{C}_i \boldsymbol{d}_i^{(n)} 
\neq \boldsymbol{0}$. This drift may grow over time due to round-off errors. As discussed 
earlier, one of the ways to control the drift is to employ the Baumgarte stabilization. In 
the next subsection, we discuss the stability of the Baumgarte constraint stabilization in 
the context of index-2 linear first-order transient systems.  
\begin{remark}
  Note that, unlike the $\boldsymbol{d}$-continuity and modified $\boldsymbol{d}$-continuity 
  methods, one can employ the forward Euler method $(\gamma = 0)$ in individual subdomains 
  under the $\boldsymbol{v}$-continuity domain decomposition method. 
\end{remark}

\subsection{Stability of Baumgarte stabilized domain decomposition method}
\label{Subsec:Trapezoidal_Baumgarte_theory}
We start the stability analysis by rewriting the kinematic constraints. To this end, 
equation \eqref{Eqn:Trapezoidal_Baumgarte_kinematic} implies that 
\begin{align}
  \sum_{i=1}^{S} \boldsymbol{C}_i \left(\left[\boldsymbol{v}_i^{(n)}\right] + 
    \frac{\alpha}{\Delta t} \left[\boldsymbol{d}_i^{(n)}\right] \right) = \boldsymbol{0}
\end{align}
By using the kinematic constraints given by equation \eqref{Eqn:Trapezoidal_Baumgarte_kinematic} 
and identity \eqref{Eqn:Trapezoidal_x_gamma} we get 
\begin{align}
  \label{Eqn:Trapezoidal_final_Baumgarte_constraints}
  \sum_{i=1}^{S} \boldsymbol{C}_i \left(\alpha^{*} \left[\boldsymbol{v}_i^{(n)}\right] + 
    \alpha \left\{\boldsymbol{v}_i^{(n)}\right\}\right) = \boldsymbol{0}
\end{align}
where the (dimensionless) parameter $\alpha^{*}$ is introduced for convenience, 
and is defined as 
\begin{align}
  \label{Eqn:Trapezoidal_alpha_star}
  \alpha^{*} := 1 + \alpha \left(\gamma - \frac{1}{2}\right)
\end{align}

We now rewrite the subdomain equation in a manner similar to what we did above for the 
kinematic constraints. Equation \eqref{Eqn:Trapezoidal_Baumgarte_domain} implies that
\begin{align}
  \label{Eqn:Trapezoidal_equilibrium_jump_form}
  \boldsymbol{M}_i \left[\boldsymbol{v}_i^{(n)}\right] + \boldsymbol{K}_i \left[\boldsymbol{d}_i^{(n)}\right] = 
  \boldsymbol{C}_i^{\mathrm{T}} \left[\boldsymbol{\lambda}^{(n)}\right]
\end{align}
By using the trapezoidal equation \eqref{Eqn:Trapezoidal_Baumgarte_trapezoidal} and employing 
the identity \eqref{Eqn:Trapezoidal_x_gamma} we get 
\begin{align}
  \boldsymbol{M}_i \left[\boldsymbol{v}_i^{(n)}\right] + \Delta t \boldsymbol{K}_i 
  \left( \left(\gamma - \frac{1}{2}\right) \left[\boldsymbol{v}_i^{(n)}\right] + 
    \left\{\boldsymbol{v}_i^{(n)}\right\}\right) = \boldsymbol{C}_i^{\mathrm{T}} 
  \left[\boldsymbol{\lambda}^{(n)}\right] 
\end{align}

Premultiplying both sides of the above equation by $\left(\alpha^{*} \left[\boldsymbol{v}_i^{(n)}
  \right] + \alpha \left\{\boldsymbol{v}_i^{(n)}\right\}\right)$ and then summing over all of the 
subdomains we get 
\begin{align}
  &\sum_{i=1}^{S} \left(\alpha^{*} \left[\boldsymbol{v}_i^{(n)}\right] + \alpha 
    \left\{\boldsymbol{v}_i^{(n)}\right\}\right)^{\mathrm{T}} \left(\boldsymbol{M}_i 
  \left[\boldsymbol{v}_i^{(n)}\right] + \Delta t \boldsymbol{K}_i 
  \left(\left(\gamma - 1/2\right) \left[\boldsymbol{v}_i^{(n)}\right] + 
    \left\{\boldsymbol{v}_i^{(n)}\right\}\right)\right)\notag \\
  &= \sum_{i=1}^{S} \left(\alpha^{*} \left[\boldsymbol{v}_i^{(n)}\right] + \alpha 
    \left\{\boldsymbol{v}_i^{(n)}\right\}\right)^{\mathrm{T}} \boldsymbol{C}_i^{\mathrm{T}} 
  \left[\boldsymbol{\lambda}^{(n)}\right] 
  = \left[\boldsymbol{\lambda}^{(n)}\right]^{\mathrm{T}} \sum_{i=1}^{S} \boldsymbol{C}_i 
  \left(\alpha^{*} \left[\boldsymbol{v}_i^{(n)}\right] + \alpha \left\{\boldsymbol{v}_i^{(n)}\right\}\right)    
\end{align}
For the Baumgarte stabilized domain decomposition method, using equation 
\eqref{Eqn:Trapezoidal_final_Baumgarte_constraints}, we conclude that 
\begin{align}
  \sum_{i=1}^{S} \left(\alpha^{*} \left[\boldsymbol{v}_i^{(n)}\right] + \alpha 
    \left\{\boldsymbol{v}_i^{(n)}\right\}\right)^{\mathrm{T}} \left(\boldsymbol{M}_i 
  \left[\boldsymbol{v}_i^{(n)}\right] + \Delta t \boldsymbol{K}_i 
  \left(\left(\gamma - 1/2\right) \left[\boldsymbol{v}_i^{(n)}\right] + 
    \left\{\boldsymbol{v}_i^{(n)}\right\}\right)\right) = 0
\end{align}
Using the definition of $\alpha^{*}$ (equation \eqref{Eqn:Trapezoidal_alpha_star}), 
the above equation can be rewritten as 
\begin{align}
  \sum_{i=1}^{S} \left(\alpha^{*} \left[\boldsymbol{v}_i^{(n)}\right] + \alpha 
    \left\{\boldsymbol{v}_i^{(n)}\right\}\right)^{\mathrm{T}} \boldsymbol{M}_i 
  \left[\boldsymbol{v}_i^{(n)}\right] 
  + \Delta t \sum_{i=1}^{S} \left[\boldsymbol{v}_i^{(n)}\right]^{\mathrm{T}} 
  \boldsymbol{K}_i \left(\left(\gamma - 1/2\right) \left[\boldsymbol{v}_i^{(n)}\right] + 
      \left\{\boldsymbol{v}_i^{(n)}\right\}\right) \notag \\
  + \alpha \Delta t \sum_{i=1}^{S} \left((\gamma - 1/2) \left[\boldsymbol{v}_i^{(n)}\right] + 
    \left\{\boldsymbol{v}_i^{(n)}\right\}\right)^{\mathrm{T}} \boldsymbol{K}_i 
    \left(\left(\gamma - 1/2\right) \left[\boldsymbol{v}_i^{(n)}\right] + 
      \left\{\boldsymbol{v}_i^{(n)}\right\}\right) = 0
\end{align}
Since $\alpha > 0$, $\Delta t > 0$ and $\boldsymbol{K}_i$ is positive semidefinite, we conclude 
\begin{align}
  \sum_{i=1}^{S} \left(\alpha^{*} \left[\boldsymbol{v}_i^{(n)}\right] + \alpha 
    \left\{\boldsymbol{v}_i^{(n)}\right\}\right)^{\mathrm{T}} \boldsymbol{M}_i 
  \left[\boldsymbol{v}_i^{(n)}\right] 
  + \Delta t \sum_{i=1}^{S} \left[\boldsymbol{v}_i^{(n)}\right]^{\mathrm{T}} 
  \boldsymbol{K}_i \left(\left(\gamma - 1/2\right) \left[\boldsymbol{v}_i^{(n)}\right] + 
      \left\{\boldsymbol{v}_i^{(n)}\right\}\right) \leq 0
\end{align}
By invoking the symmetry of $\boldsymbol{M}_i$, using equation \eqref{Eqn:Trapezoidal_alpha_star}, 
and rearranging the terms we get 
\begin{align}
  \label{Eqn:New_DD_Baum_stab_step5}
  \sum_{i=1}^{S} \left[\boldsymbol{v}_i^{(n)}\right]^{\mathrm{T}} 
  \left(\alpha \boldsymbol{M}_i + \Delta t \boldsymbol{K}_i\right) 
  \left\{\boldsymbol{v}_i^{(n)}\right\} 
  + \sum_{i=1}^{S} \left[\boldsymbol{v}_i^{(n)}\right]^{\mathrm{T}} 
  \boldsymbol{\tilde{A}}_i \left[\boldsymbol{v}_i^{(n)}\right] \leq 0
\end{align}
where the matrix $\boldsymbol{\tilde{A}}_i$ is defined as 
\begin{align}
  \label{Eqn:New_DD_definition_of_A_i_tilde}
  \boldsymbol{\tilde{A}}_i := \alpha^{*} \boldsymbol{M}_i + \Delta t 
  \left(\gamma - \frac{1}{2}\right) \boldsymbol{K}_i
\end{align}

If the matrices $\tilde{\boldsymbol{A}}_i \; (i = 1, \cdots, S)$ are positive semidefinite 
(later we will obtain sufficient conditions for the matrix $\tilde{\boldsymbol{A}}_i$ to be 
positive semidefinite) then equation \eqref{Eqn:New_DD_Baum_stab_step5} implies that 
\begin{align}
  \sum_{i=1}^S \left[\boldsymbol{v}_i^{(n)}\right]^{\mathrm{T}} 
  \left(\alpha \boldsymbol{M}_i + \Delta t \boldsymbol{K}_i\right) 
  \left\{\boldsymbol{v}_i^{(n)}\right\} \leq 0
\end{align}
By invoking identity \eqref{Eqn:Trapezoidal_Symmetric_S} we have 
\begin{align}
  &\sum_{i=1}^S \left[\left(\boldsymbol{v}_i^{(n)}\right)^{\mathrm{T}} 
    \left(\alpha \boldsymbol{M}_i + \Delta t \boldsymbol{K}_i\right) 
    \boldsymbol{v}_i^{(n)}\right] \leq 0 \quad \forall n
\end{align}
This implies that 
\begin{align}
  \sum_{i=1}^S \left(\boldsymbol{v}_i^{(n+1)}\right)^{\mathrm{T}} \left(\alpha \boldsymbol{M}_i 
    + \Delta t \boldsymbol{K}_i\right) \boldsymbol{v}_i^{(n+1)} \leq \cdots \leq
  \sum_{i=1}^S \left(\boldsymbol{v}_i^{(0)}\right)^{\mathrm{T}} 
  \left(\alpha \boldsymbol{M}_i + \Delta t \boldsymbol{K}_i\right) 
  \boldsymbol{v}_i^{(0)} 
\end{align}
Since $\alpha > 0$, $\boldsymbol{M}_i$ is positive definite, and $\boldsymbol{K}_i$ is 
positive semidefinite; the matrix $\alpha \boldsymbol{M}_i + \Delta t \boldsymbol{K}_i$ 
is positive definite. This implies that the vectors $\boldsymbol{v}_i^{(n)}$ are bounded 
$\forall n$ as the (initial) vectors $\boldsymbol{v}_i^{(0)}$ are bounded. From the 
trapezoidal equation \eqref{Eqn:Trapezoidal_Baumgarte_trapezoidal}, one can conclude 
that the vectors $\left[\boldsymbol{d}_i^{(n)}\right]$ are bounded $\forall n$. Using 
the same derivation as presented under the $\boldsymbol{v}$-continuity method, one 
can conclude that the vector $\left[\boldsymbol{\lambda}^{(n)}\right]$ is bounded 
$\forall n$.

Furthermore, one can easily show that under the Baumgarte stabilized method \emph{the 
  drift in the original constraint (i.e., continuity of temperature along the subdomain 
  interface) is also bounded (but may not be zero), which is not the case with the 
  $\boldsymbol{v}$-continuity method}. 
To show this (for convenience) let us work with $1$-norm, and the result (boundedness) 
will hold for other norms also (because of equivalence of norms in finite dimensional 
vector spaces). We start with the $1$-norm of the drift in the original constraint, and 
use equation \eqref{Eqn:Trapezoidal_Baumgarte_kinematic} and the triangle inequality:
\begin{align}
  \label{Eqn:DD_drift_bounded_step_1} 
  \|\sum_{i=1}^{S} \boldsymbol{C}_i \boldsymbol{d}_i^{(n)}\| 
  = \frac{\Delta t}{\alpha} \|\sum_{i=1}^{S} \boldsymbol{C}_i \boldsymbol{v}_i^{(n)}\| 
  \leq \frac{\Delta t}{\alpha} \sum_{i=1}^{S} \|\boldsymbol{C}_i \boldsymbol{v}_i^{(n)}\| 
\end{align}
Now using the definition of matrix-norm, and the boundedness of $\boldsymbol{v}^{(n)}_i$; 
equation \eqref{Eqn:DD_drift_bounded_step_1} can be written as 
\begin{align}
  \label{Eqn:DD_drift_bounded_step_2}
  \|\sum_{i=1}^{S} \boldsymbol{C}_i \boldsymbol{d}_i^{(n)}\| 
  \leq \frac{\Delta t}{\alpha} \sum_{i=1}^{S} \| \boldsymbol{C}_i\| \|\boldsymbol{v}_i^{n}\| 
  \leq \frac{\Delta t}{\alpha} \sum_{i=1}^{S} \|\boldsymbol{v}_i^{(n)}\| < C \quad \forall n
\end{align}
where $C > 0$ is some constant independent of $n$, and we have used the fact that the 
$1$-norm $\|\boldsymbol{C}_i\|_{1} \leq 1$. (Also recall that if $\boldsymbol{C}_i \neq 
\boldsymbol{0}$, which is the case in this paper, $\|\boldsymbol{C}_i\|_{1} = 1$.) Equation 
\eqref{Eqn:DD_drift_bounded_step_2} implies that \emph{the drift in the original constraint 
is bounded $\forall n$}. 

Similar to the $\boldsymbol{v}$-continuity method, even under the Baumgarte stabilized method, 
one cannot infer anything about the boundedness of $\boldsymbol{d}^{(n)}_i$ or Lagrange multipliers 
$\boldsymbol{\lambda}^{(n)}$ solely based on the energy method. But, for many practical problems 
one often gets bounded numerical results for these quantities. The boundedness of the drift 
in the original constraint is the only additional (but very important) feature that the 
Baumgarte stabilized method has compared to the $\boldsymbol{v}$-continuity method. 

We will now obtain sufficient conditions for the matrix $\tilde{\boldsymbol{A}}_i$ to be 
positive semidefinite (which we have assumed in the stability analysis of this method, see 
the line below equation \eqref{Eqn:New_DD_definition_of_A_i_tilde}). 
Using the eigenvectors of the generalized eigenvalue problem \eqref{Eqn:Trapezoidal_GEVP} as 
the basis, the matrix $\tilde{\boldsymbol{A}}_i$ can be diagonalized and obeys the similarity 
\begin{align}
  \label{Eqn:Trapezoidal_similar_matrix}
  \tilde{\boldsymbol{A}}_i \sim \alpha^{*} \boldsymbol{I} + \Delta t 
  \left(\gamma - \frac{1}{2}\right) \boldsymbol{\Omega}_i
\end{align}
In the above equation, the symbol ``$\sim$'' denotes similarity of matrices, and the matrix 
$\boldsymbol{\Omega}_i$ is defined as
\begin{align}
  \boldsymbol{\Omega}_i = \mathrm{diag}\left(\left(\omega_i\right)_1,\cdots, 
    \left(\omega_i\right)_n\right)
\end{align}
where $n$ is the size of the (square) matrix $\tilde{\boldsymbol{A}}_i$; and $\left(\omega_i 
\right)_1,\cdots, \left(\omega_i\right)_n$ are the eigenvalues of the generalized eigenvalue 
problem \eqref{Eqn:Trapezoidal_GEVP}. (It is well-known that similar matrices have the same 
characteristic polynomial, and hence the same eigenvalues \cite{Halmos}.)

Since $\boldsymbol{M}_i$ is positive definite and $\boldsymbol{K}_i$ is positive 
semidefinite, the eigenvalues $\left(\omega_i\right)_1,\cdots,\left(\omega_i\right)_n$ 
are all non-negative. This implies that (using equations \eqref{Eqn:Trapezoidal_alpha_star} 
and \eqref{Eqn:Trapezoidal_similar_matrix}, and noting that $\Delta t > 0$) the matrix 
$\tilde{\boldsymbol{A}}_i$ is positive semidefinite unconditionally for $\gamma \geq 1/2$. 
For $\gamma < 1/2$, sufficient conditions for the matrix $\tilde{\boldsymbol{A}}_i$ to be 
positive semidefinite are 
\begin{align}
  &\alpha^{*} \geq 0 \implies \alpha \leq \frac{1}{|\gamma - 1/2|} \\
  \label{Eqn:Trapezoidal_Baumgarte_t_crit}
  &\Delta t \leq \Delta t_i^{\mathrm{crit}} = \frac{\alpha^{*}}{|\gamma - 1/2| \omega_i^{\mathrm{max}}} = 
  \frac{2}{(1 - 2 \gamma) \omega_i^{\mathrm{max}}} - \frac{\alpha}{\omega_i^{\mathrm{max}}}
\end{align}
\begin{remark}
  Note that in equation \eqref{Eqn:Trapezoidal_Baumgarte_t_crit} $\Delta t_i^{\mathrm{crit}} 
  \geq 0$ as $\alpha^{*} \geq 0$ and $\omega_i^{\mathrm{max}} \geq 0$. 
\end{remark}

These results imply that the use of Baumgarte stabilization decreases the critical step 
for subdomain $i$ by $\alpha/{\omega_i^{\mathrm{max}}}$ relative to the unconstrained case. 
Summarizing the results, sufficient conditions for the stability of Baumgarte stabilized 
domain decomposition method are  
\begin{align}
  \label{Eqn:Trapezoidal_Baumgarte_final_result_leq_dot5}
  0 \leq \gamma < 1/2 & \quad \left\{\begin{array}{l}
      \alpha \leq \frac{1}{|\gamma - 1/2|} \\
      \Delta t_i^{\mathrm{crit}} = 
      \frac{2}{(1 - 2 \gamma) \omega_i^{\mathrm{max}}} - 
      \frac{\alpha}{\omega_i^{\mathrm{max}}} 
    \end{array} \right. \\
  \label{Eqn:Trapezoidal_Baumgarte_final_result_geq_dot5}
  1/2 \leq \gamma \leq 1 & \quad \left\{\begin{array}{l}
      \alpha < + \infty \\
      \Delta t_i^{\mathrm{crit}} = + \infty
    \end{array} \right. 
\end{align}

\begin{table}
  \caption{A summary of stability results \label{Table:DD_stability_results}}
  \begin{tabular}{ll} \hline
    Method & Stability characteristics  \\ \hline
    $\boldsymbol{d}$-continuity & Stable for $1/2 < \gamma \leq 1$. \\ 
    Modified $\boldsymbol{d}$-continuity & Stable for $0 < \gamma \leq 1$. \\ 
    $\boldsymbol{v}$-continuity & Drift on the interface. \\ 
    Baumgarte stabilized & Unconditionally stable for $1/2 \leq \gamma \leq 1$. 
    Stability \\ & conditions for $0 \leq \gamma < 1/2$ are given in Eq. 
    \eqref{Eqn:Trapezoidal_Baumgarte_final_result_leq_dot5}. \\ \hline
  \end{tabular} 
\end{table}

For a quick reference, the stability results derived in this section are 
summarized in Table \ref{Table:DD_stability_results}. In the next section 
we verify the theoretical predictions using numerical experiments. 

\section{NUMERICAL RESULTS}

\subsection{Numerical verification of stability results: Split-degree of freedom}
\label{Subsec:Trapezoidal_Numerical_Verification}
A simple problem that is commonly used to study the performance of coupling algorithms 
is the \emph{split-degree of freedom system}, which can be obtained by splitting a single 
degree of freedom into two single degree of freedom problems but subjected to kinematic 
and dynamic constraints, see Figure \ref{Fig:DD_split_degree_of_freedom}. 

\begin{figure}
  \psfrag{kA}{$k_A$}
  \psfrag{kB}{$k_B$}
  \psfrag{mA}{$m_A$}
  \psfrag{mB}{$m_B$}
  \psfrag{LA}{$+\lambda$}
  \psfrag{LB}{$-\lambda$}
  \psfrag{uA}{$u_A$}
  \psfrag{uB}{$u_B$}
  \includegraphics[scale=0.5]{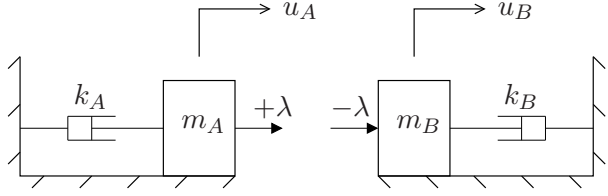}
  \caption{Problem definition: Split-degree of freedom system \label{Fig:DD_split_degree_of_freedom}}
\end{figure}

The time discrete governing equations for the subdomains $A$ and $B$ consists of  
\begin{align}
  m_A v_A^{(n)} + k_A d_A^{(n)} = +\lambda^{(n)}, \quad m_B v_B^{(n)} + k_B d_B^{(n)} = -\lambda^{(n)}
\end{align}
and along with the $\boldsymbol{d}$- or $\boldsymbol{v}$-continuity (kinematic) 
constraints, which are respectively given by
\begin{align}
  d_A^{(n)} - d_B^{(n)} = 0, \quad v_A^{(n)} - v_B^{(n)} = 0 
\end{align}
In addition, we have the equation for the generalized trapezoidal family: 
\begin{align}
  d^{(n)}_{A,B} = d^{(n-1)}_{A,B} + \Delta t \left((1 - \gamma) v^{(n-1)}_{A,B} + \gamma v^{(n)}_{A,B} \right) 
\end{align}

In this subsection, using the split-degree of freedom, we numerically verify the 
derived theoretical stability results for the $\boldsymbol{d}$-continuity and modified 
$\boldsymbol{d}$-continuity methods. The system properties are taken as $m_A = m_B = 1$, 
$k_A = 10$ and $k_B = 1$. The initial temperature is taken to be unity. The time step is 
taken as $\Delta t = 0.01 \; \mathrm{s}$. To show all the aspects of the stability results, 
we consider two different choices of the trapezoidal parameter $\gamma$. 

In the first case we chose $\gamma = 1/4$. For this choice, the critical time steps for 
(unconstrained) subdomains $A$ and $B$ are $0.4 \; \mathrm{s}$ and $4 \; \mathrm{s}$, 
respectively. The obtained numerical results using the $\boldsymbol{d}$-continuity 
method for the rate of temperature, Lagrange multiplier, and temperature are shown 
in Figures \ref{Fig:DD_Split_Degree_vn}, \ref{Fig:DD_Split_Degree_lambda} and 
\ref{Fig:DD_Split_Degree_dn}, respectively. As one can see from the figures, the 
obtained numerical results are unstable in accord with the theory presented in 
subsection \ref{Subsec:DD_d_continuity}. 

For the same case $(\gamma = 1/4)$, in Figure \ref{Fig:DD_Split_Degree_mod_vn_lambda_dot25} 
we have shown the rate of temperature and Lagrange multiplier that are obtained using the 
modified $\boldsymbol{d}$-continuity method. Both the rate of temperature and Lagrange 
multiplier are all bounded even at the integer time levels, which agrees with the theory 
presented in subsection \ref{Subsec:DD_stability_modified_d_continuity}. 

In the second case we chose $\gamma = 3/4$. This time stepping scheme is unconditionally stable 
when applied to ODEs. Even when applied to DAEs, as shown in earlier in this paper, the time 
stepping scheme should be unconditionally stable with respect to all quantities. The obtained 
numerical results for the rate of temperature and Lagrange multiplier are shown in Figure 
\ref{Fig:DD_Split_Degree_vn_lambda_gamma_dot75}, and the temperature is shown in Figure 
\ref{Fig:DD_Split_Degree_dn_gamma_dot75}. Even for this case, the obtained numerical 
results are stable, which agrees with the theoretical predictions.

\subsection{One- and two-dimensional problems, and multiple subdomains}
The model problem is transient linear heat conduction, and the governing equation can 
be written as 
\begin{align}
  \rho c_p \dot{u} - \nabla \cdot \left(k \nabla u\right) = f 
\end{align}
where $u(\boldsymbol{x},t)$ is the temperature, $\nabla$ denotes the spatial gradient operator, 
$f(\boldsymbol{x},t)$ is the volumetric source, $\boldsymbol{x}$ is the position vector (in 2D, 
$\boldsymbol{x} = (x,y)^{\mathrm{T}}$), $k$ is the coefficient of conductivity, $\rho$ is the density, and 
$c_p$ is the coefficient of heat capacity. In all our numerical simulations we have employed the 
standard Galerkin finite element formulation for the spatial discretization \cite{Hughes}. 

We consider two test problems, one each in 1D and 2D. For the 1D problem, a domain of length 
$L = 2.0 \; \mathrm{m}$ is divided into two equal subdomains, and each subdomain is divided 
into $10$ equal linear finite elements. For the 2D problem, a square of size $2 \; \mathrm{m} 
\times 2 \; \mathrm{m}$ is divided into four equal subdomains, and each subdomain is divided 
into $10 \times 10$ square elements (see Figure \ref{Fig:DD_typical_convergence_mesh}). 

For both the test problems we assume homogeneous Neumann boundary conditions, and zero volumetric 
source (i.e., $f = 0$). The analytical solutions for 1D and 2D test problems are, respectively, 
given by 
\begin{align}
  \label{Eqn:DD_1D_analytical}
  u(x,t) &= \exp\left[-\frac{k}{\rho c_p} \frac{\pi^2}{4} t\right] \cos\left(\frac{\pi x}{2}\right) \\
  \label{Eqn:DD_2D_analytical}
  u(x,y,t) &= \exp\left[-\frac{k}{\rho c_p} \frac{\pi^2}{2} t\right] 
  \cos\left(\frac{\pi x}{2}\right) \cos\left(\frac{\pi y}{2}\right)
\end{align}
The problems are driven by prescribed non-zero initial conditions, which are obtained by 
evaluating the analytical solution (given in equations \eqref{Eqn:DD_1D_analytical}-
\eqref{Eqn:DD_2D_analytical}) at time $t = 0$. We use the standard linear and bilinear 
finite elements for 1D and 2D test problems, respectively. We assume the material 
properties to be $k = \rho c_p = 1$. 
 
In Figure \ref{Fig:DD_test_prob_1_d_continuity_gamma_dot25} we have shown the rate 
of temperature from the $\boldsymbol{d}$-continuity method using $\gamma = 1/4$. 
As predicted by the theory, even for the 1D test problem, the numerical results 
are unstable. In Figure \ref{Fig:DD_test_prob_1_d_continuity_gamma_dot75} we have 
shown the temperature and rate of temperature under the $\boldsymbol{d}$-continuity 
using $\gamma = 3/4$, and as predicted by the theory the numerical results are 
stable and match well with the analytical solution. 
In Figure \ref{Fig:DD_test_prob_1_modified_d_continuity}, we have plotted the 
temperature obtained using the modified $\boldsymbol{d}$-continuity method for 
both $\gamma = 1/4$ and $\gamma = 3/4$. For both the values of the Baumgarte 
parameter, the numerical results are stable and match well the analytical solution. 

In Figures \ref{Fig:DD_test_prob_2_d_continuity} and \ref{Fig:DD_test_prob_2_mod_d_continuity} we 
have shown the numerical results for the 2D test problem from the $\boldsymbol{d}$-continuity and 
modified $\boldsymbol{d}$-continuity methods, respectively, for $\gamma = 3/4$ and time step 
$\Delta t = 0.001 \; \mathrm{s}$. For both these methods, as expected, the obtained numerical 
results matched well with the analytical solution. 

\subsubsection{Spatial numerical $h$-convergence analysis}
From the standard approximation theory \cite{Brenner_Scott}, it is well-known that the theoretical 
rate of $h$-convergence for the above test problems (which have smooth solutions) using linear 
finite elements is $2$. For a robust converging domain decomposition method, one would prefer 
to have the numerical rate of convergence to be the same as for the undecomposed problem, which 
for the chosen test problems is $2$. The numerical spatial convergence analysis for linear transient 
heat conduction is typically done either by simultaneously decreasing the time step proportional to 
$h^2$ or by choosing a very small time step so that the error due to the time discretization is negligible. 
In our numerical simulations, we have employed the later approach, and have taken the time step to be 
$\Delta t = 0.00001 \; \mathrm{s}$. We now investigate the numerical rate of $h$-convergence for the 
$\boldsymbol{d}$-continuity and modified $\boldsymbol{d}$-continuity methods. 

The numerical convergence analysis is performed at time level $t = 0.01 \; \mathrm{s}$ 
using a hierarchy of uniform meshes. A typical 2D mesh along with subdomains is shown in 
Figure \ref{Fig:DD_typical_convergence_mesh}. In Figure \ref{Fig:DD_u_v_spatial_convergence}, 
the numerical convergence results for 1D and 2D test problems for the temperature and rate of 
temperature using the $\boldsymbol{d}$-continuity method with the trapezoidal parameter 
$\gamma = 3/4$ are shown. 
In Figure \ref{Fig:DD_modified_d_continuity_spatial_convergence}, we have shown the numerical 
convergence results for the two test problems for the temperature under the modified 
$\boldsymbol{d}$-continuity method using $\gamma = 1/4$ and $\gamma = 3/4$.  
Under both the $\boldsymbol{d}$-continuity and modified $\boldsymbol{d}$-continuity methods 
and for both the test problems, the obtained numerical rate of spatial convergence is $2$, 
which implies the methods are convergent schemes as predicted by the theory. 

\subsection{Baumgarte stabilized domain decomposition method} 
In this subsection we study the numerical performance of the Baumgarte stabilized domain 
decomposition method, and compare numerical results with the theoretical results derived 
in subsection \ref{Subsec:Trapezoidal_Baumgarte_theory}. 

Consider linear transient heat conduction in a one-dimensional bar of length $L = 2 \; \mathrm{m}$ 
with material properties to be $k = 1$ and $\rho c_p = 1$, and the initial temperature to be unity. 
At the left end of the domain the flux is zero, and on the right end we prescribe a constant 
temperature of zero. The analytical solution is given by 
\begin{align}
  u(x,t) = \frac{4}{\pi} \sum_{n=0}^{\infty} \frac{(-1)^n}{(2 n + 1)} 
  \exp\left[- \frac{(2 n + 1)^2 \pi^2 t}{4L}\right]
  \cos\left[\frac{(2n + 1) \pi x}{2L}\right]
\end{align}
The computational domain is divided into two equal subdomains, and each subdomain 
is divided into $10$ equal finite elements. The left subdomain (which has Neumann 
boundary condition on its left end) is denoted as subdomain 1, and the right subdomain 
(which has Dirichlet boundary condition on its right end) is denoted as subdomain $2$.

We consider three different cases, and for the first two cases the trapezoidal parameter 
is taken as $\gamma = 0.1$. For the chosen properties, $\omega_1^{\mathrm{max}} = 1200$ 
and $\omega_2^{\mathrm{max}} = 1178.1$, and the critical time steps for individual 
\emph{unconstrained} subdomains are $2.083 \times 10^{-3} \; \mathrm{s}$ and $2.122 
\times 10^{-3} \; \mathrm{s}$, respectively. The upper bound for the stabilization 
parameter is $2.5$ (see equation \eqref{Eqn:Trapezoidal_Baumgarte_final_result_leq_dot5}), 
and the critical time steps for constrained subdomains are 
\begin{align*}
  \Delta t_1^{\mathrm{max}} = 2.083 \times 10^{-3} \left(2.5 - 2 \alpha \right), \quad 
  \Delta t_2^{\mathrm{max}} = 2.122 \times 10^{-3} \left(2.5 - 2 \alpha \right) 
\end{align*}

In the first case, we have taken $\alpha = 1$, and the critical time steps for individual 
constrained subdomains will be $1.042 \times 10^{-3} \; \mathrm{s}$ and $1.061 \times 10^{-3} 
\; \mathrm{s}$. We have taken the time step as $\Delta t = 0.001 \; \mathrm{s}$, which satisfies 
the stability criteria \eqref{Eqn:Trapezoidal_Baumgarte_final_result_leq_dot5}. In Figure 
\ref{Fig:DD_Baumgarte_temp_stable_alpha} we have shown the numerical results for temperature 
and rate of temperature using $\alpha = 1$. The obtained numerical results are stable as 
predicted by the theory, and also there is no visible drift in the temperature or rate of 
temperature. 

In the second case, we illustrate that (as predicted by the theory) larger Baumgarte parameters 
affect the critical time step, and can make the algorithm unstable. To this end, we have taken 
$\alpha = 2.6$. We have again chosen $\Delta t = 0.001 \; \mathrm{s}$ as it satisfies critical 
time steps for individual \emph{unconstrained} subdomains (and also for $\alpha = 0$). But, 
according to the theory, the chosen time step \emph{may} be unstable for $\alpha = 2.6$. (Note 
that the energy method provides sufficient conditions for stability, and the bounds can be very 
conservative). 
In Figure \ref{Fig:DD_Baumgarte_temp_unstable_alpha} we have shown the obtained numerical 
results for temperature and its rate using the Baumgarte parameter $\alpha = 2.6$. In Figure 
\ref{Fig:DD_Baumgarte_lambda_unstable_alpha} we have plotted the interface Lagrange multiplier 
against time for both $\alpha = 1$ and $\alpha = 2.6$. As one can see, the numerical results 
using $\alpha = 1$ are stable whereas the results using $\alpha = 2.6$ are unstable. From the 
above discussion, we conclude that the Baumgarte parameter does affect the critical time step 
for $\gamma < 1/2$, which is in accord with the theory. Also, from these numerical results one 
can see that the theoretical stability bounds obtained for the Baumgarte stabilized method 
using the energy method are rather sharp. 

In the third case, we have taken $\gamma = 1/2$. For this case, the theory predicts that 
the Baumgarte stabilized domain decomposition method is unconditionally stable, and there 
is no restriction on the choice of the Baumgarte parameter $\alpha > 0$. (Of course, to 
avoid ill-conditioned systems, which may give numerical results that are unreliable, one 
may not use a huge value for $\alpha$.) We choose $\alpha = 2.6$ (which is unstable under 
$\gamma = 0.1$), and also a larger value of $\alpha = 10$. As one can see from Figure 
\ref{Fig:DD_Baumgarte_temp_gamma_dot5}, the numerical results are stable under $\gamma 
= 1/2$ even for higher values of the Baumgarte parameter, which again agrees with the 
theory. 

\section{CONCLUSION}
We have presented four variants of dual substructuring domain decomposition method for 
first-order transient systems. Two of them (the $\boldsymbol{d}$-continuity and modified 
$\boldsymbol{d}$-continuity methods) enforce the continuity of temperatures along the subdomain 
interface. The third one (the $\boldsymbol{v}$-continuity method) enforces the continuity of rate 
of temperature along the subdomain interface. The fourth method (the Baumgarte stabilized 
method) enforces the continuity of a linear combination of the temperature and rate of 
temperature along the subdomain interface. We have assessed their numerical stability and 
accuracy (in particular, drift in the constraints) under the generalized trapezoidal 
family of time integrators. 

Time stepping schemes from the trapezoidal family are primarily developed for solving ODEs, and 
the stability criteria widely used for solving first-order transient problems are (in general) 
not valid for transient domain decomposition methods (as the governing equations are DAEs instead 
of ODEs). For example, under the $\boldsymbol{d}$-continuity method (which enforces the subdomain 
equation and kinematic constraints at integer time levels) we have shown that all time stepping 
schemes with the trapezoidal parameter $0 \leq \gamma \leq 1/2$ are \emph{unconditionally unstable}. 
This is a surprising result, as in the case of ODEs the time stepping schemes with $0 \leq \gamma < 
1/2$ are stable (conditionally) and the mid-point rule with $\gamma = 1/2$ is actually 
unconditionally stable. 

The modified $\boldsymbol{d}$-continuity method enforces the subdomain equation at intermediate time levels, 
and kinematic constraints at integer time levels has been proposed. Under this domain decomposition method, 
we have shown that all the quantities of interest (temperature, rate of temperature, and interface Lagrange 
multipliers) are all bounded provided that the time step is less than the critical time steps of individual 
unconstrained subdomains. 

The stability analysis of the $\boldsymbol{v}$-continuity method based on the energy method reveals 
that the rate of temperatures at integer time levels are bounded. However, solely based on the energy 
method, one cannot infer anything about the boundedness of the temperature and interface Lagrange 
multiplier at integer time levels. One main drawback of the $\boldsymbol{v}$-continuity method is 
that one may get irrecoverable drift in the original constraint (i.e., continuity of temperature), 
which may make the numerical results unphysical. 

We have also performed stability analysis on the Baumgarte stabilization method for first-order 
systems, and derived simple bounds for the user-defined parameter that ensures stability. To our 
knowledge no prior work has derived bounds on the Baumgarte parameter for index-2 DAEs to ensure 
stability in a mathematically rigorous way. 

\section*{APPENDIX: IMPLEMENTATION DETAILS}
In this section we will outline a numerical algorithm for implementing the modified $\boldsymbol{d}$-continuity 
method. One can easily modify this numerical implementation procedure for other domain decomposition methods 
presented in Section \ref{Sec:Trapezoidal_DD_methods}. 

Using equations \eqref{Eqn:Trapezoidal_modified_d_continuity_DD_domain}-\eqref{Eqn:Trapezoidal_modified_d_continuity_DD_trapezoidal}, 
the governing equation for subdomain $i$ can be written as
\begin{align}
  \label{Eqn:Trapezoidal_Appendix_subdomain}
  \tilde{\boldsymbol{M}}_i \boldsymbol{d}_i^{(n+1)} 
  = \left(\boldsymbol{M}_i - (1 - \gamma) \Delta t \boldsymbol{K}_i\right) \boldsymbol{d}_i^{(n)} 
  + \Delta t \boldsymbol{f}_i^{(n + \gamma)} + \Delta t \boldsymbol{C}_i^{\mathrm{T}} \boldsymbol{\lambda}^{(n + \gamma)}
\end{align}
where the matrix $\tilde{\boldsymbol{M}}_i := \boldsymbol{M}_i + \gamma \Delta t \boldsymbol{K}_i$. 
It is easy to see that the matrices $\tilde{\boldsymbol{M}}_i \; (i = 1, \cdots, S)$ are symmetric 
and positive definite. By substituting equation \eqref{Eqn:Trapezoidal_Appendix_subdomain} into 
equation \eqref{Eqn:Trapezoidal_modified_d_continuity_DD_kinematic}, the interface problem can 
be written as 
\begin{align}
  \label{Eqn:Trapezoidal_appendix_interface_problem}
  \boldsymbol{G} \boldsymbol{\lambda}^{(n+\gamma)} = \boldsymbol{R} 
\end{align}
where the Schur complement operator $\boldsymbol{G}$ and interface vector $\boldsymbol{R}$ 
are defined as 
\begin{align}
  &\boldsymbol{G} := \sum_{i=1}^{S} \boldsymbol{C}_i \tilde{\boldsymbol{M}}_i^{-1} 
  \boldsymbol{C}_i^{\mathrm{T}}, \quad 
  \boldsymbol{R} := - \sum_{i=1}^{S} \boldsymbol{C}_i \tilde{\boldsymbol{M}}_i^{-1} \tilde{\boldsymbol{f}}_i^{(n + \gamma)} \\
  &\tilde{\boldsymbol{f}}_i^{(n+\gamma)} := \frac{1}{\Delta t} \left(\boldsymbol{M}_i - 
    (1-\gamma)\Delta t \boldsymbol{K}_i\right) \boldsymbol{d}_i^{(n)} + \boldsymbol{f}_i^{(n+\gamma)}
\end{align}
\begin{remark}
  It is easy to show that the Schur complement operator $\boldsymbol{G}$ is symmetric 
  and positive definite (as, in this paper, we have assumed that there are no redundant 
  constraints). Hence, the interface problem \eqref{Eqn:Trapezoidal_appendix_interface_problem} 
  is solvable and has a unique solution (as $\boldsymbol{G}$ is invertible). 
\end{remark}

Then a simple numerical algorithm for implementing the modified $\boldsymbol{d}$-continuity 
method reads 
\begin{itemize}
\item Solve the interface problem \eqref{Eqn:Trapezoidal_appendix_interface_problem} 
  to obtain the Lagrange multipliers $\boldsymbol{\lambda}^{(n+\gamma)}$. 
\item Using the obtained interface Lagrange multipliers $\boldsymbol{\lambda}^{(n+\gamma)}$, for each 
subdomain $(i = 1, \cdots, S)$ obtain $\boldsymbol{d}_i^{(n+1)}$ by solving the subdomain equation 
\eqref{Eqn:Trapezoidal_Appendix_subdomain}. 
\item Using equation $\eqref{Eqn:Trapezoidal_modified_d_continuity_DD_trapezoidal}_1$, solve for 
$\boldsymbol{v}_i^{(n+\gamma)}$. That is, 
$\boldsymbol{v}_i^{(n+\gamma)} = \frac{1}{\Delta t} \left(\boldsymbol{d}_i^{(n+1)} - \boldsymbol{d}_i^{(n)}\right)$.
\item If desired, calculate the rates $\boldsymbol{v}_i^{(n+1)}$ and Lagrange multipliers $\boldsymbol{\lambda}^{(n+1)}$ at 
integer time levels using equations \eqref{Eqn:DD_modified_v_interpolation} and \eqref{Eqn:DD_modified_lambda_interpolation}, 
respectively.
\end{itemize}
In the FETI method (see \cite{Farhat_Roux_IJNME_1991_v32_p1205}, and also \cite{Nakshatrala_Hjelmstad_Tortorelli_IJNME} 
and references therein), one employs an iterative solver (in this case, one can use the Preconditioned Conjugate Gradient 
method \cite{Golub} as the Schur complement operator $\boldsymbol{G}$ is symmetric and positive definite) for the interface 
problem, and a direct solver for solving the governing equation for individual subdomains.

\section*{ACKNOWLEDGMENTS}
The first author (KBN) would like to thank Professor Albert Valocchi and Professor Daniel 
Tortorelli for their valuable discussions and encouragement. 
The research reported herein was partly supported (KBN) by the Department of Energy through a SciDAC-2 
project (Grant No. DOE DE-FC02-07ER64323). This support is gratefully acknowledged. The opinions expressed 
in this paper are those of the authors and do not necessarily reflect that of the sponsor.

\bibliographystyle{plain}
\bibliography{Books,Master_References}

\newpage
\begin{figure}
  \centering
  \includegraphics[scale=0.4]{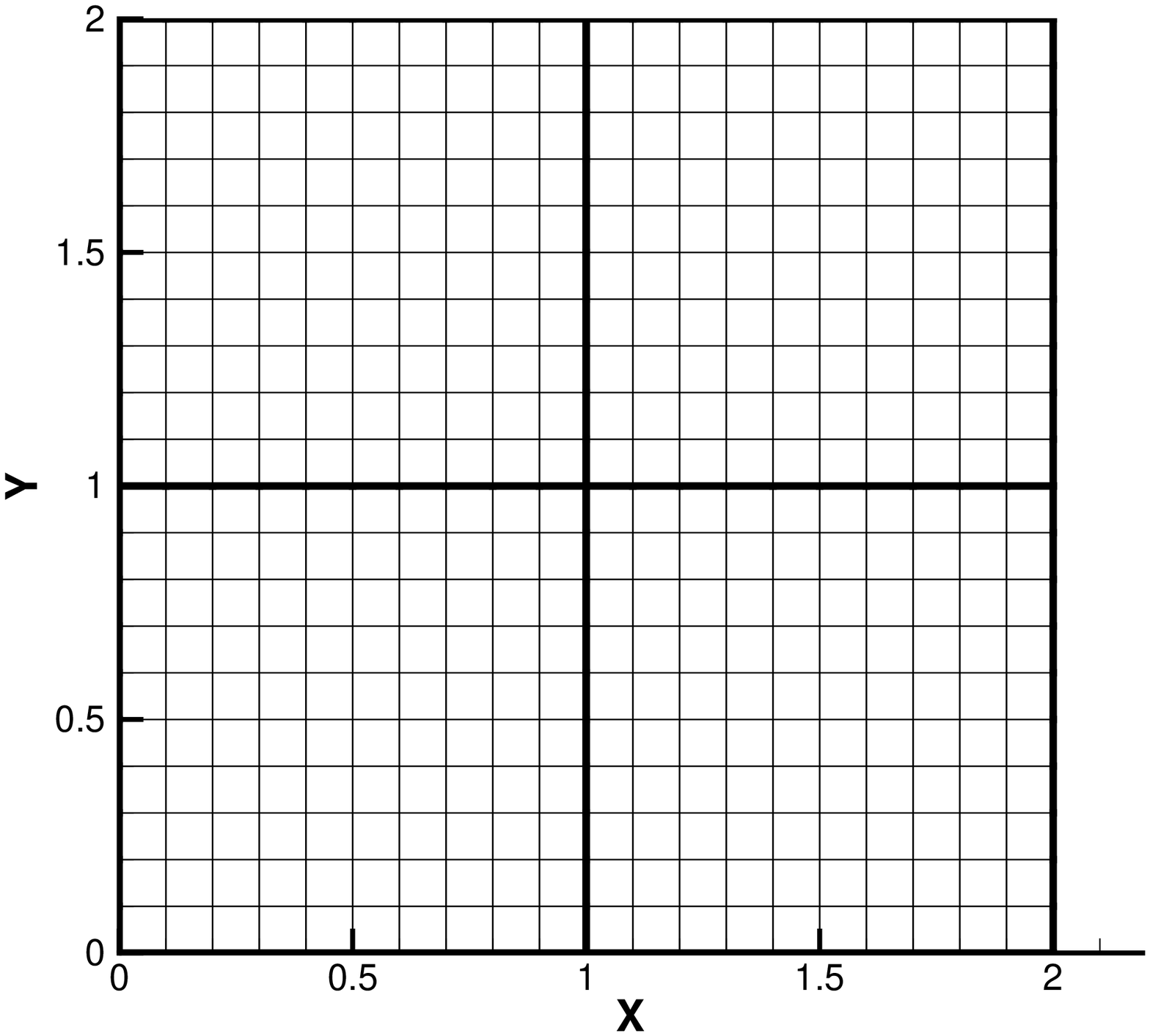}
  \caption{A typical 2D mesh used in the numerical convergence analysis. A computational domain is divided into 
    $4$ subdomains of equal size. The mesh and subdomain interface are shown in the figure. 
    \label{Fig:DD_typical_convergence_mesh}}
\end{figure}

\begin{figure}
  \centering
  \subfigure[]{
    \includegraphics[scale=0.3]{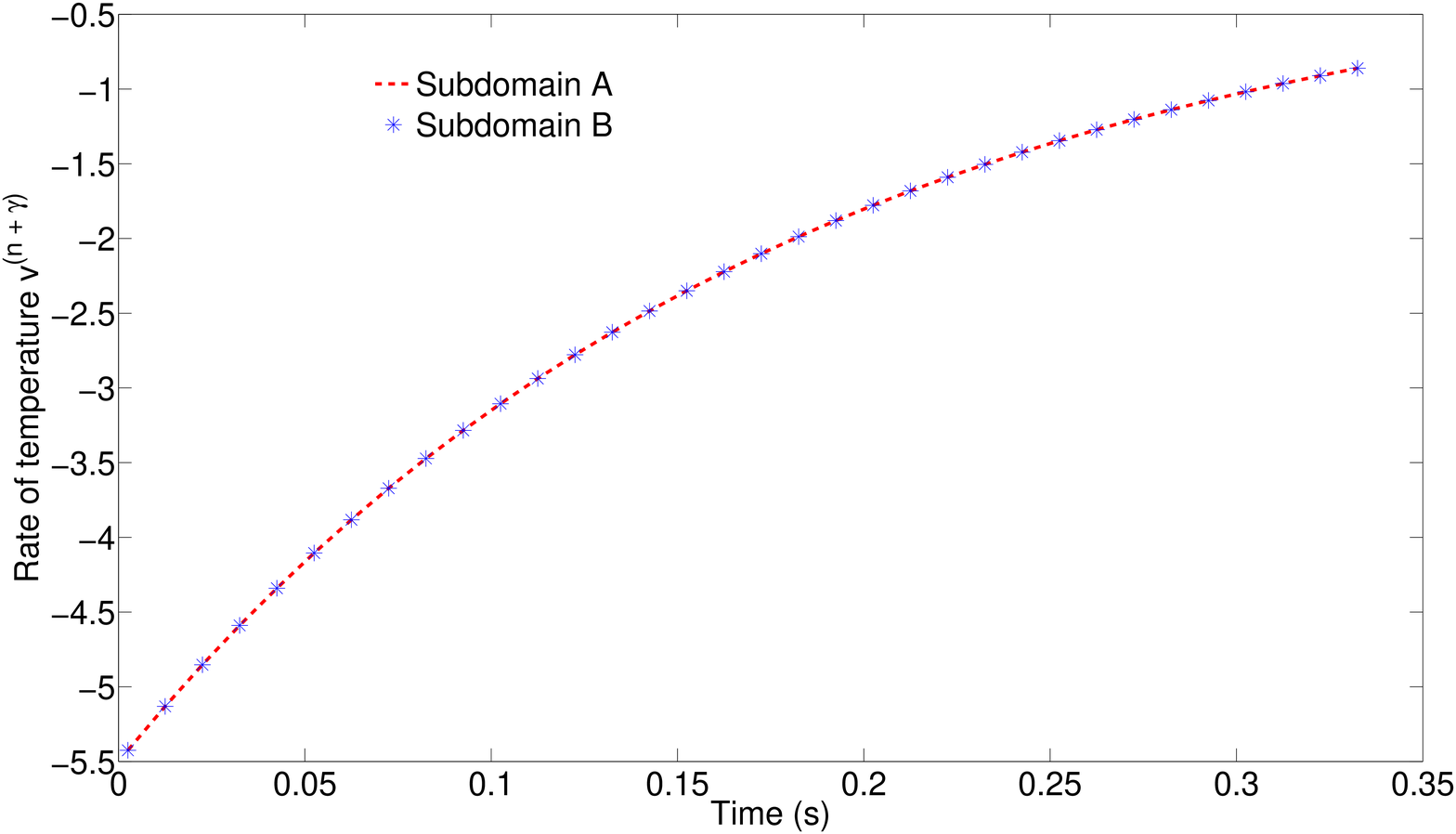}}
  \subfigure[]{
    \includegraphics[scale=0.3]{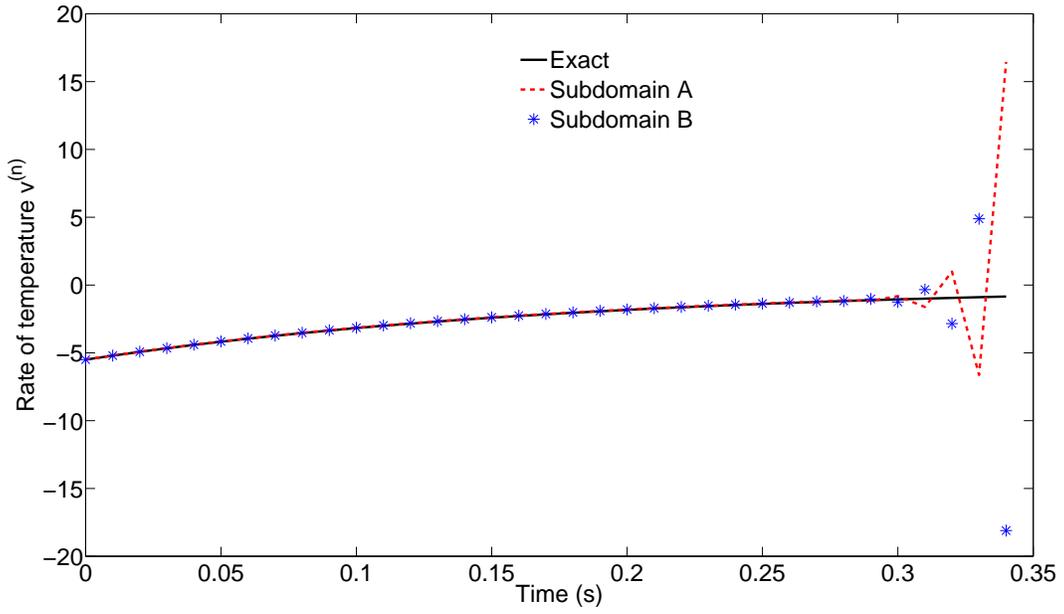}}
  \caption{Split degree of freedom system $(\gamma = 1/4)$: $\boldsymbol{d}$-continuity domain decomposition method. 
    (a) $v_A^{(n + \gamma)}$ and $v_B^{(n + \gamma)}$ are bounded. 
    (b) The quantities $v_A^{(n)}$ and $v_B^{(n)}$ are growing, which is in accord with the theory. 
    \label{Fig:DD_Split_Degree_vn}}
\end{figure}

\begin{figure}
  \centering
  \includegraphics[scale=0.3]{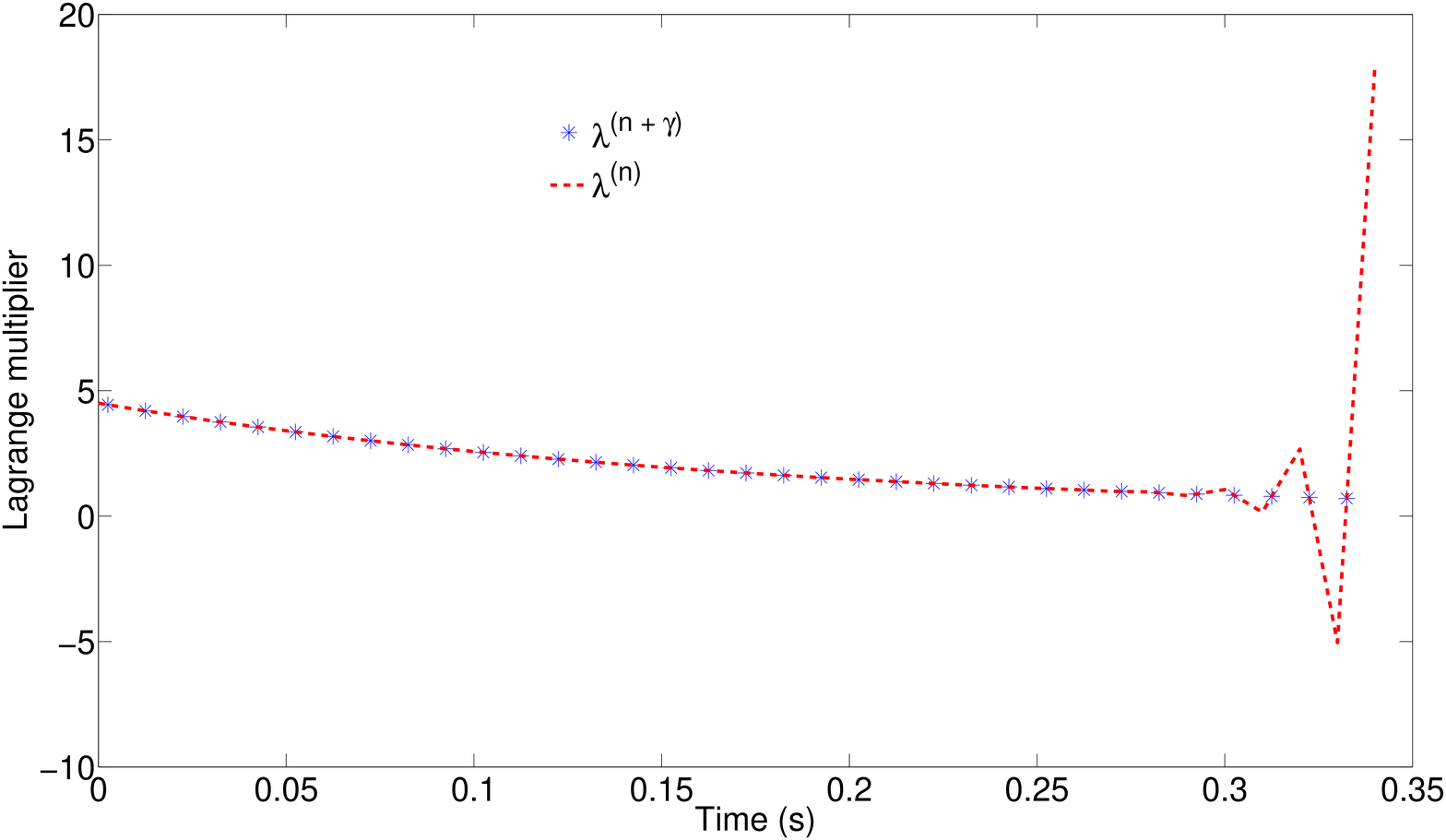}
  \caption{Split degree of freedom system $(\gamma = 1/4)$: $\boldsymbol{d}$-continuity domain decomposition method. 
    The Lagrange multiplier $\lambda^{(n + \gamma)}$ is bounded but $\lambda^{(n)}$ is growing, which 
    is similar to the counterexample presented in Figure 
    \ref{Fig:DD_Unstable_sequence_gamma_leq_dot5}. \label{Fig:DD_Split_Degree_lambda}}
\end{figure}

\begin{figure}
  \includegraphics[scale=0.3]{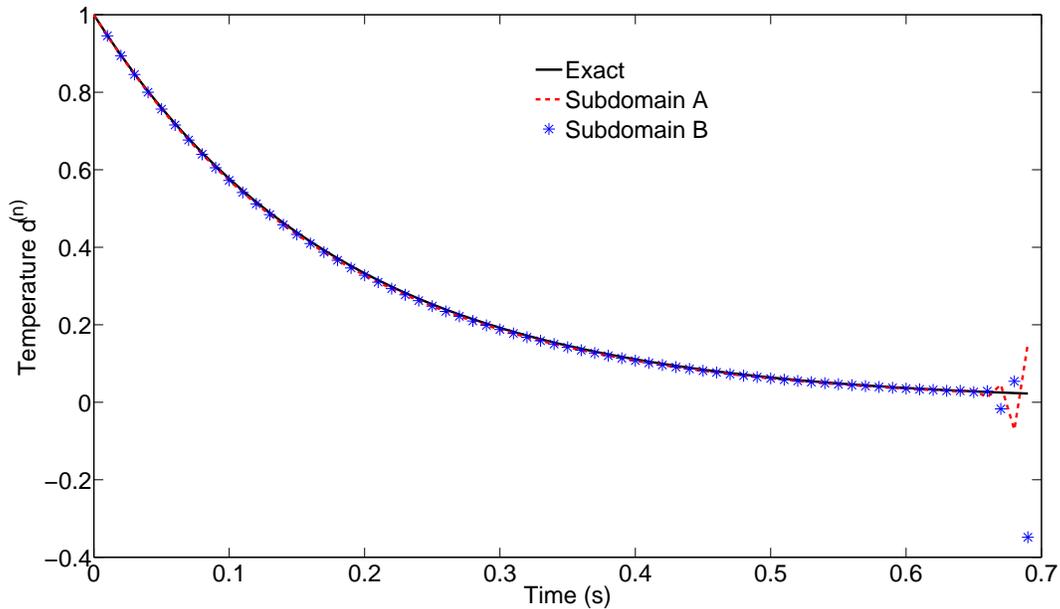}
  \caption{Split degree of freedom system $(\gamma = 1/4)$: $\boldsymbol{d}$-continuity domain decomposition 
    method. The temperatures $d_A^{(n)}$ and $d_B^{(n)}$ are bounded for a while. But due to large (and growing) 
    values of $v_A^{(n)}$, $v_B^{(n)}$ and $\lambda^{(n)}$; the whole numerical algorithm fails after sufficiently 
    long time. Theoretically the quantities $d_A^{(n)}$ and $d_B^{(n)}$ are bounded $\forall n$. But on a computer, 
    both these values also blow up after some time because of finite precision arithmetic. At $t = 0.69 \; \mathrm{s}$, 
    the obtained numerical values are $d_A^{(n)} \approx 0.1517$, $d_B^{(n)} \approx -0.3483$, and $\lambda^{(n)} \approx 8.6457 
    \times 10^{17}$. The exact values for temperature and Lagrange multipliers are, respectively, 
    $d_A^{\mathrm{exact}} = d_B^{\mathrm{exact}}\approx 0.0225$ and $\lambda^{\mathrm{exact}} = 0.1011$. Also note 
    that in the earlier figures the values are plotted up to time $t = 0.35 \; \mathrm{s}$, but in this figure, 
    in order to see appreciable differences, we went up to $t = 0.7 \; \mathrm{s}$.
    \label{Fig:DD_Split_Degree_dn}}
\end{figure}

\begin{figure}
  \centering
  \subfigure{
    \includegraphics[scale=0.3]{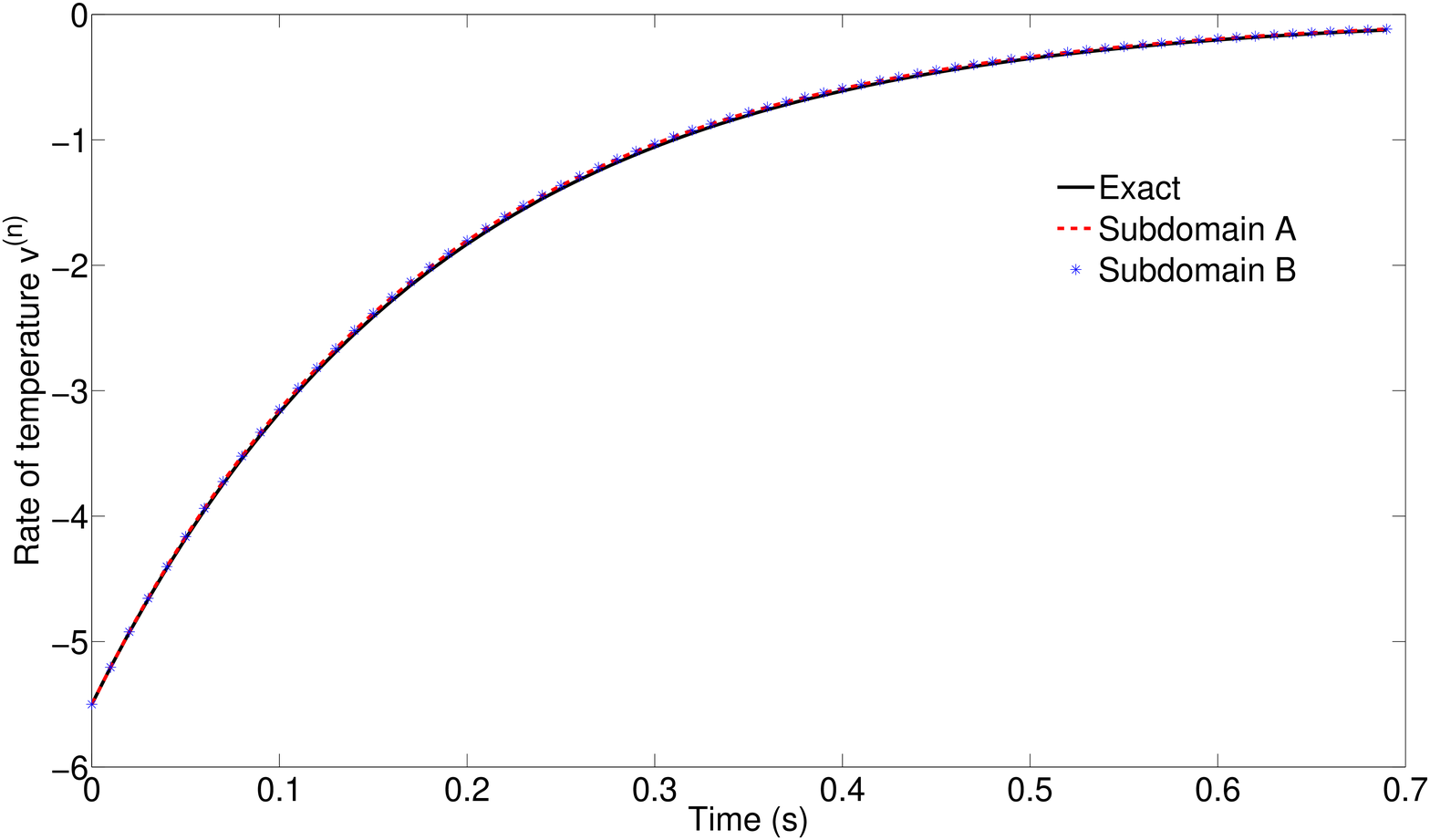}}
  \subfigure{
    \includegraphics[scale=0.3]{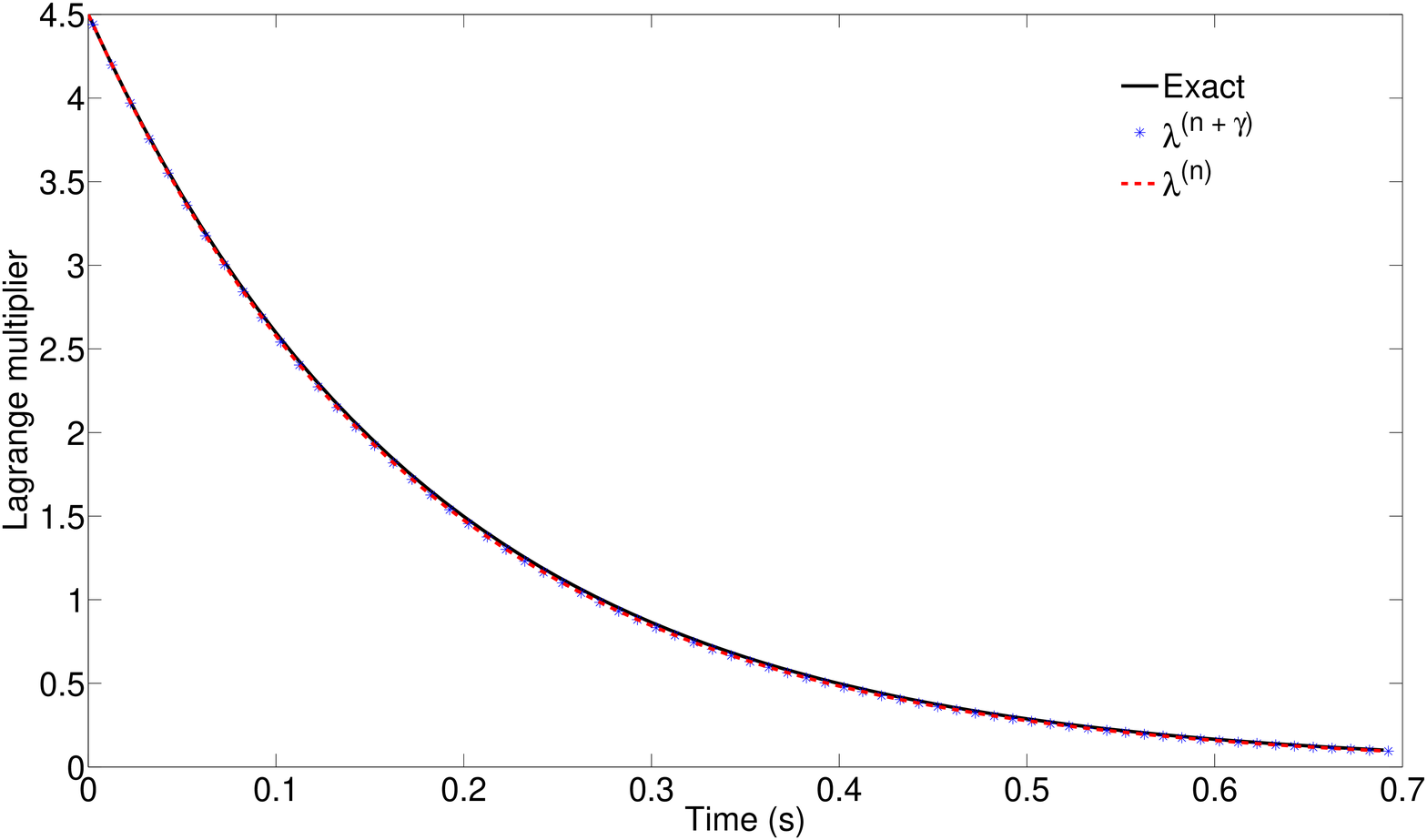}}
  \caption{Split degree of freedom system $(\gamma = 1/4)$: Modified $\boldsymbol{d}$-continuity method. 
    In this figure, the rate of temperature (top) at integral time steps for both subdomains, and the 
    interface Lagrange multiplier (bottom) at both integer and intermediate time levels are plotted. 
    As one can see, all these quantities are bounded, which agrees with the theory. Note that the rate 
    of temperature and Lagrange multiplier at integer time steps are evaluated using equations 
    \eqref{Eqn:DD_modified_v_interpolation} and \eqref{Eqn:DD_modified_lambda_interpolation}, 
    respectively. \label{Fig:DD_Split_Degree_mod_vn_lambda_dot25}}
\end{figure}

\begin{figure}
  \subfigure{
    \centering
    \includegraphics[scale=0.3]{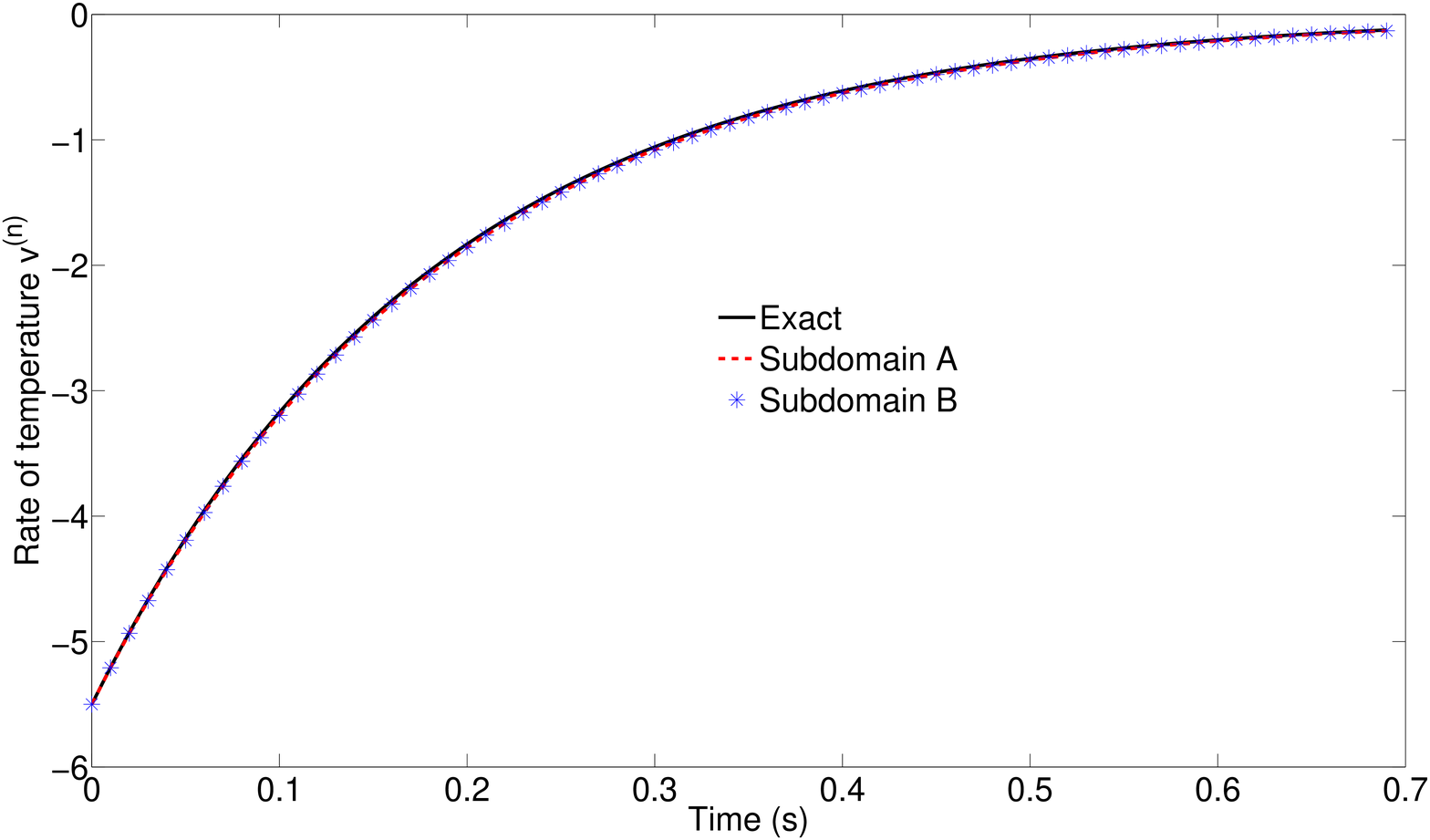}}
  \subfigure{
    \centering
    \includegraphics[scale=0.3]{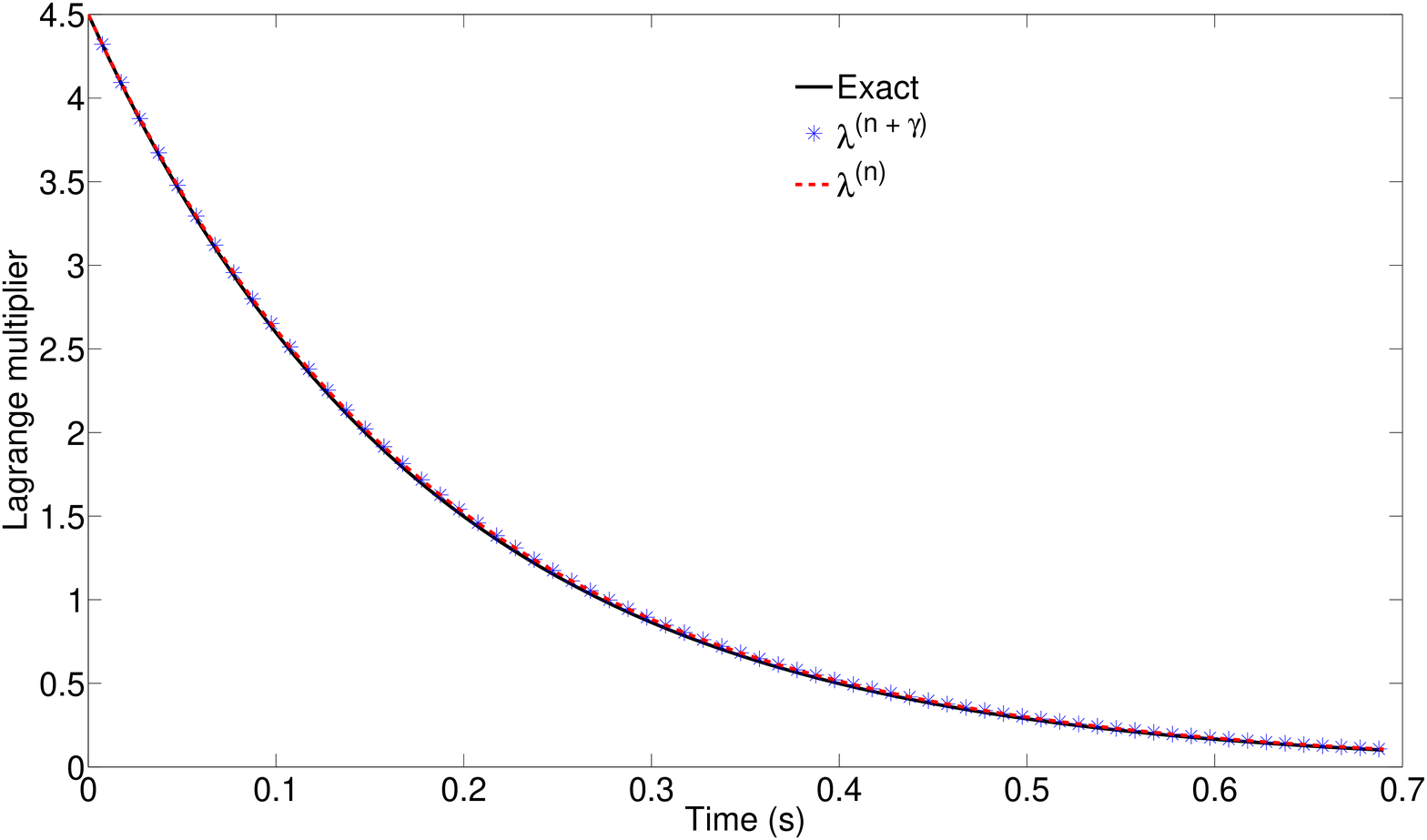}}
  \caption{Split degree of freedom system $(\gamma = 3/4)$: $\boldsymbol{d}$-continuity domain 
    decomposition method. The rate of temperature (top) at integer time levels are plotted for 
    both subdomains (i.e., $v_A^{(n)}$ and $v_B^{(n)}$). The Lagrange multiplier (bottom) at both 
    integer $\lambda^{(n)}$ and weighted $\lambda^{(n + \gamma)}$ time levels are also plotted. All 
    these quantities are bounded, and the numerical results agree with the theory. 
    \label{Fig:DD_Split_Degree_vn_lambda_gamma_dot75}}
\end{figure}

\begin{figure}
  \centering
  \includegraphics[scale=0.3]{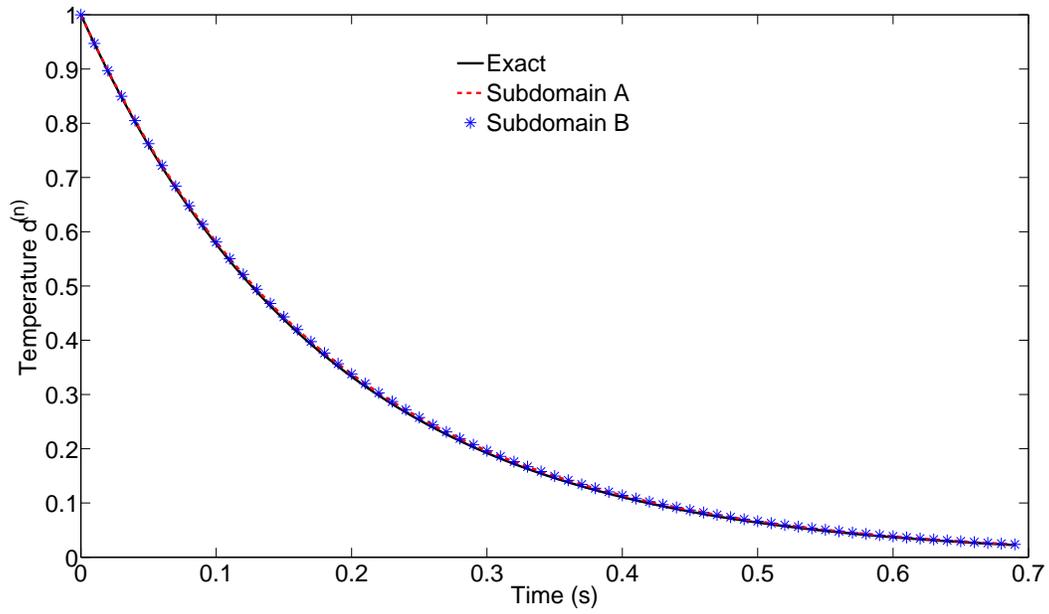}
  \caption{Split degree of freedom system $(\gamma = 3/4)$: $\boldsymbol{d}$-continuity domain decomposition method. 
    For this case, both the temperatures $d_A^{(n)}$ and $d_B^{(n)}$ are bounded. 
    \label{Fig:DD_Split_Degree_dn_gamma_dot75}}
\end{figure}

\begin{figure}
  \centering
  \includegraphics[scale=0.3]{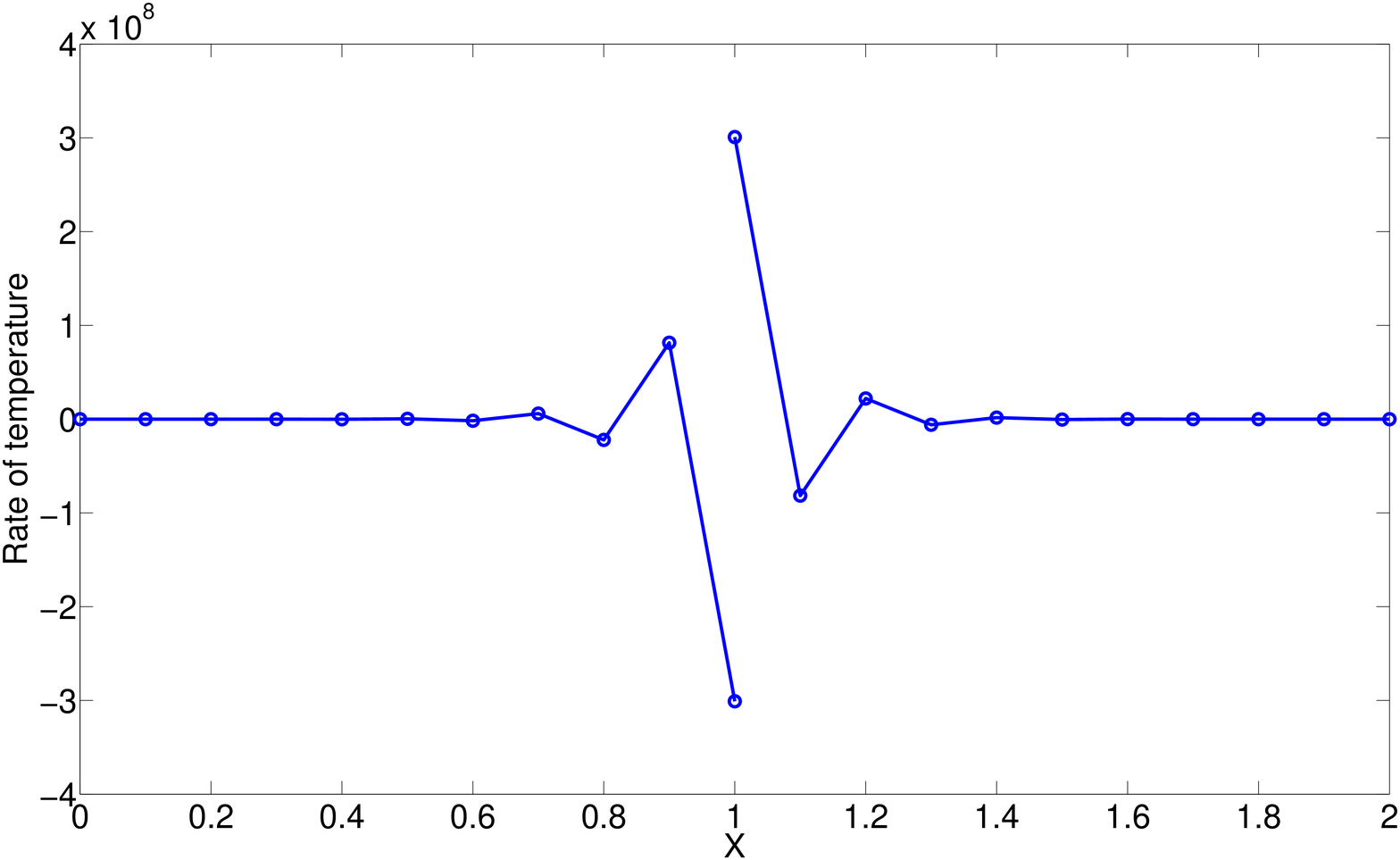}
  \caption{$\boldsymbol{d}$-continuity method: 1D test problem. The trapezoidal parameter is taken as $\gamma = 1/4$, 
    and the time step is taken as $\Delta t = 10^{-5} \; \mathrm{s}$. In this figure, the rate of temperature 
    is plotted against the spatial coordinate at $t = 0.001 \; \mathrm{s}$. As predicted by the theory, 
    the $\boldsymbol{d}$-continuity with $\gamma = 1/4$ is not stable. 
    \label{Fig:DD_test_prob_1_d_continuity_gamma_dot25}}
\end{figure}

\begin{figure}
  \centering
  \subfigure{
    \includegraphics[scale=0.3]{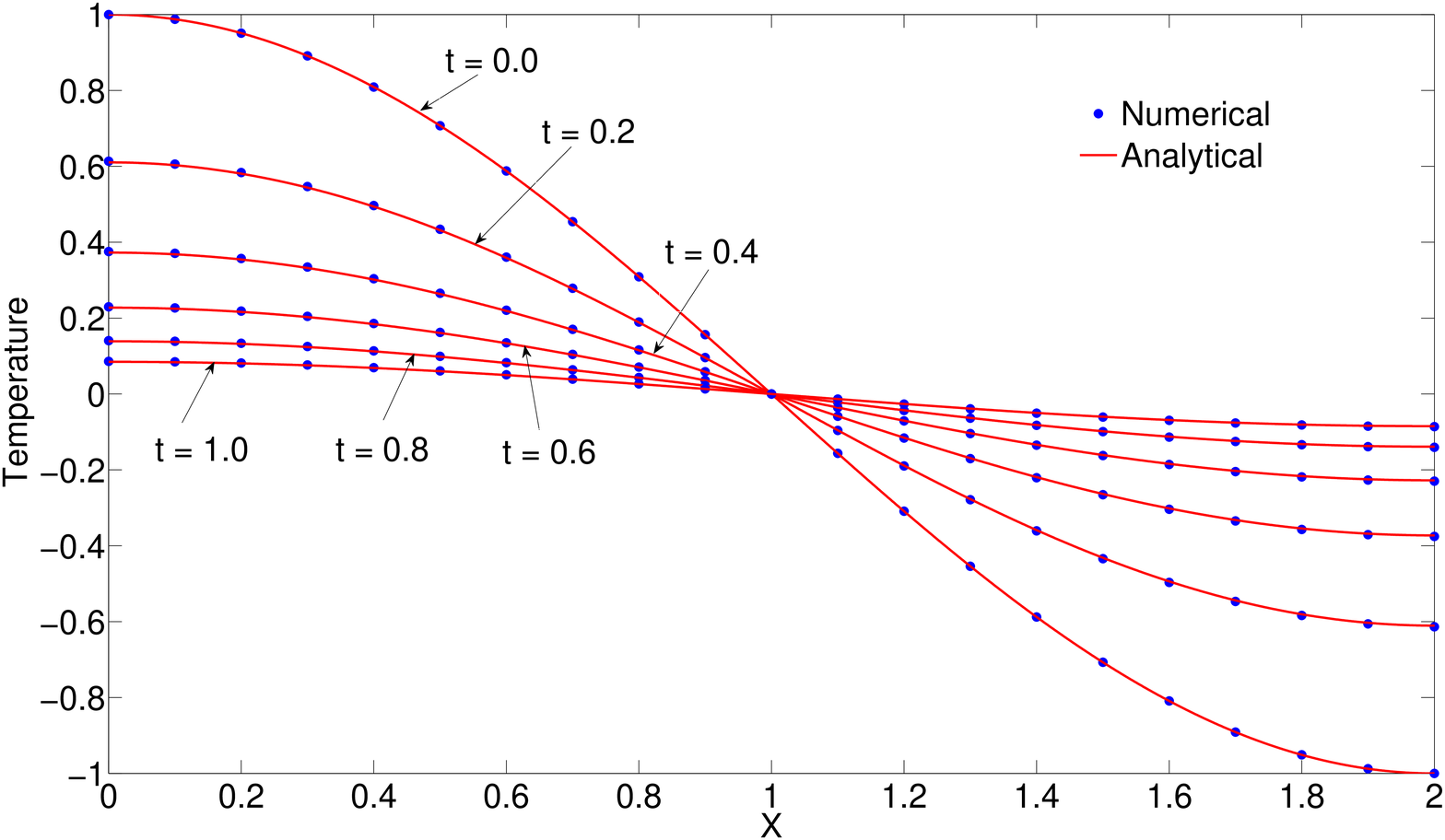}}
  \subfigure{
    \includegraphics[scale=0.3]{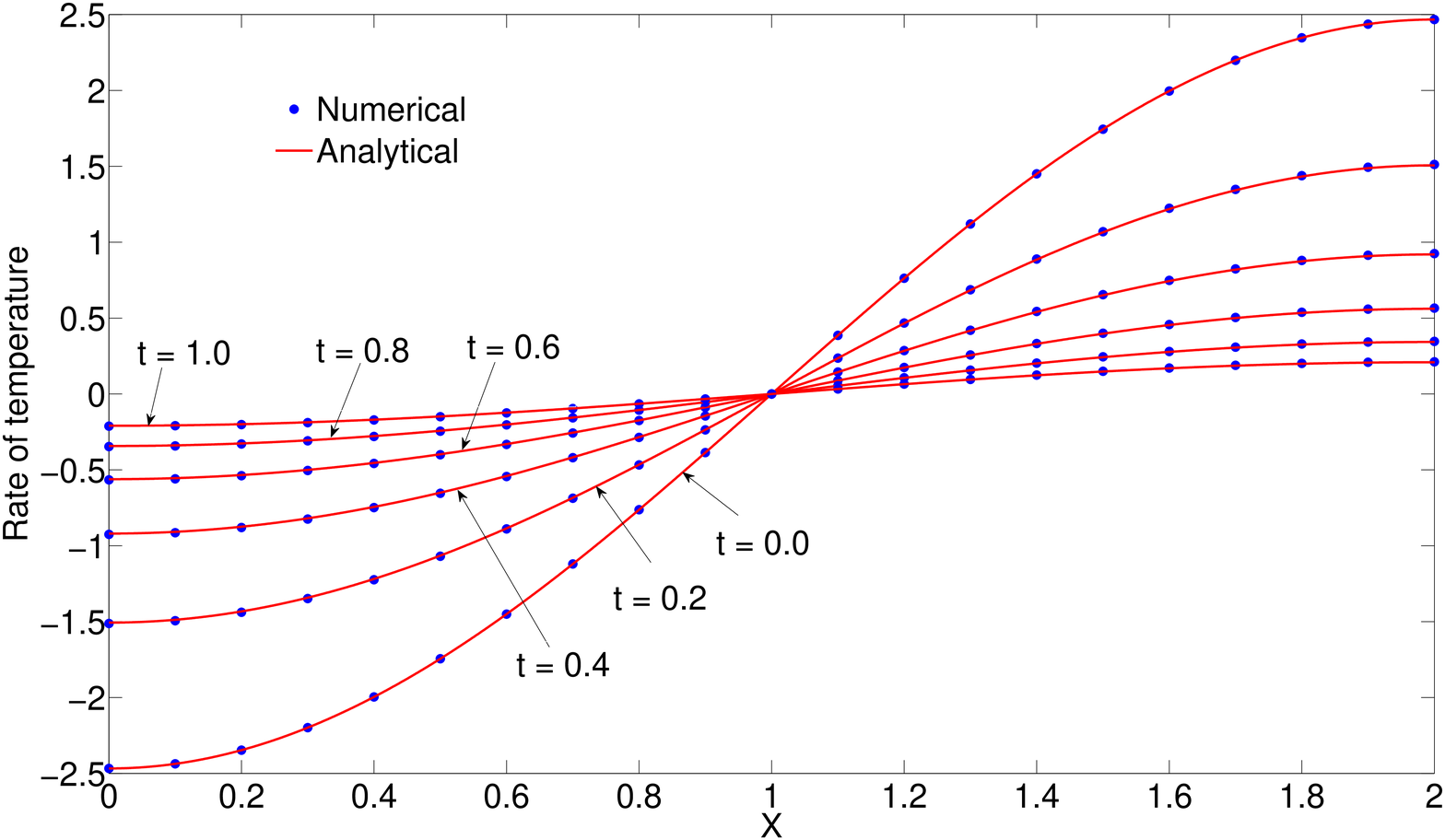}}
  \caption{$\boldsymbol{d}$-continuity method: 1D test problem. The trapezoidal parameter is taken as $\gamma = 3/4$, 
    and the time step is taken as $\Delta t = 0.001 \; \mathrm{s}$. The temperature (top) and rate of temperature 
    (bottom) are plotted against the spatial coordinate at various time levels. 
    \label{Fig:DD_test_prob_1_d_continuity_gamma_dot75}}
\end{figure}

\begin{figure}
  \centering
  \subfigure{
    \includegraphics[scale=0.3]{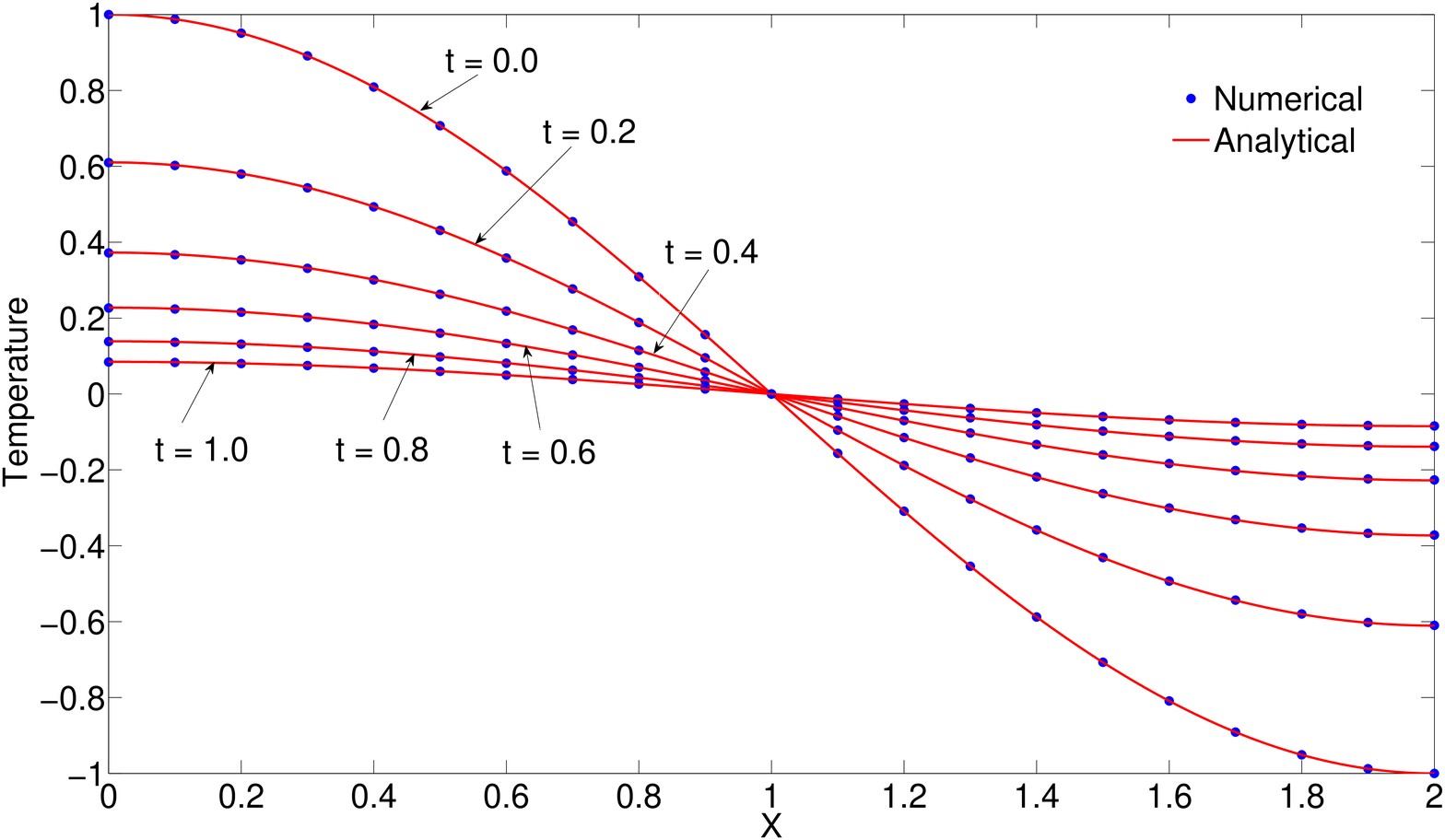}}
  \subfigure{
    \includegraphics[scale=0.3]{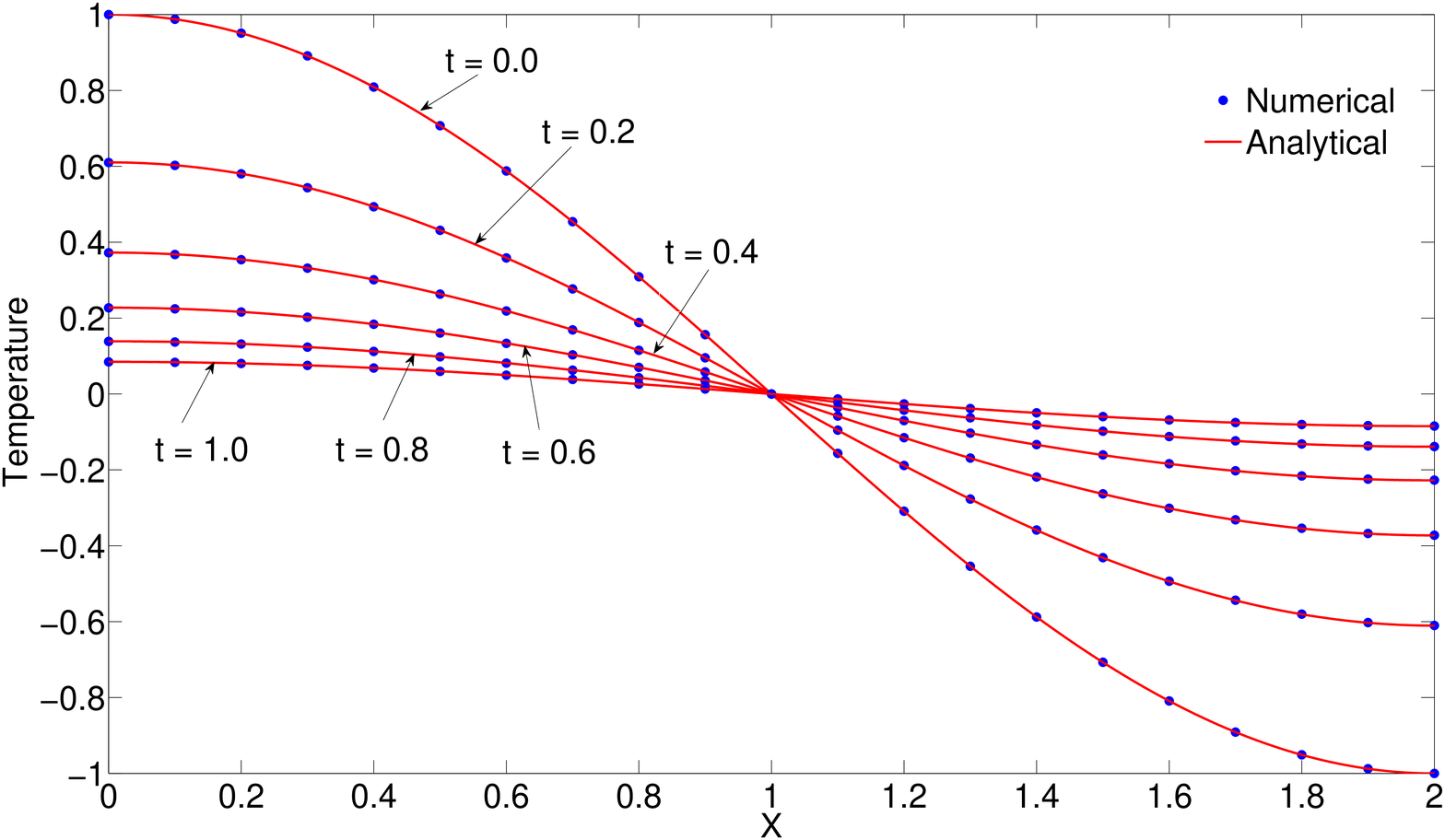}}
  \caption{Modified $\boldsymbol{d}$-continuity method: 1D test problem. The trapezoidal parameter is taken 
    as $\gamma = 1/4$ (top) and $\gamma = 3/4$ (bottom), and the time step is taken as $\Delta t = 0.001 \; 
    \mathrm{s}$. The temperature is plotted against the spatial coordinate at various time levels. Under 
    the modified $\boldsymbol{d}$-continuity method, both the Baumgarte parameters $\gamma = 1/4$ and 
    $\gamma = 3/4$, which is also observed in the numerical simulation of 1D test problem. 
    \label{Fig:DD_test_prob_1_modified_d_continuity}}
\end{figure}

\begin{figure}
  \centering
  \subfigure{
    \includegraphics[scale=0.45]{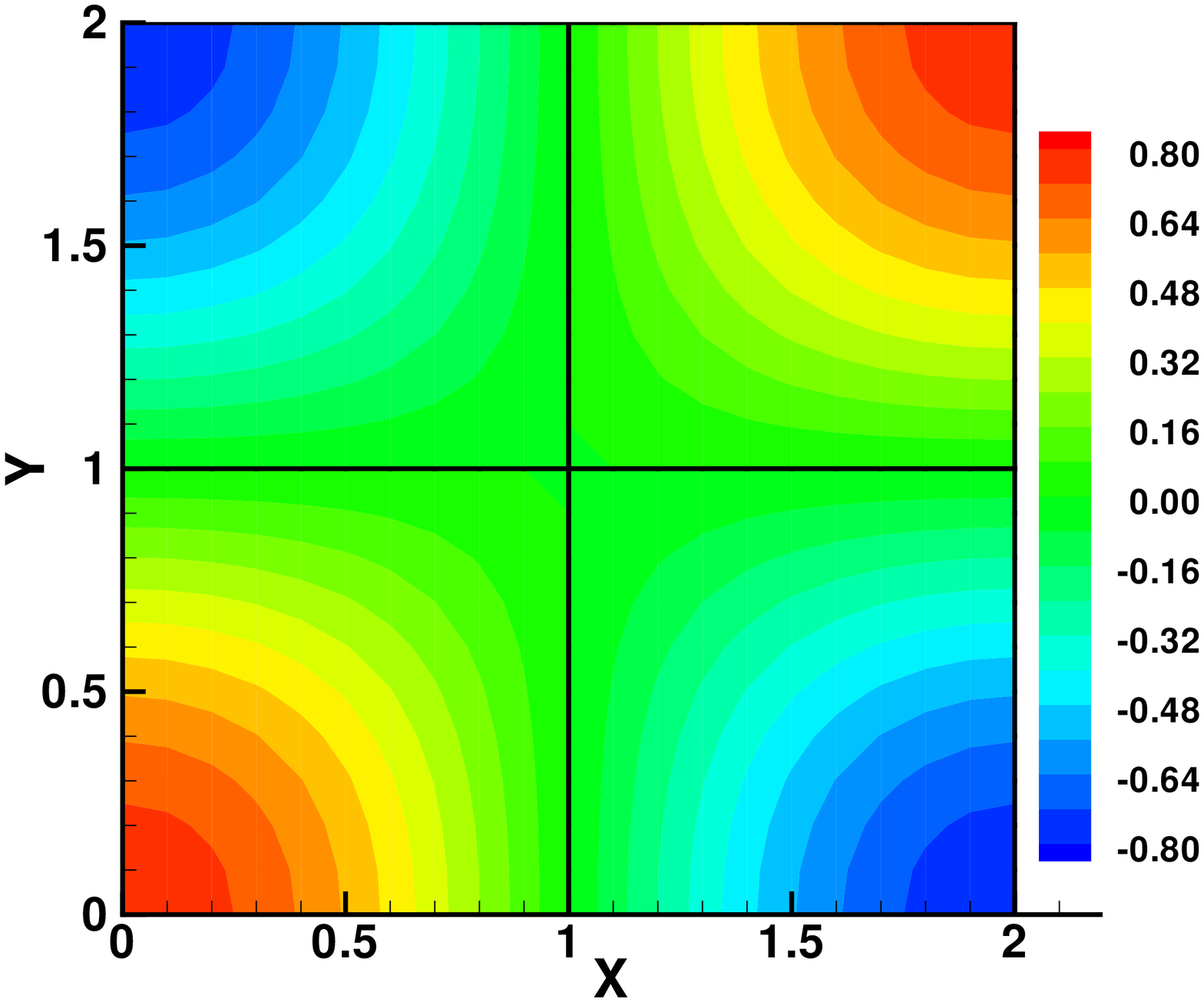}}
  \subfigure{
    \includegraphics[scale=0.45]{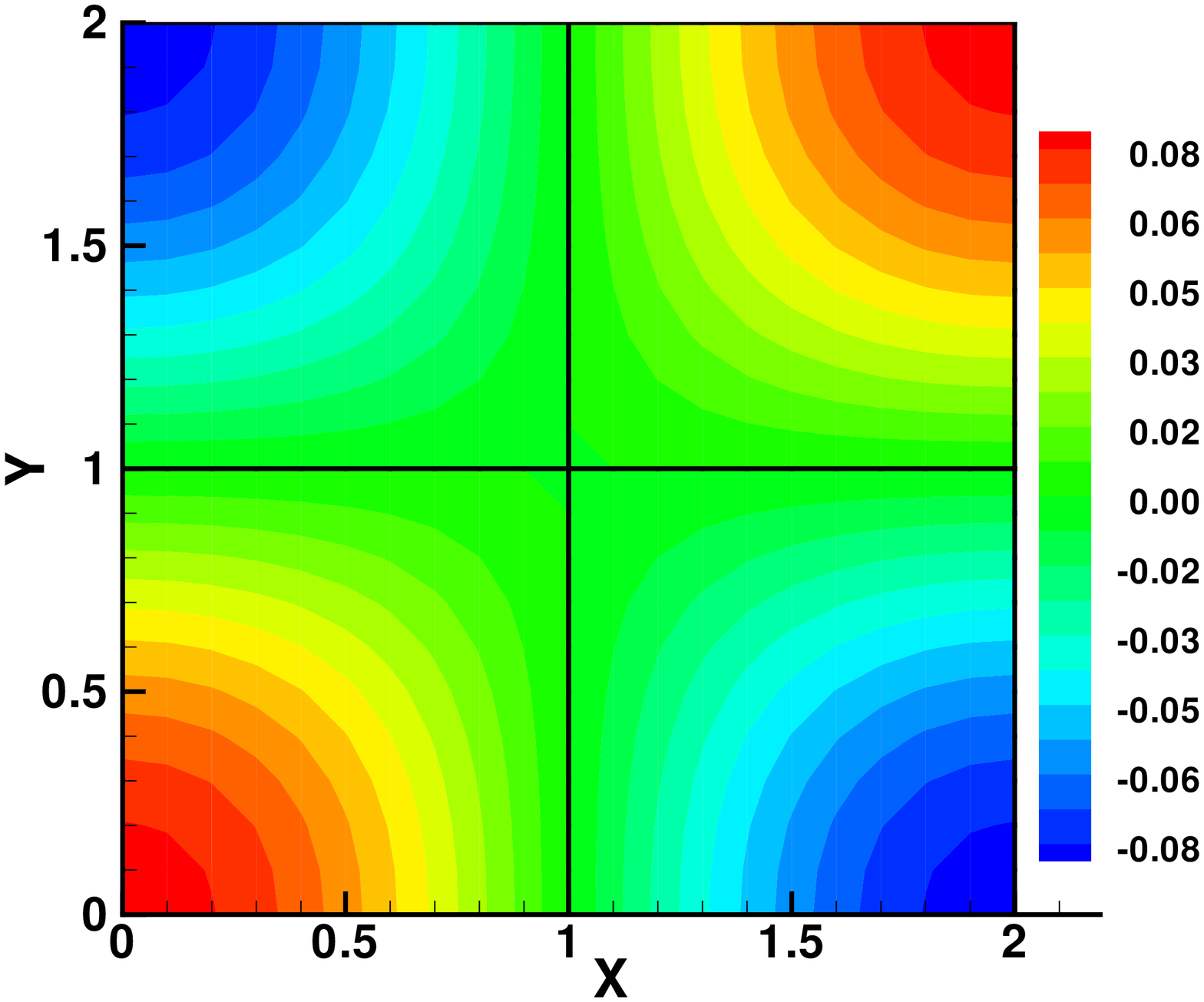}}
  \caption{$\boldsymbol{d}$-continuity: 2D test problem. The trapezoidal parameter 
    is taken as $\gamma = 3/4$, and the time step is taken as $\Delta t = 0.001 \; 
    \mathrm{s}$. The temperature profiles are shown for $t = 0.05 \; \mathrm{s}$ (top) 
    and $t = 0.5 \; \mathrm{s}$ (bottom). \label{Fig:DD_test_prob_2_d_continuity}}
\end{figure}

\begin{figure}
  \centering
  \subfigure{
    \includegraphics[scale=0.45]{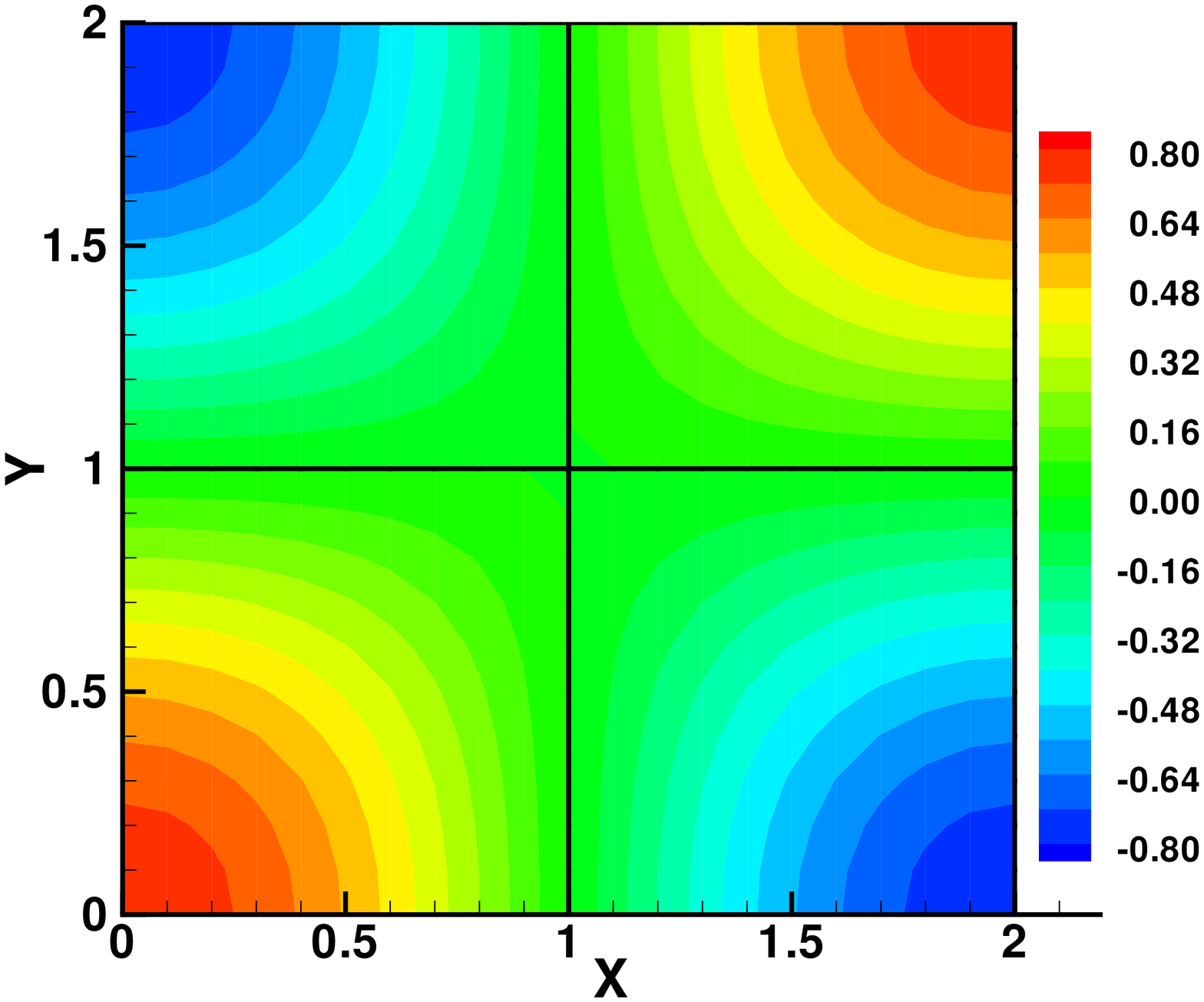}}
  \subfigure{
    \includegraphics[scale=0.45]{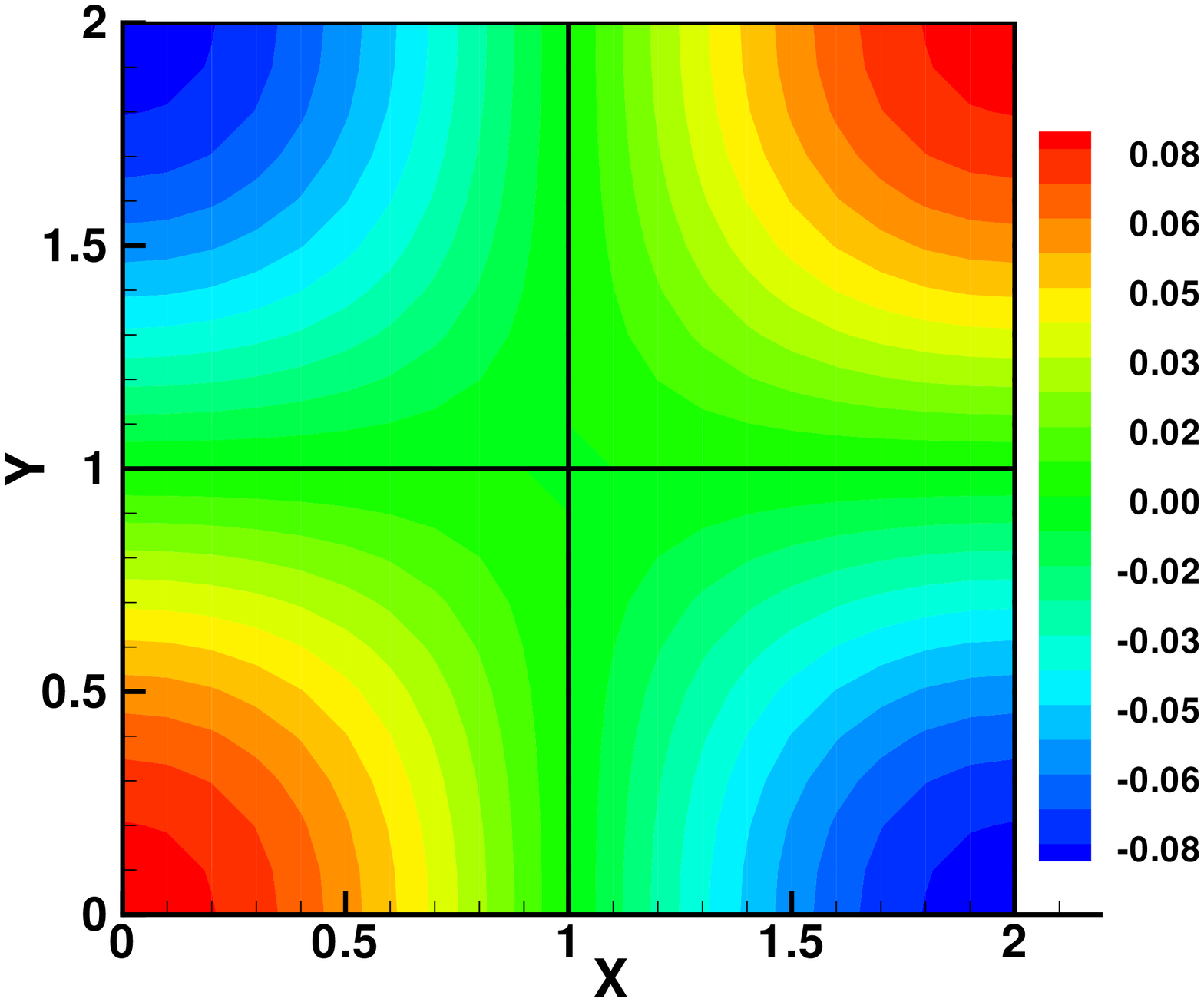}}
  \caption{Modified $\boldsymbol{d}$-continuity: 2D test problem. The trapezoidal parameter 
    is taken as $\gamma = 3/4$, and the time step is taken as $\Delta t = 0.001 \; 
    \mathrm{s}$. The temperature profiles are shown for $t = 0.05 \; \mathrm{s}$ (top) 
    and $t = 0.5 \; \mathrm{s}$ (bottom). \label{Fig:DD_test_prob_2_mod_d_continuity}}
\end{figure}

\begin{figure}
  \centering
  \subfigure{
    \centering
    \includegraphics[scale=0.3]{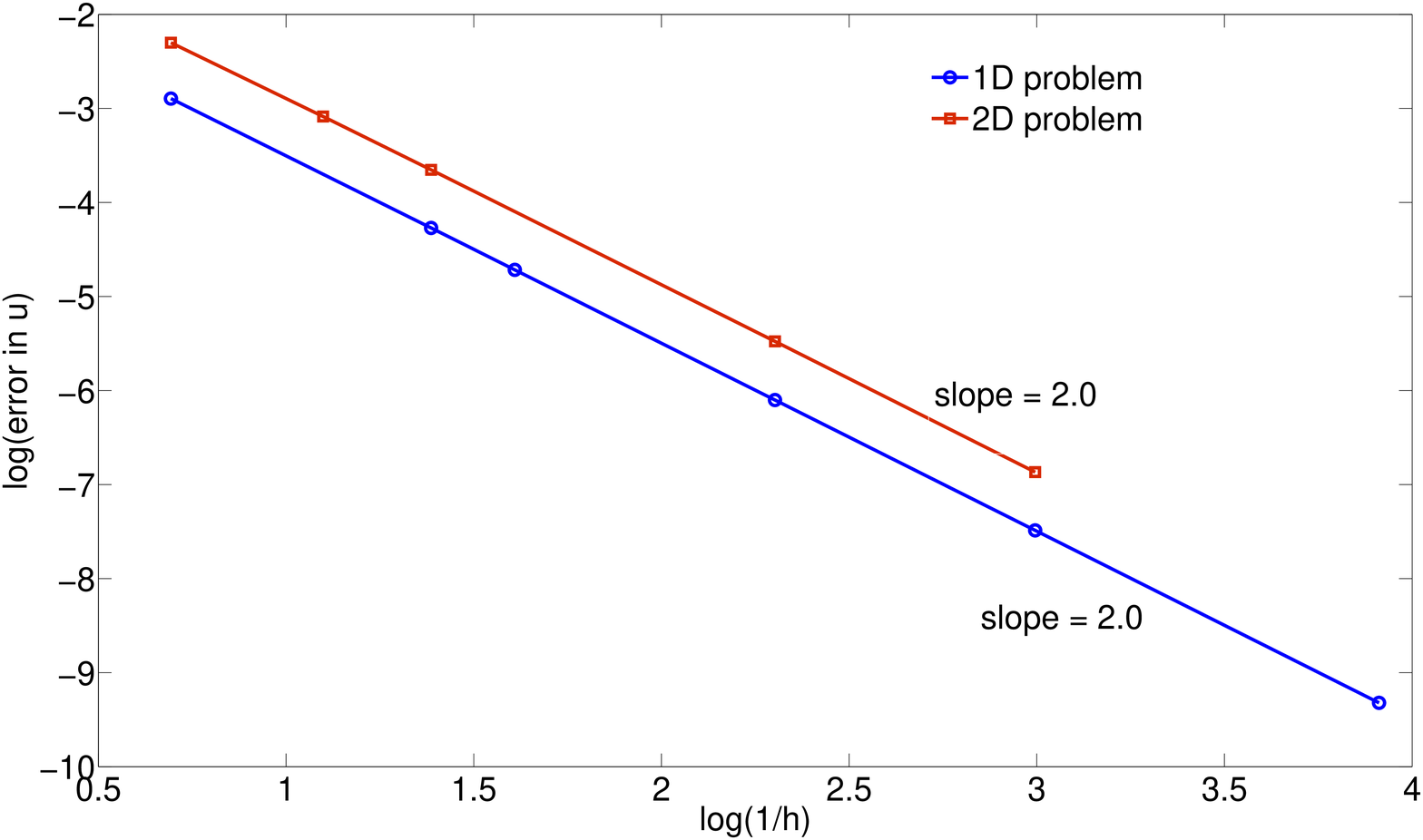}}
  \centering
  \subfigure{
    \centering
    \includegraphics[scale=0.3]{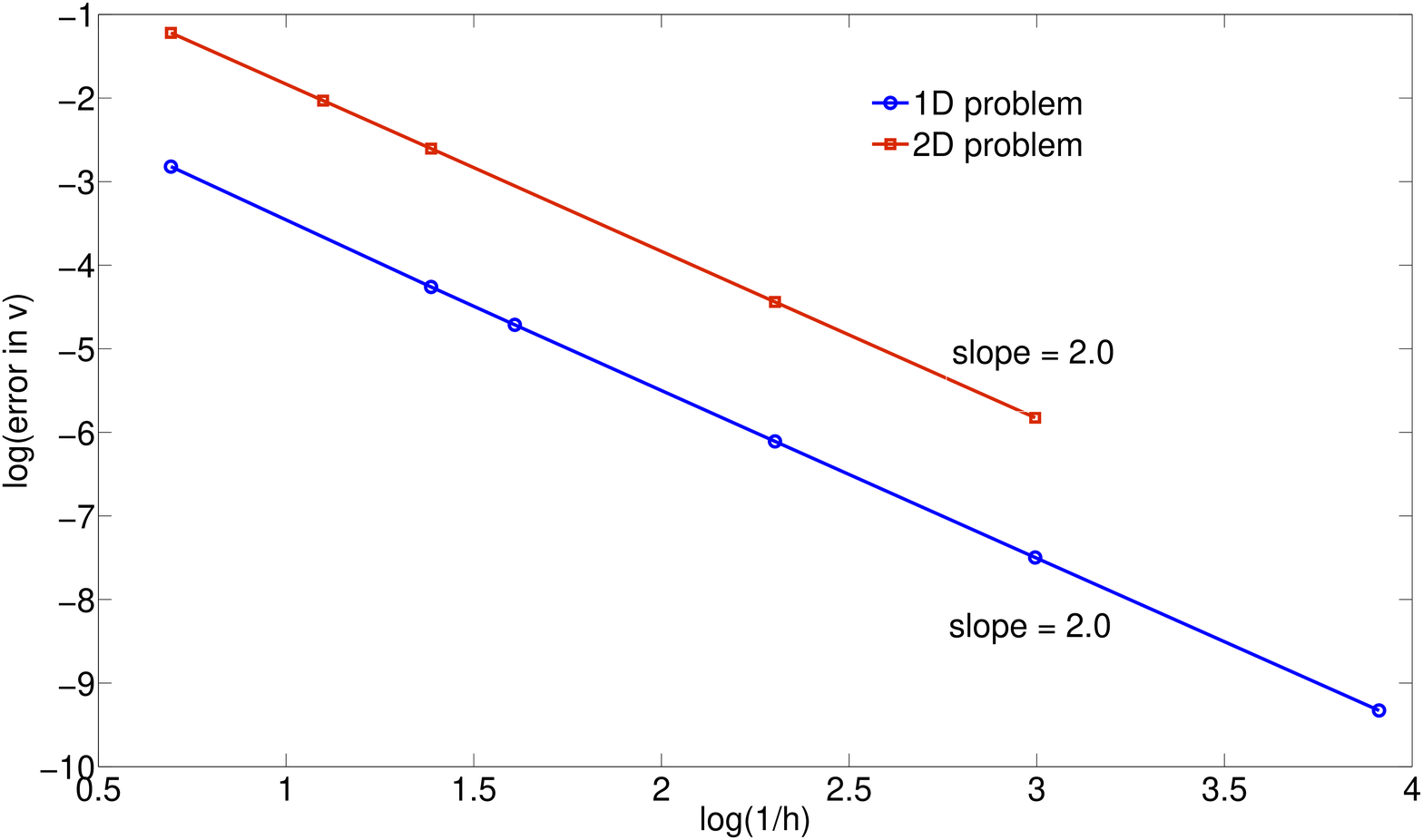}}
  \caption{$\boldsymbol{d}$-continuity method: Spatial numerical convergence analysis 
    on 1D and 2D test problems for temperature (top) and rate of temperature (bottom) 
    using $\gamma = 3/4$. The analysis is performed at the time level $t = 0.01 \; 
    \mathrm{s}$. Standard linear and bilinear elements are used for 1D and 2D problems, 
    respectively. \label{Fig:DD_u_v_spatial_convergence}}
\end{figure}

\begin{figure}
  \centering
  \subfigure{
    \centering
    \includegraphics[scale=0.3]{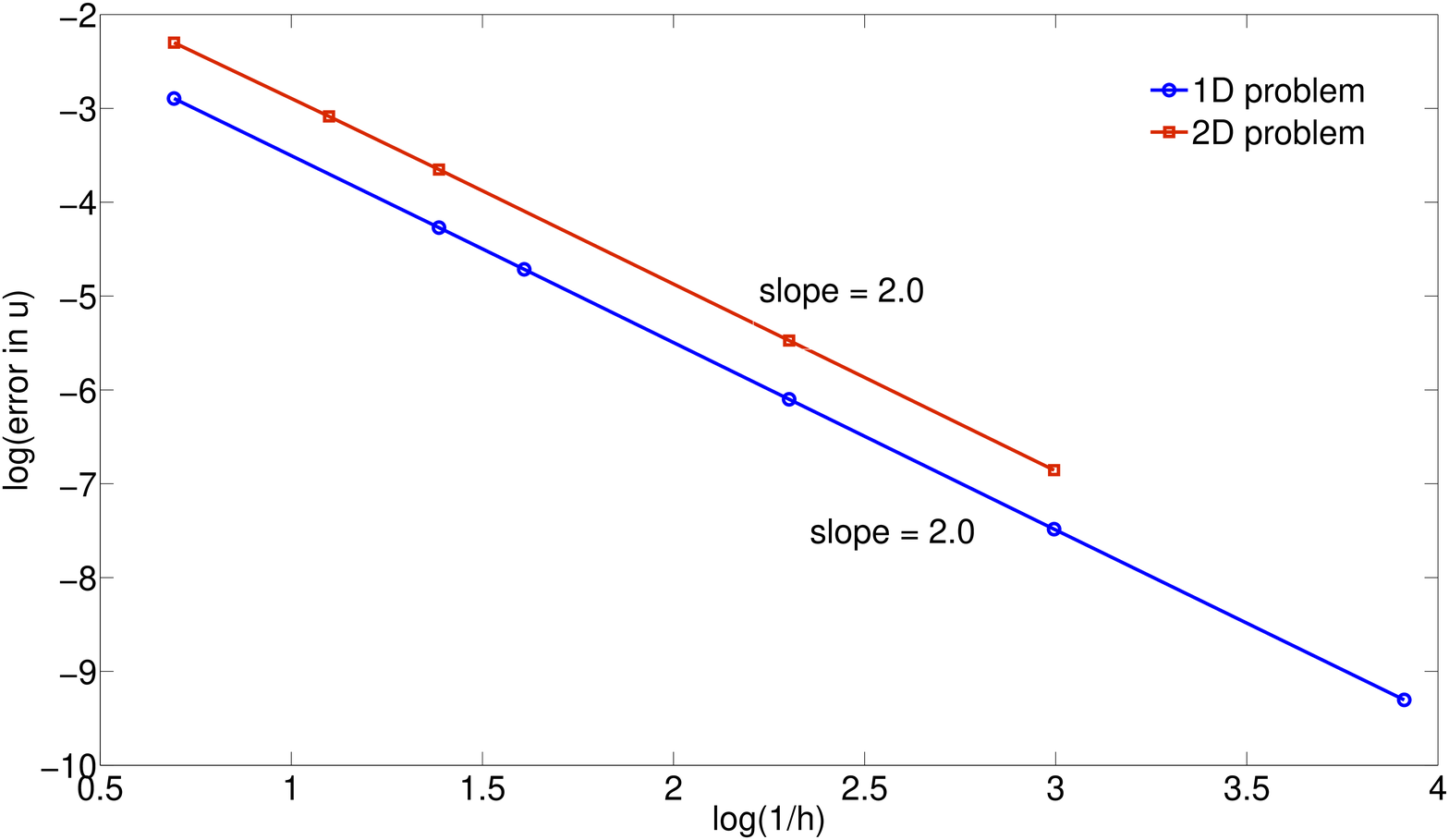}}
  \centering
  \subfigure{
    \centering
    \includegraphics[scale=0.3]{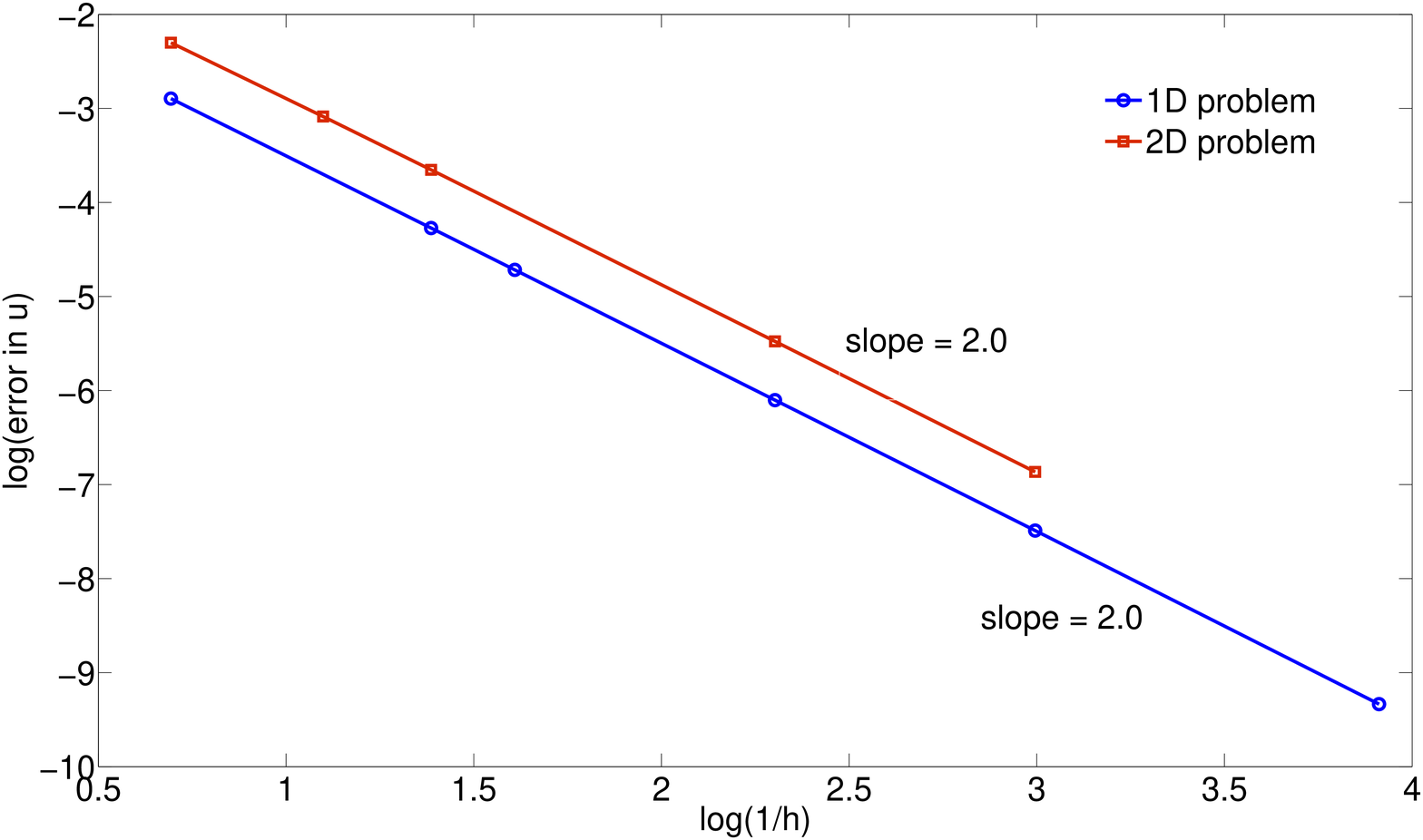}}
  \caption{Modified $\boldsymbol{d}$-continuity method: Spatial numerical convergence analysis 
    on 1D and 2D test problems for temperature using $\gamma = 1/4$ (top) and $\gamma = 3/4$ 
    (bottom). The analysis is performed at the time level $t = 0.01 \; \mathrm{s}$. Standard 
    linear and bilinear elements are used for 1D and 2D problems, respectively. 
    \label{Fig:DD_modified_d_continuity_spatial_convergence}}
\end{figure}


\begin{figure}
\centering
\subfigure{
\centering
\includegraphics[scale=0.25]{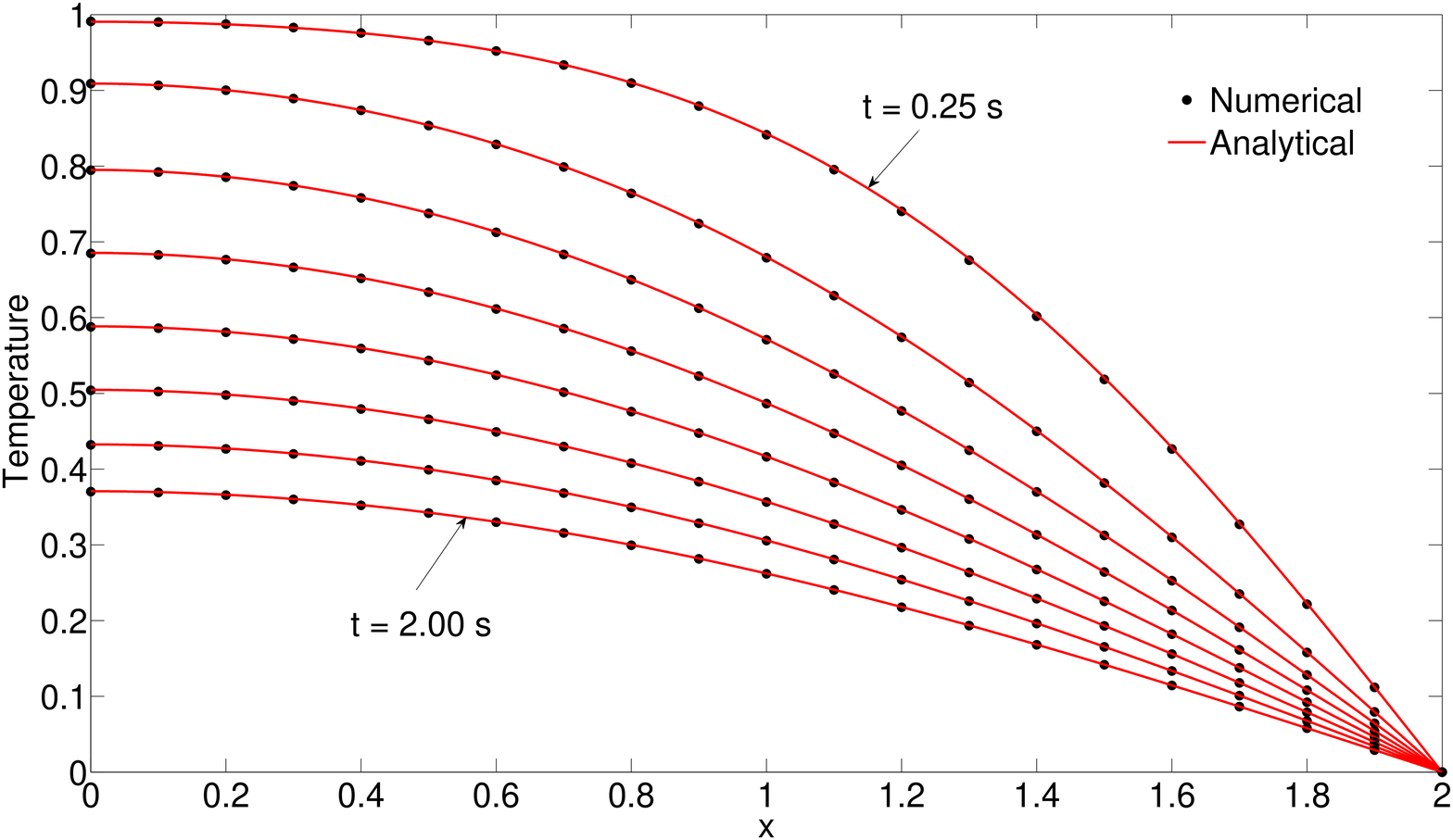}}
\subfigure{
\centering
\includegraphics[scale=0.25]{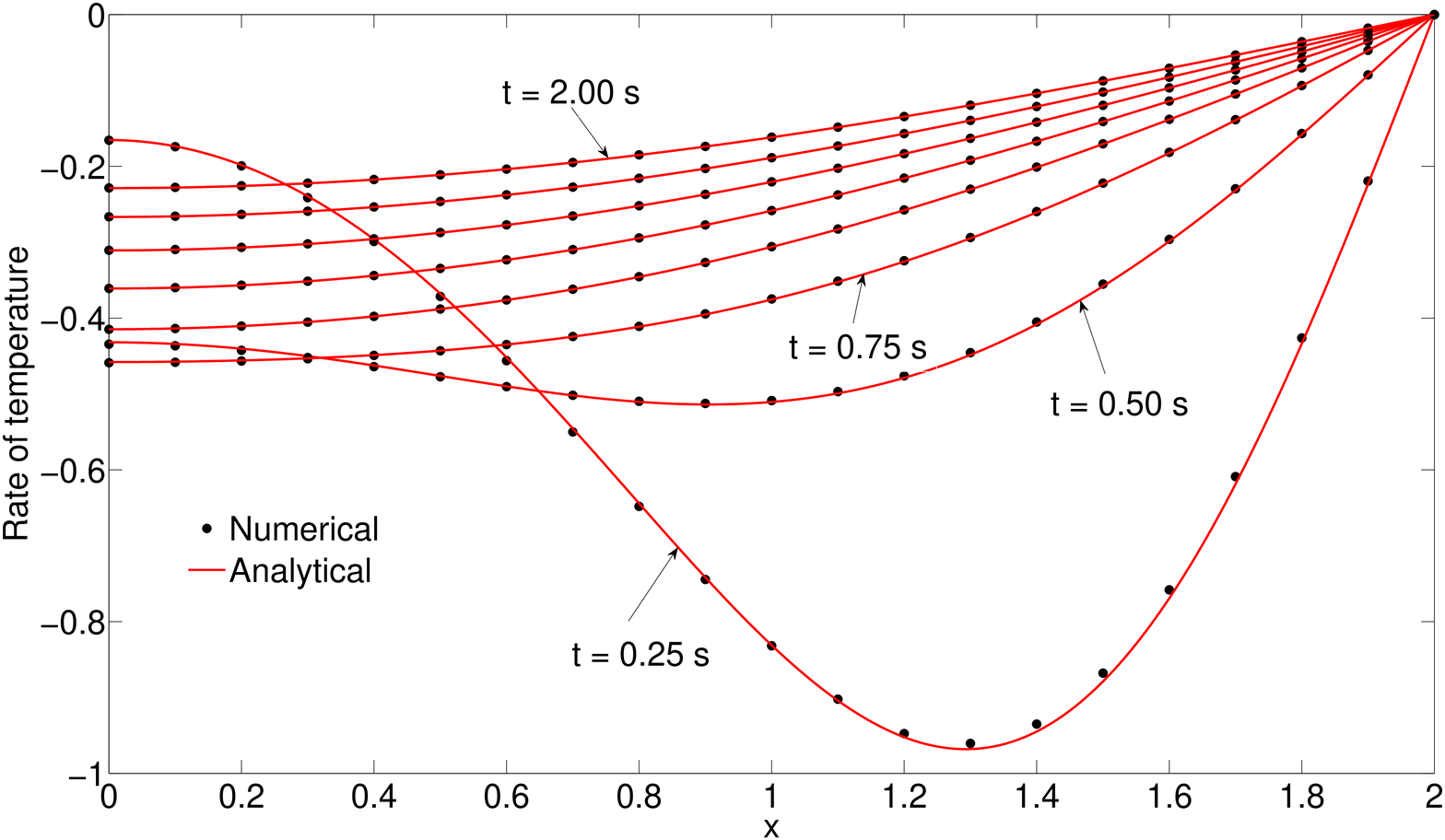}}
\caption{Baumgarte stabilized DD method. The Baumgarte parameter is taken as $\alpha = 1$, 
trapezoidal parameter as $\gamma = 0.1$, and the time step as $0.001 \; \mathrm{s}$. Each 
subdomain is divided into $10$ equal elements. The obtained numerical solutions for temperature 
(top) and rate of temperature (bottom) are plotted against $x$ at time levels $0.25 \; \mathrm{s}$ 
to $2.00 \; \mathrm{s}$ with time increments of $0.25 \; \mathrm{s}$.
\label{Fig:DD_Baumgarte_temp_stable_alpha}}
\end{figure}

\begin{figure}
\centering
  \subfigure{
\centering
    \includegraphics[scale=0.25]{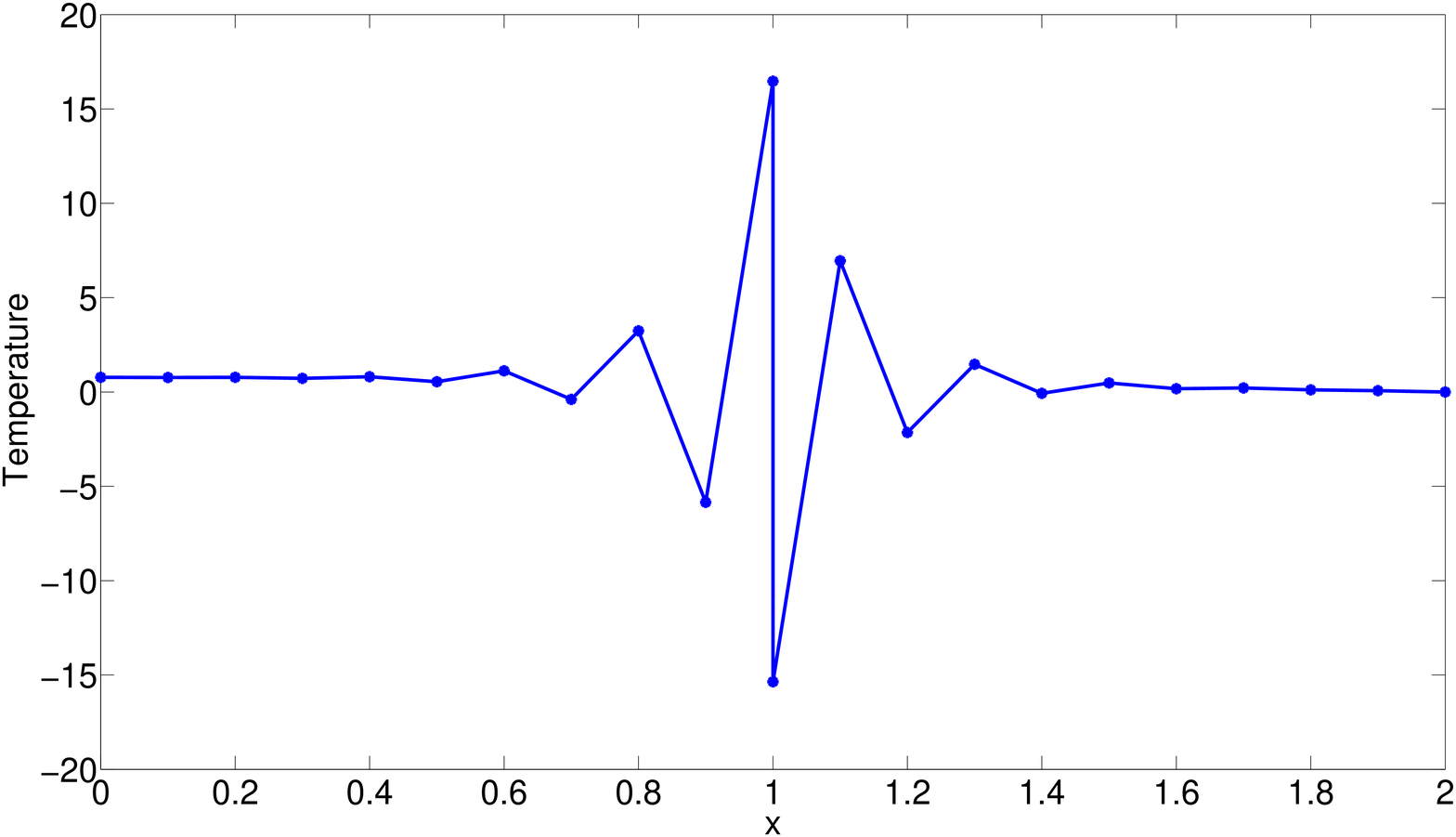}}
  \subfigure{
\centering
    \includegraphics[scale=0.25]{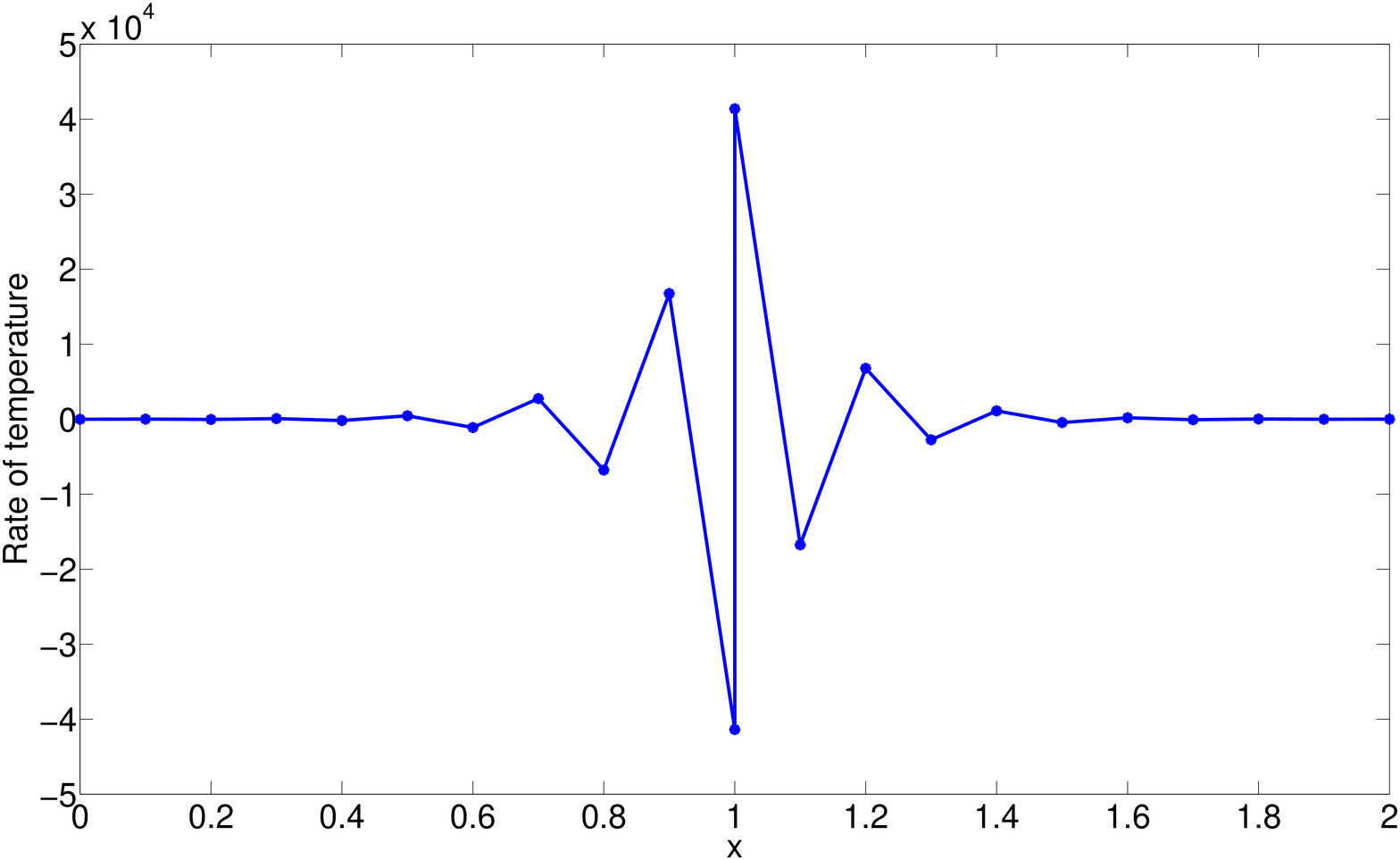}}
  \caption{Baumgarte stabilized DD method. The Baumgarte parameter is taken as $\alpha = 2.6$, 
    trapezoidal parameter as $\gamma = 0.1$, and the time step as $0.001 \; \mathrm{s}$. Each subdomain 
    is divided into $10$ equal elements. The obtained numerical solutions at $t = 0.8 \; \mathrm{s}$ 
    for temperature (top) and rate of temperature (bottom) are plotted against $x$. 
    \label{Fig:DD_Baumgarte_temp_unstable_alpha}}
\end{figure}

\begin{figure}
  \centering
  \subfigure{
    \centering
    \includegraphics[scale=0.25]{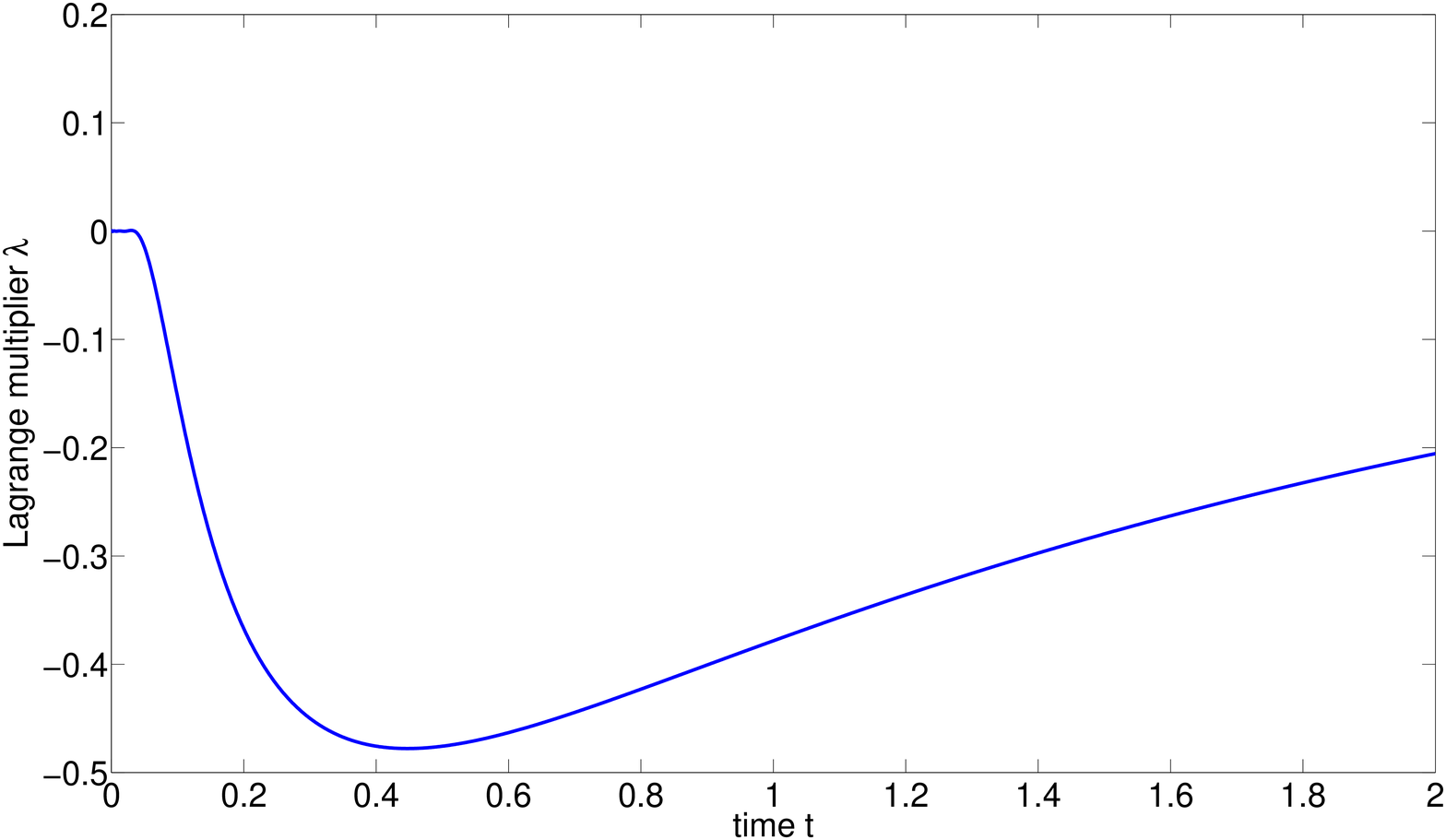}}
  \subfigure{
    \centering
    \includegraphics[scale=0.25]{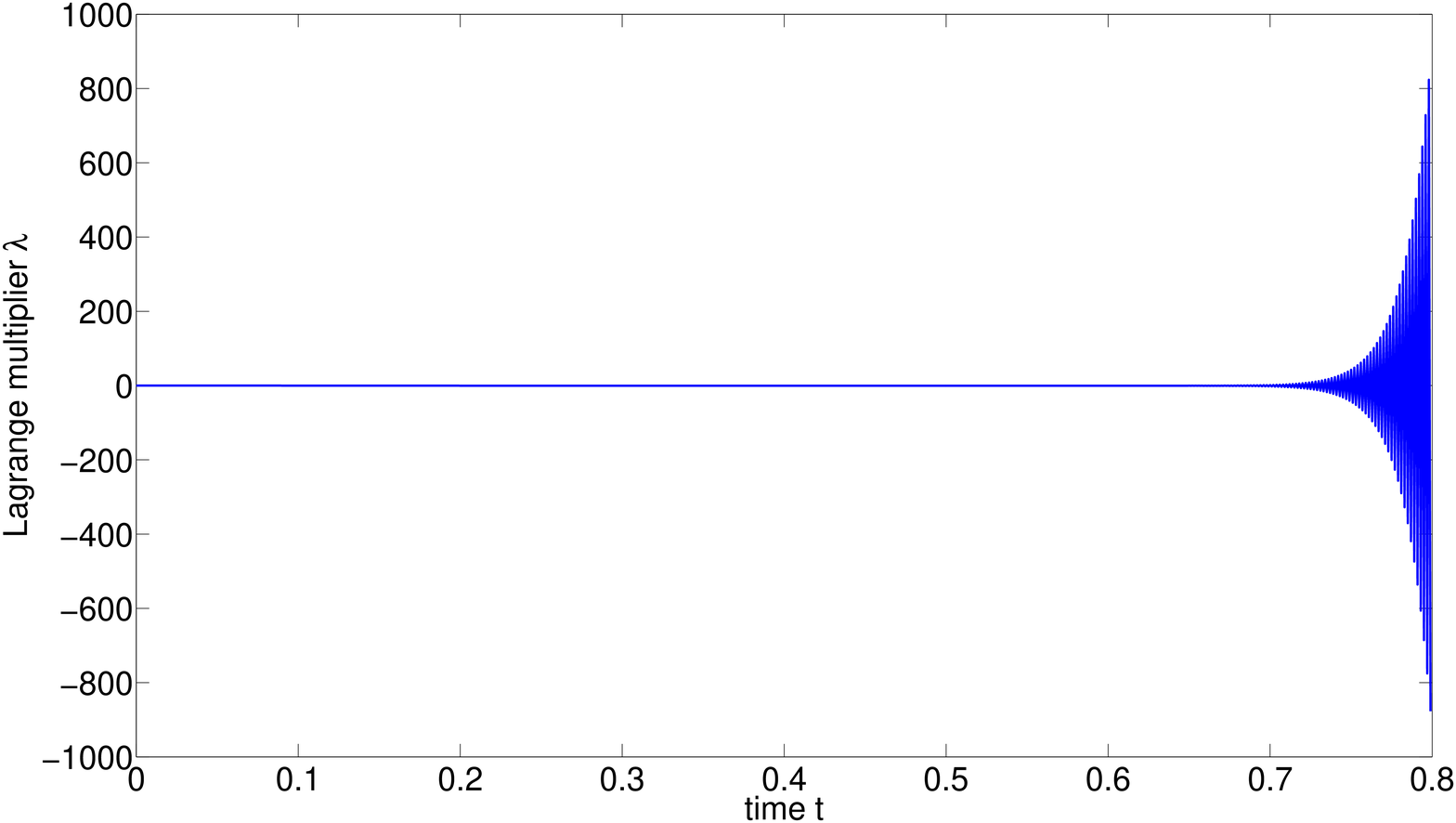}}
  \caption{Baumgarte stabilized DD method. The trapezoidal parameter as $\gamma = 0.1$, the time step as 
    $0.001 \; \mathrm{s}$, and each subdomain is divided into $10$ equal elements. In this figure we have 
    plotted the interface Lagrange multiplier against time for the Baumgarte parameters $\alpha = 1$ (top) 
    and $\alpha = 2.6$ (bottom). \label{Fig:DD_Baumgarte_lambda_unstable_alpha}}
\end{figure}

\begin{figure}
  \centering
  \subfigure{
    \includegraphics[scale=0.25]{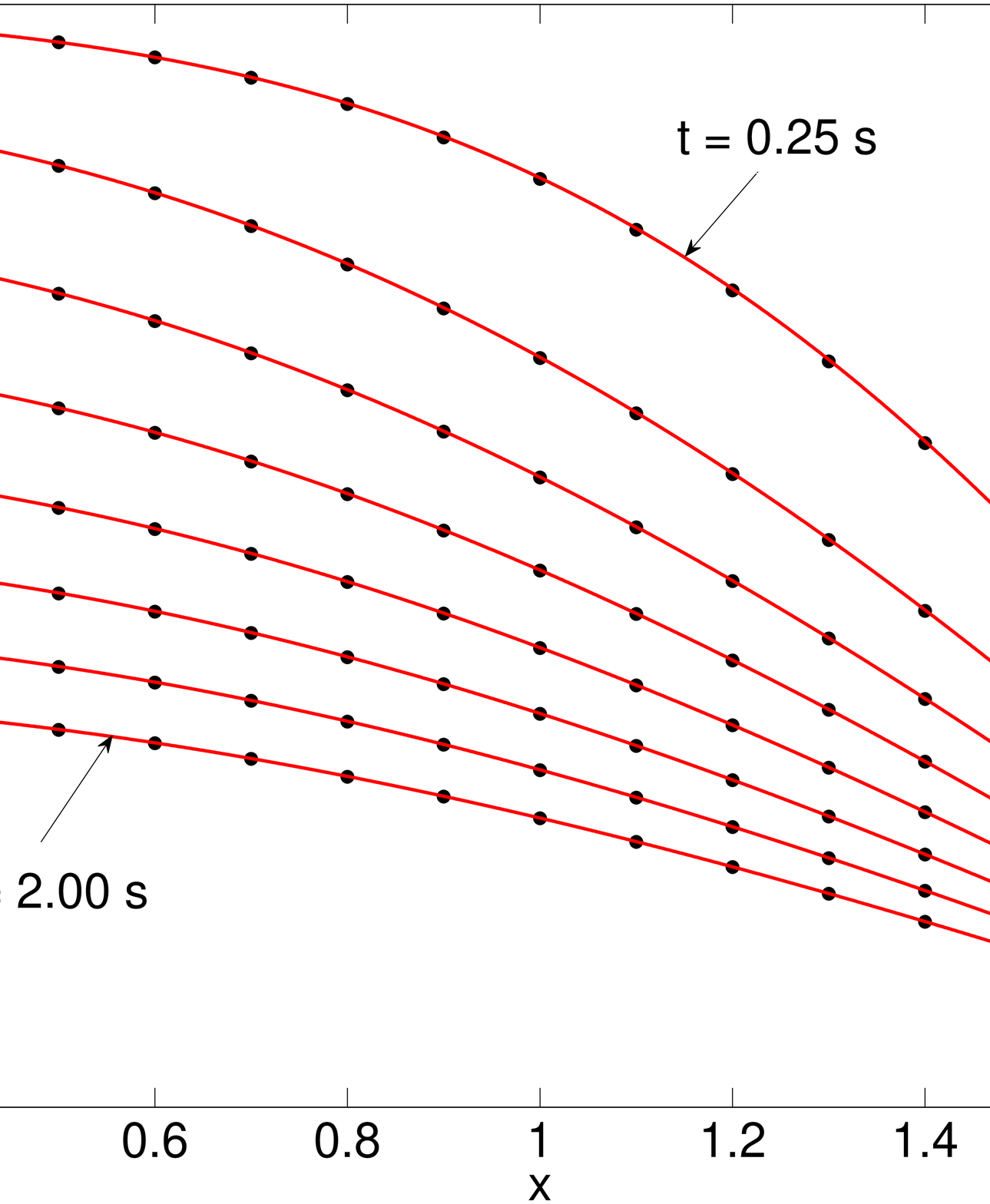}}
  \subfigure{
    \includegraphics[scale=0.25]{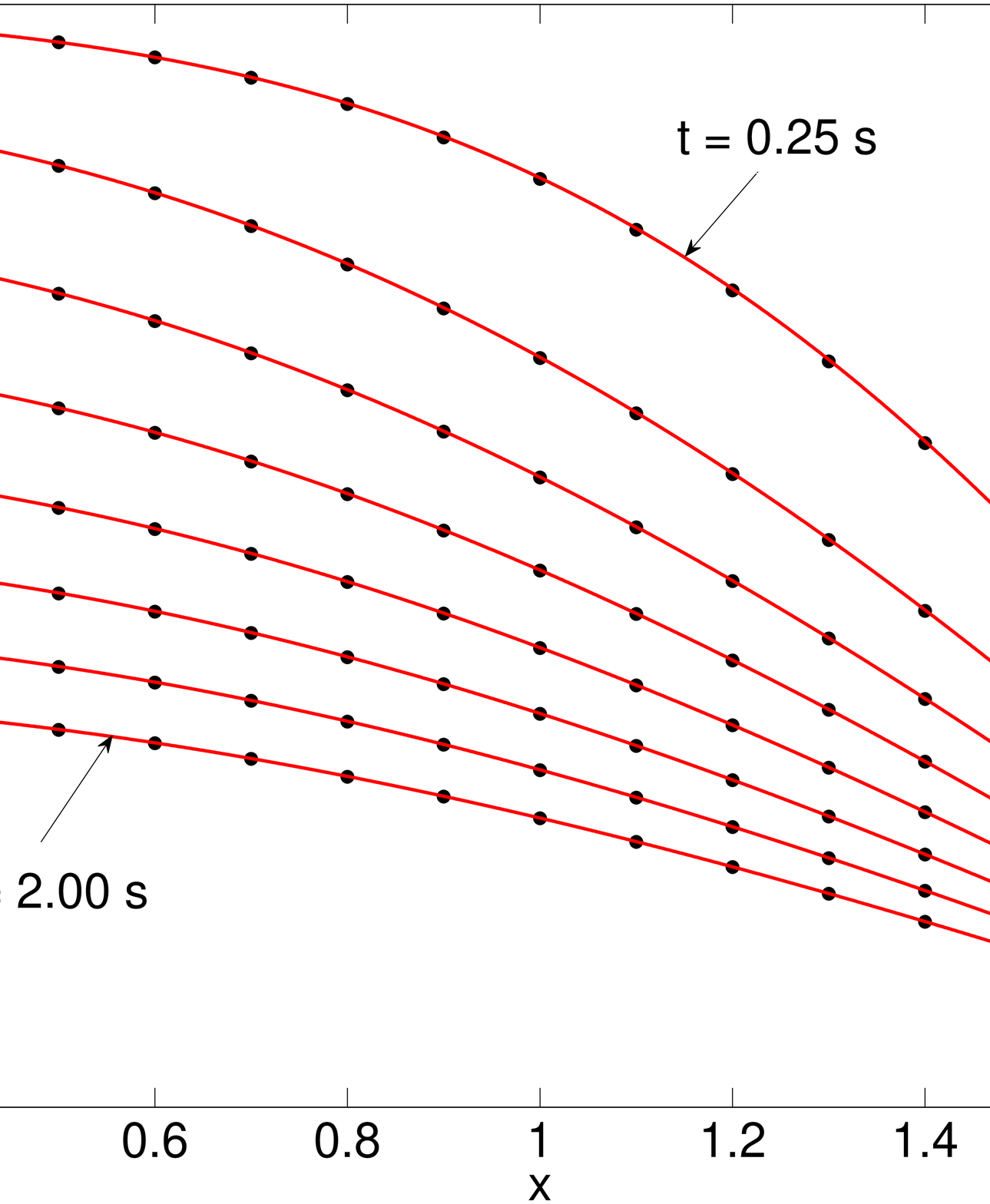}}
  \caption{Baumgarte stabilized DD method. The trapezoidal parameter is taken as $\gamma = 1/2$, 
    and the time step as $0.001 \; \mathrm{s}$. Each subdomain is divided into $10$ equal elements. 
    The temperature profiles for $\alpha = 2.6$ (top) and $\alpha = 10$ (bottom) are plotted against 
    $x$ at time levels $0.25 \; \mathrm{s}$ to $2.00 \; \mathrm{s}$ with time increments of $0.25 \; 
    \mathrm{s}$. \label{Fig:DD_Baumgarte_temp_gamma_dot5}}
\end{figure}
\end{document}